\documentclass[10pt]{amsart}
\usepackage{macros}
\begin{document}
\title{Generalized blow-up of corners and fiber products}
\author{Chris Kottke}
\address{Department of Mathematics\\Brown University}
\email{ckottke@math.brown.edu}
\author{Richard B. Melrose}
\address{Department of Mathematics\\Massachusetts Institute of Technology}
\email{rbm@math.mit.edu}
\thanks{The second author was supported in part by NSF grant DMS-1005944.}
\date{\today}
\begin{abstract}
Real blow-up, including inhomogeneous versions, of boundary faces of a
manifold (with corners) is an important tool for resolving singularities,
degeneracies and competing notions of homogeneity. These constructions are
shown to be particular cases of \emph{generalized boundary blow-up} in which
a new manifold and blow-down map are constructed from, and conversely
determine, combinatorial data at the boundary faces in the form of a refinement of 
the \emph{basic monoidal complex} of the manifold. This data specifies
which notion of homogeneity is realized at each of the boundary
hypersurfaces in the blown-up space.

As an application of this theory, the existence of fiber products is
examined for the natural smooth maps in this context, the
b-maps. Transversality of the b-differentials is shown to ensure that the
set-theoretic fiber product of two maps is a \emph{binomial
  variety}. Properties of these (extrinsically defined) spaces, which
generalize manifolds but have mild singularities at the boundary, are
investigated and a condition on the basic monoidal complex is found under
which the variety has a smooth structure. Applied to b-maps this additional
condition with transversality leads to a universal fiber product in the
context of manifolds with corners. Under the transversality condition alone
the fiber product is resolvable to a smooth manifold by generalized blow-up
and then has a weaker form of the universal mapping property requiring
blow-up of the domain.
\end{abstract}
\maketitle

\section*{Introduction}\label{S:introduction}
{\renewcommand{\theequation}{I.\arabic{equation}} 

Real blow-up of a submanifold introduces a new manifold (always here
implicitly meaning `with corners') in which the submanifold in question is
replaced by one or more boundary hypersurfaces in a prescribed manner. Here
we introduce the notion of the \emph{generalized boundary blow-up} of a
manifold with corners, $Y,$ which is a new manifold $X$ along with a smooth
and proper `blow-down' map $\beta :X\to Y$.  The latter is by definition a
b-map, restricting to a diffeomorphism on the interiors, with bijective
b-differential. It includes standard radial blow-up of boundary faces,
iterated boundary blow-up, and inhomogeneous blow-up as special cases.

It is shown here that such blow-down maps are, up to diffeomorphism,
characterized in an essentially algebraic manner by a `monoidal complex', a
consistent choice of certain combinatorial data at each face of $Y$.  One
application of the constructive part of this result is to fiber products of
maps. Working in the category of smooth manifolds and b-maps, it is shown
that under the natural condition of b-transversality the fiber product of
two maps can be decomposed and smoothed by generalized blow-up and this resolved fiber
product has a weakened form of the universal factoring property for fiber
products in which generalized blow-up of the domain may be required.

Although this paper is concerned with geometric and algebraic questions, one of
the contexts in which blow-up appears is analytic. Most analytic problems on a
manifold with corners do not have solutions `within the smooth category'. For
instance, harmonic forms with absolute or relative boundary conditions are
not typically smooth up to corners. Rather they lie in conormal spaces which
themselves are only `resolved' to polyhomogeneity by blow-up. Here
polyhomogeneity is to be thought of as the natural extension of smoothness, in
which $\CI(Y)$ is extended to include functions which have non-integral
(including possibly negative) powers in their `Taylor series' at the boundary.
The selection of the \emph{correct} resolution, in the sense of blow-up, of
spaces and their (fiber) products, on which the kernels of operators are
defined, is one of the prime motivations for the discussion here.

A manifold (meaning from now on a manifold with corners) $Y$ is a topological
manifold with boundary which is locally diffeomorphic to the model space
$[0,\infty)^k\times \R^{n-k}$ with its sheaf of smooth functions.  The set of
(connected) boundary faces of $Y$ is denoted $\cM(Y)$, with $\cM_l(Y)$ denoting
those faces of codimension $l \in \set{0,\ldots,k}$.  We require that boundary
hypersurfaces of a manifold be embedded, which means that the functions
vanishing on each such hypersurface form a principal ideal $\cI_H \subset
C^\infty(Y)$, $H \in \cM_1(Y).$  A smooth map $f : X \to Y$ is an (interior)
boundary map or `b-map' as defined in \cite{melrosedifferential,
melrose1992calculus} provided it pulls back each of these principal ideals to a
product
\begin{equation}
	f^\ast \cI_H= \prod_{G \in \cM_1(Y)} \cI_G^{\alpha(H,G)}, \quad \alpha(\cdot,\cdot) \in \Zp
	\label{E:intro_b_map}
\end{equation}
of similar ideals in $C^\infty(X)$.

For each face $F \in \cM_l(X)$ the freely generated monoid
\[
	\sigma_F = \bigoplus_{H \in \cM_1(X), F \subseteq H} \N e_H
\]
along with the inclusions $i_{GF} : \sigma_G \hookrightarrow \sigma_F$ for $F \subset G$
constitute what we call the `basic monoidal complex' of $X$:
\[
	\cP_X = \set{(\sigma_F,\ i_{GF}) \;;\; F\subseteq G \in \cM(X)},
\]
and a b-map $f : X \to Y$ defines a morphism of the complexes
\begin{equation}
\begin{gathered}
	f_\natural : \cP_X \to \cP_Y, \\
	\sigma_F \stackrel {f_\natural} \to \sigma_{f_{\#}(F)}.
\end{gathered}
	\label{E:intro_morphism}
\end{equation}
which is fundamental to our discussion. Here $f_\# : \cM(X) \to \cM(Y)$ maps
each boundary face of $X$ to the boundary face of largest codimension in $Y$
which contains its image, and the coefficients of $f_\natural : \sigma_F \to
\sigma_{f_\#(F)}$ are the relevant exponents $\alpha(\cdot,\cdot) \in \Zp$
occuring in  \eqref{E:intro_b_map}.

In general, the monoidal complexes and their morphisms capture only limited
combinatorial information regarding the boundary faces of $X$ and $Y$. However
in certain special cases, such as when $X$ is the blow-up of a boundary face
$F$ in $Y$, this information is enough to completely specify, up to
diffeomorphism, the domain $X = [Y; F]$ and the map $\beta : X \to Y$ in terms
of the range space $Y.$ In such a case the morphism \eqref{E:intro_morphism} 
forms what is called below a `smooth refinement' of $\cP_Y$, and on the analytical
side $\beta$ satisfies additional properties, namely
\begin{gather}
\beta : X \setminus \pa X \to Y \setminus \pa Y\text{ is a diffeomorphism, and}
\label{E:intro_blow_down_diffeo}\\
\bd \beta_\ast : \bT_pX \to \bT_{\beta(p)} Y \ \text{is an isomorphism for
  all }p\in X,
\label{E:intro_blow_down_bdiffl}
\end{gather}
where $\bd \beta_\ast$ is the differential acting on the b-tangent
bundles.

Our first main result consists of two parts: a generalization of this
`boundary blow-up' with respect to arbitrary smooth refinements of $\cP_Y$, and
an analytical characterization of such a maps.

\begin{thma} 
To any smooth refinement $\cR \to \cP_Y$ there corresponds a manifold denoted
$X = [Y; \cR]$, unique up to diffeomorphism, with $\cP_X = \cR$ and a unique b-map $f:
X = [Y;\cR] \to Y$ satisfying \eqref{E:intro_blow_down_diffeo} and
\eqref{E:intro_blow_down_bdiffl}. Conversely, any smooth proper b-map 
satisfying \eqref{E:intro_blow_down_diffeo} and \eqref{E:intro_blow_down_bdiffl}
determines a smooth refinement $\cP_X \to \cP_Y.$
\end{thma}

The manifold $[Y; \cR]$ is referred to as the `generalized blow-up' of $Y$ by
the refinement $\cR$, and the lifting problem for b-maps under generalized
blow-ups of the domain and/or range is resolved by examining the corresponding
lifting of the monoidal complex morphisms.

Our second main result is an application of this theory to fiber products.
Recall that a (smooth) fiber product of two maps $f_i : X_i \to Y$, $i = 1,2$
does not generally exist even in the category of smooth manifolds {\em without}
boundary. There is however a sufficient condition for existence, namely that
$f_1$ and $f_2$ be {\em transversal}, meaning that if $f_1(p_1) = f_2(p_2) = q
\in Y$ then
\begin{equation}
	(f_1)_\ast (T_{p_1} X_1) + (f_2)_\ast (T_{p_2} X_2) = T_q Y,
	\label{E:intro_transversality}
\end{equation}
for in this case the set-theoretic fiber product
\begin{equation}
\begin{gathered}
X_1 \times_Y X_2 = \set{(p_1,p_2) \;;\; f_1(p_1) = f_2(p_2)} \subset X_1 \times X_2, \\
h_i(p_1,p_2) = p_i,
\end{gathered}
	\label{E:intro_set_fib_prod}
\end{equation}
is a smooth manifold and the maps afforded by the universal property of fiber
products are smooth.

The natural analog, `b-transversality', of \eqref{E:intro_transversality} in the setting of
manifolds with corners is the requirement that
\begin{equation}
	\bd(f_1)_\ast (\bT_{p_1} X_1) + \bd(f_2)_\ast (\bT_{p_2} X_2) = \bT_q Y.
\label{E:intro_b_transversality}\end{equation}
Under this condition, \eqref{E:intro_set_fib_prod} is not necessarily a
manifold, but can be decomposed as a union of what are here called `interior
binomial subvarieties'. These are objects generalizing manifolds with corners,
with smooth interiors but mild singularities at the boundary and boundary
faces which are of the same type. They can be resolved, by generalized
boundary blow-up, to manifolds with corners.

As for a manifold, there is a natural monoidal complex $\cP_D$ defined over the
boundary faces of an interior binomial subvariety $D \subset X,$ the difference
being that the monoids may not be freely generated, i.e. may not be smooth.
If the monoidal complex is smooth, then $D$ has a natural structure of a smooth
manifold although this may not be induced from $X,$ in that $D$ may not be
embedded.  Even if the complex is not smooth, there is a smooth manifold $[D; 
\cR] \to D$ corresponding to every smooth refinement $\cR \to \cP_D.$

In the case of fiber products, the monoids in the monoidal complex over $X_1
\times_Y X_2$ are of the form
\begin{equation}
	\sigma_{F_1} \times_{\sigma_G} \sigma_{F_2}, \quad F_i \in \cM(X_i),
	\label{E:intro_fib_prod_monoids}
\end{equation}
where $G = (f_1)_\#(F_1) \cap (f_2)_\#(F_2) \in \cM(Y).$

\begin{thmb}
If $f_i : X_i \to Y$ are b-maps which satisfy \eqref{E:intro_b_transversality},
and if each of the monoids \eqref{E:intro_fib_prod_monoids} is freely
generated, then there exists a universal fiber product of $f_1$ and $f_2$ in
the category of manifolds with corners, which is given by a decomposition of
\eqref{E:intro_set_fib_prod} into a finite union of sets with natural smooth
structures.
\end{thmb}
\noindent This result extends a theorem of Joyce \cite{joyce2009manifolds}
on fiber products of more restricted maps.

In the general case, when the monoids are not necessarily smooth, there is
no such universal object and it is in general necessary to blow up to get
smoothness and also to factor maps.

\begin{thmc} 
For every smooth refinement $\cR$ of the complex $\cP_{X_1\times_Y X_2},$ there
is a manifold $[X_1 \times_Y X_2\;;\; \cR]$ with b-maps to $X_i$ commuting with
the $f_i :X_i \to Y$ and such that if $g_i :Z \to X_i$, $i = 1,2$ are b-maps
commuting with the $f_i$ then there exists a generalized blow-up $[Z;\cS]
\to Z$ and a unique map $g : [Z; \cS] \to [X_1\times_Y X_2 \;;\; \cR]$
giving a commutative diagram
\[
\begin{tikzpicture}[->,>=to,auto]
\matrix (m) [matrix of math nodes, column sep=1cm, row sep=1cm, text depth=0.25ex]
{ [Z; \cS] & & \\ Z & {[X_1\times_Y X_2;\cR]} & X_2 \\ & X_1 & Y. \\};
\path (m-1-1) edge [dashed] node {$g$} (m-2-2);
\path (m-1-1) edge [dashed] node {$\beta_{\cS}$} (m-2-1);
\path (m-2-2) edge node {$h_2$} (m-2-3);
\path (m-2-2) edge node {$h_1$} (m-3-2);
\path (m-3-2) edge node {$f_1$} (m-3-3);
\path (m-2-3) edge node {$f_2$} (m-3-3);
\path (m-2-1) edge [bend left] node  {$g_2$} (m-2-3);
\path (m-2-1) edge [bend right] node  {$g_1$} (m-3-2);
\end{tikzpicture}
\]
\end{thmc}

There is a close relationship between the material discussed here and the
theory of logarithmic structures in algebraic geometry.  The monoidal complexes
we describe are related to `fans' as defined by Kato \cite{kato1994toric},
which essentially are to (toric) monoids what schemes are to rings.  Our
generalized blow-up is then related to a result of Kato, which produces a
resolution of a logarithmic scheme with `toric singularities' (also called a
`log-smooth scheme') from a subdivision of the fan associated to the scheme.
Our manifolds with corners and binomial subvarieties can be viewed as log
smooth spaces, albeit with a stronger differentiable structure than is usually
considered in algebraic geometry. There is however an important freedom
afforded by working over $\R_+$ (as is the case for manifolds with corners and
boundary blow-up) rather than $\R$ or $\C$ which is that affine charts may be
glued together by maps involving radicals and still result in smooth objects
(as opposed to quite singular objects like stacks in the more conventional
toric settings). We make heavy use of this fact and note that many of the
smooth maps we produce would not even be birational in an algebraic setting
over $\C.$

In Section \ref{S:bnotation} we briefly recall the theory of manifolds with
corners and b-maps, and establish some notation.  Section \ref{S:monoids}
contains a discussion of monoids (here always meaning what are often called
`toric monoids') and their refinements, which is enough to establish the local
version of generalized blow-up in Section \ref{S:local}.  In Sections
\ref{S:complexes} and \ref{S:fib_prod_complexes} we discuss the theory of
monoidal complexes and refinements of these.  In Section \ref{S:global} we
complete one half of Theorem A, showing the existence of a generalized blow-up
of $X$ given a smooth refinement of $\cP_X$, along with the lifting results for
b-maps, before exhibiting ordinary boundary blow-up and its iterated version as
special cases in Section \ref{S:blowup}.  Section \ref{S:characterization}
completes Theorem A, giving the characterization of generalized blow-down maps
by the properties \eqref{E:intro_blow_down_diffeo} and
\eqref{E:intro_blow_down_bdiffl}.  In Sections \ref{S:binvars} and
\ref{S:binvarres} we discuss the theory of interior binomial varieties and
their resolution, finally applying this to fiber products in Section
\ref{S:fiber}.

\begin{ack}
The authors would like to thank Dan Abramovich, Pierre Albin, William Gillam,
Daniel Grieser, Samouil Molcho and Michael Singer for helpful discussions
during the preparation of this manuscript, and to thank the referee for many insightful
comments. We should also like to point out the important role in the origins of
this work played by Umut Varolgunes and Jonathan Zhu who investigated the
resolution of sums of positive monomials in the model, Euclidean, case in an
`Undergraduate Research Opportunity' at MIT.
\end{ack}

}
\section{Manifolds with corners} \label{S:bnotation}

In this section, we fix notation used for manifolds with corners and b-maps.
For background, see \cite{melrosedifferential} and \cite{melrose1992calculus}.  Set
\[
	\R_+ = [0,\infty)\ \text{and}\ \Zp = \set{0,1,2,\ldots}.
\]
The model manifold with corners is a product
\[
	\bbR^{n,k}=\R^k_+ \times \R^{n-k},
\]
for $k \in \set{0,\ldots, n},$ on which the smooth functions, forming the ring
$\CI(\bbR^{n,k})=\CI(\bbR^n)\big|_{\bbR^{n,k}},$ are taken to be those
obtained by restriction from the smooth functions on $\bbR^n.$

In general, a {\em manifold with corners} $X$ is a (paracompact, Hausdorff)
topological manifold with boundary, $X,$ with a ring of smooth functions,
$\CI(X)$ with respect to which it is everywhere locally diffeomorphic
to one of these model spaces. Thus, $X$ has a covering by coordinate
patches with homeomorphisms to open subsets of the local model spaces such
that $u\in\CI(X)$ if and only if its image in each coordinate system is
smooth on the model. The pull-backs of the functions on the model spaces
give local coordinate systems.

Each point $p \in X$ necessarily has a well-defined (boundary) {\em
  codimension} given by the number of independent non-negative coordinate
functions vanishing at $p$ in such a coordinate system. A boundary face of
codimension $k$ is the closure of one of the components of the set of
points of codimension $k;$ the set of such faces is denoted $\cM_k(X).$
In particular $\cM_1(X)$ consists of {\em boundary hypersurfaces} and
$\cM_0(X) = X$ (or the set of components of $X$ if it is not connected). We
set $\cM(X) = \bigcup_k \cM_k(X),$ this is a partially ordered set under inclusion.

As part of the definition of a manifold with corners, we require that all
boundary  hypersurfaces $H \in \cM_1(X)$ be embedded; this is equivalent to
insisting that the ideal of smooth functions vanishing on $H,$ $\cI_H \subset
\CI(X),$ is principal. A non-negative generator $\rho _H\in\CI(X)$ of this
ideal is a \emph{defining function,} so $\cI_H=\rho _H\cdot\CI(X).$ It follows
that each boundary face of $X$ is itself a manifold with corners. While it
would be possible to drop this requirement of embedded hypersurfaces and still
retain many of the results in this paper, manifolds violating this property
tend to be very ill behaved from an analytical standpoint, and often the first
step when encountering such a space is to resolve it via boundary blow-up to a
space {\em with} embedded hypersurfaces. In light of this, the simplification
in the combinatorial description of boundary faces which results from the
stronger requirement is a worthy trade-off.

The diffeomorphisms of $X,$ homeomorphisms which map $\CI(X)$ to itself,
must preserve the stratification by boundary codimension, and the
infinitesimal diffeomorphisms correspond to the Lie subalgebra 
\begin{equation}
	\cV_\mathrm{b}(X) \subset \cV(X) 
	\label{E:b_vector_fields}
\end{equation}
of the usual algebra of real smooth vector fields, consisting of those vector fields
which are tangent to each of the boundary faces of $X.$ This subalgebra is
the space of global smooth sections of the {\em b-tangent bundle}, $\bT X
\longrightarrow X.$ Locally, whenever $(x,y) : U \to \R^k_+\times \R^{n-k}$
are coordinates centered at $p,$
\[
\bT_p X = \sspan_\R\set{x_1\pa_{x_1},\ldots,x_k\pa_{x_k},\pa_{y_{k+1}},\ldots,\pa_{y_n}}.
\]

From \eqref{E:b_vector_fields}, there is a natural evaluation map
\begin{equation}
	\bT X\longrightarrow T X
	\label{E:bT_to_T}
\end{equation}
which is an isomorphism over the interior of $X$ but not over the boundary.
Over interior points $p \in \mathring F \in \cM_k(X)$ (and extending to all of
$F$ by continuity), the kernel of \eqref{E:bT_to_T} defines the {\em b-normal
bundle}:
\[
	\bN F\longrightarrow F,\ \bN_pF =
        \sspan_\R\set{x_1\pa_{x_1},\ldots,x_k\pa_{k_k}},\ \text{locally.}
\]
In fact, up to reordering, the sections $x_i\pa_{x_i}$, $i = 1,\ldots,k$
of $\bN F$ are well-defined independent of coordinates since any other set
$(x'_1,\ldots,x'_k)$ must have the form (reordering if necessary) $x'_i =
a_i(x,y)x_i$ with $a_i > 0$, so
\[
	x'_i\pa_{x'_i} = x_i\pa_{x_i} + \cO(x)\cV_\mathrm{b}(X),
\]
and hence $x'_i\pa_{x'_i} \equiv x_i\pa_{x_i}$ at $F.$

Thus, $\bN F \to F$ is canonically trivial, with a well-defined lattice
\begin{equation}
	(\bN F)_\Z = \sspan_\Z\set{x_1\pa_{x_1},\ldots,x_k\pa_{x_k}}
	\label{E:bN_lattice}
\end{equation}
and cone of {\em inward-pointing} vectors
\[
	\bN_+ F = \sspan_{\R_+}\set{x_1\pa_{x_1},\ldots,x_k\pa_{x_k}}.
\]
These play a fundamental role in the discussion here.

A {\em b-map} $f : X \to Y$ between manifolds with corners is a map which
is smooth, meaning $f^\ast\CI(Y)\subset \CI(X),$ and is such that $f$ pulls
back each (principal) ideal $\cI_H$ to either a product of powers of 
such ideals on $X$,
\begin{equation}
	f^\ast\cI_H =  \prod_{G \in \cM_1(X)} \cI^{\alpha(G,H)}_G,\ \alpha(G,H) \in \Zp
\label{E:f_on_ideals}\end{equation}
to the zero ideal, $f^\ast \cI_H  = 0$ or to  $\CI(X).$ If the second case
does not occur we say $f$ is an {\em interior} b-map. Otherwise $f :
X\longrightarrow F$ is an interior b-map for some $F \in \cM(Y).$ We shall mostly 
be concerned with interior b-maps in this paper. In practical terms \eqref{E:f_on_ideals}
means that $f$ has the local form
\[
\begin{gathered}
	f : \R^n_+ \times \R^k \ni (x,y) \mapsto (x',y') \in \R^{n'}_+ \times \R^{k'}\\
	(x'_1,\ldots,x'_{n'}) = \pns{a_1(x,y)\prod_i x_i^{\alpha(i,1)}, \ldots, a_{n'}(x,y)\prod_i x^{\alpha(i,n')}} \\
	y' = b(x,y),
\end{gathered}
\]
where the $a_i(x,y)$ are strictly positive smooth functions and $\alpha(i,j) =
\alpha(G,H)$ where $G$ and $H$ are the boundary hypersurfaces for which $x_i$
and $x'_j$ are locally defining functions, respectively.

For an interior b-map the differential $f_\ast: T X\longrightarrow T Y$
extends by continuity from the interior to the {\em b-differential}
\[
	\bd f_\ast : \bT_p X \to \bT_{f(p)} Y,\ \forall\ p\in X.
\]	
We denote by $f_{\#}:\cM(X)\longrightarrow \cM(Y)$ the map which assigns to
each boundary face of $X$ the boundary face of $Y$ of largest
codimension which contains it. Then the differential restricts to a map
\begin{equation}
	\bd f_\ast : \bN F \to \bN f_{\#}(F)
	\label{E:bf_on_bN}
\end{equation}
which is integral with respect to the lattices \eqref{E:bN_lattice}. Indeed, 
\eqref{E:bf_on_bN} is given by a matrix with integer entries coming
from the exponents $\alpha(G,H)$ in \eqref{E:f_on_ideals}, and $\bd f_\ast$ 
maps inward-pointing vectors to inward-pointing vectors since these entries
are non-negative.

It is convenient to use multi-index notation for certain maps in
coordinates.  If $t = (t_1,\ldots,t_n)$ and $x = (x_1,\ldots,x_k)$ are
local coordinates on spaces $U$ and $V$ respectively, and $\mu \in
\Mat(n\times k, \R)$ is a matrix, the map $(t_1,\ldots,t_n) \mapsto (x_1,\ldots,x_n)$ where 
\begin{equation*}
x_i = \prod_j t_j^{\mu_{ji}}\text{ is denoted } t \mapsto t^\mu = x.
\end{equation*}
With this convention on the order of the indices, $\pns{t^\mu}^\nu = t^{\mu\nu}.$

\section{Monoids}\label{S:monoids}

In general terms a monoid is a set which is closed under an associative binary
operation (usually required to be commutative and have an identity element), so
in essence a group without inverses or a ring without addition. Here we
restrict attention to monoids of the special type usually known as \emph{toric
monoids}, which for our purposes may be characterized as follows.

\begin{defn}
A {\em toric monoid} is a set $\sigma$, closed under a binary operation of
addition, which can be expressed as the intersection 
\[
	\sigma = N_\sigma \cap C
\]
of a finitely generated integral lattice $N_\sigma$ and a proper convex
polyhedral cone $C$ in the vector space $N_\sigma^\R = N_\sigma \otimes_\Z \R$
which is integral with respect to $N_\sigma$, and for which $N_\sigma^\R =
\sspan_{\R}\sigma$. The cone, which is determined from the monoid by $C =
\sspan_{\R_+} \sigma \subset N_\sigma^\R$, will be called the {\em support} of
$\sigma$ and denoted by
\[
	\supp(\sigma) = C \subset N_\sigma^\R.
\]
The minimal integral generators $\set{v_1,\ldots,v_k} \subset N_\sigma$ of the
cone $\supp(\sigma)$ (which do not necessarily generate $\sigma$ as a monoid)
will be called the {\em extremals} of $\sigma$; by assumption no $v_i$ is a
combination of the others with non-negative coefficients but every point of
$\supp(\sigma)$ is a non-negative linear combination of the $v_i.$
\label{D:monoid}
\end{defn}

\noindent Toric monoids may be characterized equivalently (though more
abstractly) as follows \cite{ogus2006lectures}. For $\sigma$ a commutative,
finitely generated monoid with identity, there is a canonical abelian group
$\sigma^{\mathrm{gp}}$ which is universal with respect to monoid homomorphisms
from $\sigma$ to groups. Then $\sigma$ is a toric monoid provided it is {\em
sharp}, meaning $\sigma$ has no invertible elements besides $0$; {\em
integral}, meaning that cancellation holds: if $x + y = z + y \in \sigma$ then
$x = z \in \sigma$ (equivalently the map from $\sigma$ to
$\sigma^{\mathrm{gp}}$ is injective); {\em saturated}, meaning that if $n\,v
\in \sigma$ for some $n > 0$ and $v \in \sigma^{\mathrm{gp}}$, then $v \in
\sigma$; and  {\em torsion free}, meaning that $\sigma^{\mathrm{gp}}$ is
torsion free. Indeed under these conditions $N_\sigma := \sigma^{\mathrm{gp}}$
is a lattice, $C = \sspan_{\R_+} \sigma
\subset N_\sigma \otimes_{\Z} \R$ is a proper convex integral cone and $\sigma
\equiv N_\sigma \cap C$.

For the remainder of the paper, {\em monoid} will always mean {\em toric monoid}.

The {\em dimension} $\dim(\sigma)$ is the dimension of $N_\sigma^\R$;
equivalently, $\dim(\sigma)$ is the maximum number of linearly independent
extremals. A monoid $\sigma$ is said to be {\em simplicial} if the
extremals $\set{v_1,\ldots,v_n}$ are independent, in which case
$\dim(\sigma) = n$ and $\supp(\sigma)$ is a cone over the $n-1$ simplex
defined by $\set{v_1,\ldots,v_n}.$ A simplicial monoid is further said to
be {\em smooth} if it is generated as a monoid by its extremals. A smooth
monoid is therefore freely generated and isomorphic to $\Zp^n$. We will
use the notation $\sigma = \Zp\pair{v_1,\ldots,v_n}$ to denote a smooth
monoid freely generated by independent vectors $v_1,\ldots,v_n$.

One monoid $\sigma'$ in $N$ is a {\em submonoid} of another $\sigma$ if
$\sigma'\subset\sigma$ so the generators of $\sigma'$ are non-negative integral
combinations of the generators of $\sigma.$  

A monoid $\tau$ is a {\em face} of $\sigma$, written $\tau \leq \sigma$, if
$\tau$ is a submonoid such that whenever $v,w \in \sigma$ and $v + w \in \tau$, then
$v,w \in \tau$.  Equivalently, $\tau$ is the largest submonoid contained in a
face (in the sense of cones) of $\supp(\sigma)$, which is in turn equivalent to
the existence of a functional $u \in N_\sigma^\ast$ such that
\[
\begin{aligned}
	\pair{u,v} &= 0\quad \text{for $v \in \tau$,} \\
	\pair{u,v} &> 0\quad \text{for $v \in \sigma \setminus \tau$.}
\end{aligned}
\]
In particular, the trivial monoid $\set{0}$ with no generators is a face of
every monoid.



A monoid homomorphism $\phi : \sigma' \to \sigma$ is just an addition-preserving map.
We will also denote by $\phi$ the induced linear maps $\phi : N_{\sigma'} \to N_\sigma$,
$\phi : N^\R_{\sigma'} \to N^\R_{\sigma}$ and $\phi : \supp(\sigma') \to \supp(\sigma).$
Note that unless $\phi$ is injective, the image of
$\sigma'$ in $\sigma$ need not be a monoid in the sense we have defined, since
it may not be saturated (consider the map $\Zp^2 \to \Zp$ sending $(1,0)$ to
$2$ and $(0,1)$ to $3$, for instance). The inverse image of a face of $\sigma$
is always a face of $\sigma'.$

\begin{defn} Given a monoid $\sigma$, a {\em refinement} of $\sigma$ is a collection $R$
of submonoids of $\sigma$, such that
\begin{enumerate}[{\normalfont (i)}]
\item \label{I:monref_1} if $\sigma' \in R$ and $\tau' \leq \sigma'$, then $\tau' \in R$,
\item \label{I:monref_2} for any two monoids $\sigma'_1,$ $\sigma'_2 \in R$, $\sigma'_1 \cap
	\sigma'_2$ must be a face of both $\sigma'_1$ and $\sigma'_2$, and
\item \label{I:monref_3} $\bigcup_i \supp(\sigma_i) = \supp(\sigma)$ (as viewed in the vector space $N_\sigma^\R$ using the natural inclusions $N_{\sigma_i}^\R \to N_\sigma^\R$).
\end{enumerate}
We say $R$ is a {\em simplicial} (resp.\ {\em smooth}) refinement if 
each $\sigma_i \in R$ is simplicial (resp.\ smooth).
The set of faces of $\sigma$ (including $\sigma$ itself) form the {\em trivial
refinement} of $\sigma.$
\label{D:monoid_refinement}
\end{defn}

\noindent These conditions imply that the cones $\supp(\sigma_i)$ form a {\em
fan} --- in the classical sense used in the theory of toric varieties
\cite{fulton1993introduction} --- whose support is equal to that of $\sigma$, though
it is important to note that this notion of refinement differs from the standard
one in toric geometry in that the monoids need not all be defined with respect
to a single lattice. In particular the `Fulton-style' fan refinement underlying
a smooth refinement need not be smooth.

A refinement of $\sigma$ is a special case of a {\em monoidal complex},
discussed in Section~\ref{S:complexes}.

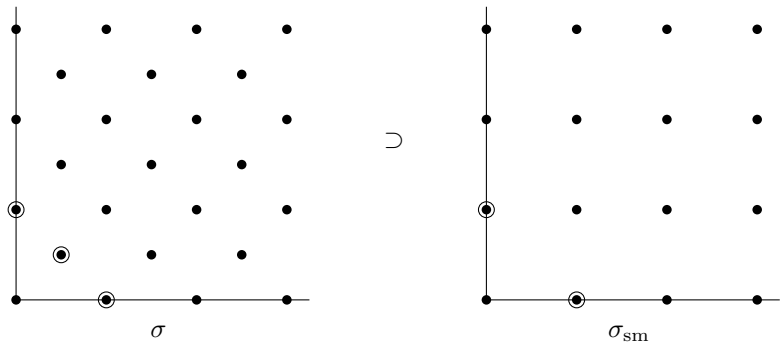
\begin{figure}[tb]
\begin{tikzpicture}[->,>=to,auto]
\matrix (m) [matrix of nodes, column sep=2cm]
{ 
 \begin{tikzpicture}[scale=0.6]
 \draw[-] (0,0) -- (6.5,0);
 \draw[-] (0,0) -- (0,6.5);
 \foreach \x in {0,2,4,6} {
	 \foreach \y in {0,2,4,6} {
		 \fill (\x,\y) circle (3pt);
		 }
	 }
 \foreach \x in {1,3,5} {
	 \foreach \y in {1,3,5} {
		 \fill (\x,\y) circle (3pt);
		 }
	 }
 \draw (2,0) circle (5pt);
 \draw (0,2) circle (5pt);
 \draw (1,1) circle (5pt);
 \end{tikzpicture}
& 
 \begin{tikzpicture}[scale=0.6]
 \draw[-] (0,0) -- (6.5,0);
 \draw[-] (0,0) -- (0,6.5);
 \foreach \x in {0,2,4,6} {
	 \foreach \y in {0,2,4,6} {
		 \fill (\x,\y) circle (3pt);
		 }
	 }
 \draw (2,0) circle (5pt);
 \draw (0,2) circle (5pt);
 \end{tikzpicture}
\\
$\sigma$ & $\sigma_{\mathrm{sm}}$
\\};
\path (m-1-1) to node {$\supset$} (m-1-2);
\end{tikzpicture}

\caption{A non-smooth, simplicial monoid $\sigma$, along with its smoothing 
$\sigma_{\mathrm{sm}}$.  Generators of each monoid are highlighted.}
\label{F:nonsmooth}
\end{figure}

\begin{ex} A simplicial monoid, $\sigma,$ with extremals $\set{v_1,\ldots,v_n},$
has a canonical smooth refinement consisting of the faces of the smooth 
monoid 
\[
	\sigma_{\mathrm{sm}} := \Zp\pair{v_1,\ldots,v_n}.
\]
These are of the form $\Zp\pair{w_1,\ldots,w_k}$ for all subsets 
$\set{w_1,\ldots,w_k} \subseteq \set{v_1,\ldots,v_n}$; this refinement will
be called the {\em smoothing} of $\sigma.$  See Figure~\ref{F:nonsmooth}.
\label{X:smoothing}
\end{ex}
This example illustrates the point that while $\supp(\sigma)$ is covered by
supports of the $\tau \in R$, $\sigma$ need not be the union of the $\tau \in
R$ as a monoid.  A simpler example of this phenomenon is the refinement of
$\Zp$ given by the submonoid $2 \Zp = \set{0, 2, 4, \ldots}$. 

It is important to observe that this smoothing operation is rather nonstandard;
in the theory of complex toric varieties for instance, it may generate a map of
varieties which is not even birational. However in the category of manifolds
with corners (where `$\R_+$-points' are considered and radicals are
consequently well-behaved) it becomes very useful and will be fundamental to
our results below.

\begin{lem}
If $R$ is a refinement of $\sigma$ and $\tau \leq \sigma$ is a face then
\begin{equation}
	R(\tau) = \set{\sigma' \in R \;;\; \sigma' \subset \tau}
	\label{E:localization_of_refinement}
\end{equation}
is a refinement of $\tau.$
\label{L:refinements_and_faces}
\end{lem}

\begin{proof}
Clearly $R(\tau)$ satisfies the first two properties of a refinement, since
if $\sigma' \in R(\tau)$ and $\tau' \leq \sigma'$, then $\tau' \subset
\tau$ and hence $\tau' \in R(\tau)$.  For the third property, choose any $v
\in \supp(\tau)$.  Since also $v \in \supp(\sigma)$, there must be a $\sigma'
\in R$ for which $v \in \supp(\sigma')$.  By the support condition,
$\supp(\sigma') \cap \supp(\tau)$ is a face of $\supp(\sigma')$, which
corresponds to a monoid $\tau' \leq \sigma'$ such that $\supp(\tau') \subset
\supp(\tau)$.
\end{proof}

A particularly important example of a refinement is the operation of {\em star
subdivision}, which is well-known in toric geometry
\cite{fulton1993introduction} and its generalizations, though the definition
below is not standard. (In the classical toric theory, subdivision of a smooth
fan need not be smooth, for instance.)

\begin{prop} If $\sigma$ is a monoid and $0\not= v\in \sigma$ the collection,
$S(\sigma,v),$ of monoids consisting of 
\begin{enumerate}[{\normalfont (i)}]
\item all faces $\tau \leq \sigma$ such that $v \notin \tau$, and
\item monoids $\tau + \Zp v := \sspan_{\Zp} \tau \cup \set{v}$ where $\tau \leq
	\sigma$ and $v \notin \tau$
\end{enumerate}
is a refinement of $\sigma.$ If all proper faces of $\sigma$ are smooth,
then $S(\sigma,v)$ is a smooth refinement.
\label{P:star_subdivision}
\end{prop}

\begin{proof} 
For two monoids $\tau_1$, $\tau_2$ of the first type, the intersection $\tau_1
\cap \tau_2$ is a face of $\sigma$ not containing $v$, hence a monoid of the
first type.  For two monoids $\tau_1 + \Zp v$ and $\tau_2 + \Zp v$ of the second type,
the intersection is a monoid $(\tau_1 \cap \tau_2) + \Zp  v$ of the second type, and
for one of each type, the intersection $(\tau_1 + \Zp  v) \cap \tau_2 = \tau_1 \cap
\tau_2$ is of the first type.  The support condition follows from the fact that
any $w \in \supp(\sigma)$ lies in some cone of the form $\supp(\tau + \Zp v)$, with
$\tau \leq \sigma$.

If $\tau = \Zp \pair{v_1,\ldots,v_k}$ is smooth and $v \notin \tau$, then $\tau
+ \Zp  v$ has independent generators $\set{v_1,\ldots,v_k,v}$, so $\tau + \Zp 
v = \Zp \pair{v_1,\ldots,v_k,v}$ is smooth. Since $\sigma$ itself is not in
$S(\sigma,v)$ smoothness follows from the smoothness of all the proper
faces of $\sigma.$
\end{proof}

In the context of generalized blow-up, star subdivision is the operation which
realizes the ordinary blow-up of a boundary face of a manifold.  There is a
similar construction, ``planar refinement,'' which we
discuss next although it is not used until Section~\ref{S:binvarres}, in which
$v$ is replaced by the intersection of $\sigma$ with a subspace.

First, as a matter of notation, suppose $\tau_i \subset \sigma$, $i = 1,2$ are
{\em full submonoids} in the sense that each 
\begin{equation}
	\tau_i = \sigma \cap \supp(\tau_i).  
	\label{E:full_submonoid}
\end{equation}
Then we define the {\em join} of $\tau_1$ and $\tau_2$ to be the full submonoid
\begin{equation}
	\tau_1 \ast \tau_2 := \sigma \cap \big(\supp(\tau_1) + \supp(\tau_2)\big)
\label{E:join}\end{equation}
where $\supp(\tau_1) + \supp(\tau_2) = \sspan_{\R_+}\pns{\tau_1 \cup \tau_2}$ is
the convex hull of the supports of the $\tau_i$.

Let $\sigma$ be a smooth monoid, and $\mu \subset \sigma$ be a submonoid
which is obtained by intersection with an integral subspace of $N_\sigma^\R$:
\begin{equation}
	\mu = \sigma \cap M, \quad M \subset N_\sigma^\R \ \text{integral.}
	\label{E:mu_intersection}
\end{equation}
Note that $\mu$ need not be simplicial, though it is a full submonoid 
in the sense of \eqref{E:full_submonoid}. Let $\Lambda$ consist of the maximal faces
of $\sigma$ meeting $\mu$ trivially:
\begin{equation}
	\Lambda = \set{\tau \leq \sigma \;;\; \tau \cap \mu = \set{0},\
	\text{but for all}\ \tau' > \tau,\ \tau' \cap \mu \neq \set{0}}.
\label{E:maximal_faces}
\end{equation}

\begin{lem} 
If $\sigma$ is a smooth monoid, $M$ is an integral subspace  and $\mu$ and
$\Lambda$ are defined by \eqref{E:mu_intersection} and \eqref{E:maximal_faces}
then for any $\tau_1 \neq \tau_2 \in\Lambda,$ there exists a functional $u\in
N^*_{\sigma}$ such that $\pair{u,\tau_1} \geq 0,$ $\pair{u,\tau_2} \leq 0,$
$\pair{u,\mu} = 0$ and if $v$ is a generator of either $\tau_1$ or $\tau_2$
then $v\in\tau_1\cap\tau_2$ if and only if $\pair{u,v} = 0.$
\label{L:planar_refinement_separation}
\end{lem}

\begin{proof} 
If $M$ has full dimension the result is trivial, so we may assume that $M$ is a
proper subspace. If $M$ is a hyperplane, then $u$ is a defining functional,
$u^\perp = M,$ and is determined up to non-vanishing constant. Since
$\tau_i\cap M=\{0\},$ $u$ is definite on each and cannot have the same sign on
both or they would be equal by maximality. Thus, either $\pair{u,\tau_1} > 0,$
$\pair{u, \tau_2} < 0$ or with signs reversed. In this case their intersection
is trivial and the result clearly holds.  Conversely if $\tau_1 \cap \tau_2
\neq \set{0},$ then $\codim(M) > 1.$

In the general case we proceed by induction on $\codim(M)$. 
Choose a generator $w \in \tau_1\cap \tau_2,$ and consider the smooth monoid
$\sigma' < \sigma$ generated by all the generators of $\sigma$ except $w$, with
$\tau'_i := \tau_i \cap \sigma'$.  Consider the subspace $M' = (M+\R\,w)\cap
N_{\sigma'} \subset N_{\sigma'}$ and set $\mu'=M'\cap \sigma'.$  Since $w\notin
M,$ $M'$ has the same dimension as $M,$ hence the codimension of $M'$ in
$N_{\sigma'}$ is one less than the codimension of $M$ in $N_\sigma$. 

Now, the $\tau'_i$ are maximal faces of $\sigma'$ not meeting $\mu'.$
To see this first note that $\tau'_i\cap M'=\set{0}$ since the generators of $\mu$ can be
written $e_l=e'_l+c_lw$ where the $e'_l$ are generators for $\mu'$ and $c_l\geq
0.$ Thus, if there was a point $p\in\tau'_i\cap\mu',$ necessarily of the form
$p=\sum_{l}d_le'_l$ with the $d_l\geq 0$, then there would be a point
$p+\sum\limits_{l}d_lc_lw\in\tau_i\cap\mu.$  Similarly, the $\tau'_i$ are
maximal since if $\tau' > \tau'_i$ is a face of $\sigma'$ which meets $\mu'$
trivially then $\tau'+\Zp w > \tau_i$ meets $\mu$ trivially, contradicting the
maximality of $\tau_i$. 

Thus by induction on the codimension of $M'$ in $N_{\sigma'}$ there exists a
functional $u'\in N^\ast_{\sigma'}$ with the desired separating property for
the $\tau'_i$ and $M'.$  Extending $u'$ to $u$ by requiring $u(w)=0$ gives a
functional in $N^\ast_{\sigma}$ with the desired properties and completes the
inductive step.  
\end{proof}

Now set $\mu \ast \Lambda = \set{\mu \ast \tau\;;\; \tau \in \Lambda},$
defined by \eqref{E:join}, and consider 
\begin{equation}
	S(\sigma,\mu) = \set{\gamma \leq \sigma'\;;\; \sigma' \in \mu \ast \Lambda}.
	\label{E:planar_refinement}
\end{equation}

\begin{prop}[Planar refinement] 
If $\sigma$ is smooth monoid and $\mu = \sigma \cap M$ is the intersection with
an integral subspace $M \subset N_\sigma^\R$ then
\eqref{E:planar_refinement} gives a refinement of $\sigma$ containing
$\mu;$ this operation commutes with faces in the sense that for any face $\tau
\leq \sigma$, $S(\tau, \tau \cap \mu) = \big(S(\sigma,\mu)\big)\pns{\tau}$ with
notation as in \eqref{E:localization_of_refinement}.
\label{P:planar_refinement} 
\end{prop}

\begin{proof} 
If $M = N_\sigma^\R$, then $\mu = \sigma$, $\Lambda = \set{0}$ and
$\cS(\sigma,\mu) = \mu \ast \set{0}$ consists of $\sigma$ and its faces, which
is the trivial refinement of $\sigma.$

First we show that the supports of $S(\sigma,\mu)$ cover $\supp(\sigma),$ in fact 
\begin{equation}
	\sigma =\bigcup_{\tau\in\Lambda }\mu\ast\tau,
	\label{E:planar_cover}
\end{equation}
since all the monoids are full as above.  Certainly $\mu\subset\tau\ast\mu,$ so
consider $v \in \sigma\setminus\mu$ and choose any $m_1\in \mu.$ The ray from
$m_1$ through $v$ meets some face of $\sigma$, call it $\tau_1$, so $v$ is a
positive combination of $m_1 \in \mu$ and $t_1\in\tau_1.$ If
$\tau_1\cap\mu=\set{0}$ then $\tau_1\subset\tau\in\Lambda$ and we are finished.
On the other hand, if $\tau_1$ meets $\mu$, then $t_1$ is a positive
combination of some $m_2 \in \mu$ and $t_2\in\tau_2$ with $\tau_2$ a proper
face of $\tau_1.$ Continuing this way eventually shows that $v\in\mu\ast\tau_k$
for some $\tau_k$ which does not meet $\mu,$ since the dimension decreases and
$v\notin\mu.$ Thus $v\in\mu\ast\tau$ for some $\tau\in\Lambda$ with $\tau \geq
\tau_k$ and \eqref{E:planar_cover} follows.

Next consider the intersection of two elements in the union
\eqref{E:planar_cover}. We will show that
\begin{equation}
	(\mu\ast\tau_1)\cap(\mu\ast\tau_2)=\mu\ast(\tau_1\cap\tau_2),\ \tau_i\in\Lambda.
	\label{E:planar_intersection}
\end{equation}
Certainly the right side is contained in the left.
Lemma~\ref{L:planar_refinement_separation} applies and gives a functional
$u.$ If $v$ is in the intersection it follows that $u(v)\ge0$ and $u(v)\le0$ so
$u(v)=0$ and from the last property of $u$ this implies
$v\in\mu\ast(\tau_1\cap\tau_2).$  Since $u$ is a supporting hyperplane for both
$\mu\ast\tau_i$, it follows that $\mu\ast(\tau_1\cap\tau_2)$ is a boundary face
of each of these full submonoids. This in turn implies that the intersection of
any two elements of $S(\sigma ,\mu)$ is a common boundary face of both and
hence an element of $S(\sigma,\mu).$ Indeed, such an intersection must be
contained in two $\mu\ast\tau_i$ with $\tau_i\in\Lambda.$ If they are the same,
the conclusion is immediate, and if not then they are both contained in the common
boundary face \eqref{E:planar_intersection} and again the result follows.

That $\mu$ is a boundary face of each $\mu\ast\tau,$ $\tau\in\Lambda$
follows from existence of a functional $u\in N^*_{\sigma}$ such that $u>0$
on the generators of $\tau$ and $u(M)=0.$
\end{proof}

For any $v \in \sigma$, if $\Zp  v$ is a full submonoid, then $S(\sigma,\Zp  v)$ is
just the ordinary star subdivision of $\sigma$ along $v$.  Note however, that
in the case of star subdivision we can relax the condition that $\sigma$ be
smooth, as well as the condition that $\Zp  v$ be full.

\section{Generalized blow-up of $\R^n_+$}\label{S:local}

In preparation for the global case of generalized boundary blow-up of a
manifold with corners treated in Section~\ref{S:global}, we first discuss the
`local case' of generalized blow-up of the model space $\R^n_+.$ Associated
with this basic space is the `basic monoid' which is freely generated by
the boundary hypersurfaces of $\R^n_+$, with an equivalent realization
as the non-negative integral points in $\bN\set{0}$, the b-normal to the maximum
codimension boundary face $\set{0}$. We show how to construct a blown-up space
mapping surjectively onto $\R^n_+$ corresponding to any smooth refinement of
this basic monoid. The functoriality of this operation with respect to
diffeomorphisms and b-maps is then considered.

Consider $\R^n_+$ with coordinates $(x_1,\ldots,x_n)$ and let 
$H_i = \set{x_i = 0} \in \cM_1(\R^n_+)$ denote the boundary hypersurfaces.
The {\em basic monoid of $\R^n_+$} is the smooth monoid freely generated
by the $H_i$:
\[
	\sigma_{\R^n_+} = \Zp\pair{H_1,\cdots, H_n}
\]
There is a natural embedding of $\sigma_{\R^n_+}$ into the vector space 
\[
	\bN \set{0} = \sspan_\R\set{x_1\pa_{x_1},\ldots,x_n\pa_{x_n}}.
\]
Indeed, as pointed out in Section \ref{S:bnotation}, the vectors 
$\pns{x_1\pa_{x_1},\ldots,x_n\pa_{x_n}}$ are invariantly defined up to
reordering with respect to diffeomorphisms (indeed, any
diffeomorphism of $\R^n_+$ to itself must take $0$ to $0$), so $\sigma_{\R^n_+}$
may be identified with the inward pointing lattice points
\[
	\sigma_{\R^n_+} \cong \Zp\pair{x_1\pa_{x_1},\ldots,x_n\pa_{x_n}}\ \text{in}\ \bN\set{0}.
\]
We will identify $\sigma_{\R^n_+}$ with this image in $\bN\set{0}$ from now on.

To any smooth refinement $R$ of
$\sigma_{\R^n_+}$ we proceed to associate a {\em generalized blow-up} of
$\R^n_+,$ which will be denoted, with its blow-down map
\[
	[\R^n_+; R] \stackrel{\beta}{\to} \R^n_+.
\]
The blow-down map $\beta : [\R^n_+; R] \to \R^n_+$ is an interior b-map
from this new manifold with corners, which is proper, surjective and
restricts to a diffeomorphism on the interior of its domain to the interior
of $\bbR^n_+.$ We construct $[\R^n_+; R]$ and $\beta$ from explicit
coordinate patches and transition diffeomorphisms.

For brevity we will say that a monoid $\sigma \in R$ of maximal dimension $n$
is \emph{maximal}. To any such maximal monoid $\sigma =
\Zp \pair{v_1,\ldots,v_n},$ we associate a copy of the model space $U_\sigma
=\bbR^n_+$ with coordinates $t = (t_1,\ldots,t_n),$ where the order of the
coordinates is associated to the order of the vectors.  Let $\nu \in\GL(n,\Q)$
be the matrix whose rows are the coordinates of the vectors $v_i$, so that
$\nu$ has entries $\nu_{ij}$, where 
\[
	v_i=\sum\limits_{j=1}^n\nu_{ij}x_j\pa_{x_j},
\]
(using the realization of $\sigma_{\R^n_+}$ in $\bN \set{0}$).  Since $R$ is a
refinement of $\sigma_{\R^n_+}$, $\nu$ has non-negative integral entries:
\[
	\nu \in \Mat(n\times n, \Zp ) \cap \GL(n,\Q).
\]
Then consider the smooth map
\[
	\beta_{\sigma} : U_\sigma \ni t \mapsto t^\nu = x \in \bbR^n_+ 
\]
where the image space is the fixed model manifold, and we use the 
notational convention established in Section \ref{S:bnotation}. This is an interior
b-map (which will be the local coordinate version of the blow-down map), which is a diffeomorphism
of the interiors of its domain and range, under which 
\[
\begin{gathered}
	\bd (\beta_{\sigma})_* : \bN \set{0} \subset \bT_{0} U_\sigma 
	\to \bN \set{0} \subset \bT_0 \R^n_+ \\
	t_i\pa_{t_i} \mapsto \sum\limits_{j}\nu_{ij}x_j\pa_{x_j} \cong v_i 
\end{gathered}
\]
In addition $\bd \beta_\sigma : \bT_p U_\sigma \to \bT_{\beta_\sigma(p)}\R^n_+$
is an isomorphism for all $p.$

To each face $\tau \leq \sigma$ given by a collection of generators,  $\tau =
\Zp \pair{v_{i}}_{i \in I},$ $I \subset \set{1,\ldots,n}$ we associate the
(relatively) open set 
\begin{equation}
	U_{\sigma,\tau} = 
	\set{(t_1,\ldots,t_n) \in U_\sigma \;;\; 
	t_j \neq 0 \text{ if } j \notin I} \subset U_\sigma.
	\label{E:relopen_subchart}
\end{equation}
Thus, $U_{\sigma,\tau} \cong \bbR^k_+\times (0,\infty)^{n-k} \subset \R^n_+,$
where the coordinates allowed to take the value $0$ are the $t_{i_1},$
$\ldots,$ $t_{i_k},$ corresponding to those generators of $\sigma$ which are
also generators of $\tau.$

\begin{prop}[Generalized blow-up of $\R^n_+$] 
For any two maximal monoids $\sigma_1$ and $\sigma_2$ in a smooth refinement
$R$ of $\sigma_{\R^n_+},$ with common face $\tau=\sigma_1\cap\sigma_2,$ the diffeomorphism of
the interiors of the $U_{\sigma_i}$ given by
$\beta_{\sigma_2}^{-1}\beta_{\sigma_1}$ extends by 
continuity to a diffeomorphism  $\chi_{12}:U_{\sigma_1,\tau}\longrightarrow
U_{\sigma_2,\tau}$ and the quotient space
\begin{equation}
	[\R^n_+;R]=\Big(\bigsqcup U_{\sigma}\Big)\big/\{\chi_{*}\}
\label{E:localblow_quotient}
\end{equation}
is a manifold with corners, with coordinate charts the $U_{\sigma};$ it is
equipped with a well-defined blow-down map $\beta:[\R^n_+; R]\to\R^n_+$ given
by $\beta_{\sigma}$ on each $U_{\sigma}.$
\label{P:localblow} 
\end{prop}

\begin{proof} 
The space $U_{\sigma},$ with its local blow-down map $\beta_{\sigma},$ really
depends on the choice of the order of the generators implicit in the definition
of the map. However, the coordinates $t_j$ are each naturally associated to the
corresponding generator and change of order of the generators corresponds to
the same reordering of the coordinates in $U_{\sigma}.$ Thus we can freely
reorder the generators.

For two maximal monoids, $\sigma_1$ and $\sigma_2,$ $\tau=\sigma_1\cap\sigma_2$
is generated by the common generators of $\sigma_1$ and $\sigma_2$. For
convenience of notation we can assume that these are the first $k$ generators:
\[
\begin{gathered}
\tau = \Zp \pair{v_1,\ldots,v_k}=\sigma_1\cap\sigma_2,\\
\sigma_1 = \Zp \pair{v_1,\ldots,v_k,v_{k+1}, \ldots,v_n},\
\sigma_2 = \Zp \pair{v_1,\ldots,v_k,v'_{k+1},\ldots,v'_n}.
\end{gathered}
\]
Since the $\sigma_i$ are smooth, their generators form bases so for $i > k$,
\[
v_{i} = \sum_{j=1}^k b_{ij} v_j + \sum_{j={k+1}}^n c_{ij} v'_j,\ b_{ij},\ 
c_{ij} \in \Q,
\]
where the coefficients are rational since both bases are integral with respect
to $\set{x_1\pa_{x_1},\ldots,x_n\pa_{x_n}}$.

Now, the map $\chi_{12} = \beta_{\sigma_2}^{-1}\beta_{\sigma_1}$ between the 
interiors of the two cones is given by $t \mapsto t^\mu = t'$, 
where 
\[
	\mu = \begin{pmatrix} \id & 0\\b & c\end{pmatrix}\quad [b]_{ij} = b_{ij},\ [c]_{ij} = c_{ij}.
\]
In other words, 
\[
\begin{gathered}
\chi_{12}\pns{t_1,\ldots,t_n}=\pns{t'_1,\ldots,t'_n}\in\bbR^n_+=U_{\sigma_2},\ \text{where}\\
t'_i=t_i\prod_{j>k}t_j^{b_{ji}},\ i\le k,\quad
t'_{i}=\prod_{j>k}t_j^{c_{ji}},\ i>k.
\end{gathered}
\]
Since, in $U_{\sigma_1,\tau}$, $t_j>0$ for $j>k,$ $\chi_{12}$ extends smoothly 
from the interior to $U_{\sigma_1,\tau}.$ Since the inverse, $\chi_{21}(t')=(t')^{\nu^{-1}},$ has
a similar form, $\chi_{12}$ is a diffeomorphism onto $U_{\sigma_2,\tau}.$

Next we show that the quotient \eqref{E:localblow_quotient} is Hausdorff.
Consider points $p_i\in U_{\sigma_i}$, $i = 1,2$ which are not
identified by $\chi_{12}.$ If $p_1\in U_{\sigma_1,\tau}$ then its image lies
in $U_{\sigma_2,\tau}$ and a sufficiently small neighborhood of $p_2$ does
not contain $\chi_{12}(p_1)$ in its closure, so the points have
neighborhoods in the $U_{\sigma_i}$ with no $\chi_{12}$ related
points. Using the inverse the same is true if $p_2\in U_{\sigma_2,\tau}$ so
we may assume both are in the complements of the respective
$U_{\sigma_i,\tau}.$ 

Since $\sigma_1\cap \sigma_2 = \tau,$ there is a separating
hyperplane, $u^\perp$ for $u \in \bN^\ast\set{0}$ so that
\[
\begin{gathered}
	\pair{u,v_i}  = 0, \ i = 1,\ldots, k\\
	\pair{u,v_i}  > 0,\ \pair{u,v'_i}  < 0, \  i = k+1,\ldots, n.
\end{gathered}
\]
Let $w$ and $w'$ be the coordinates of $u$ with respect to the bases dual to
the generators of $\sigma_1$ and $\sigma_2$, respectively.  Thus $w_i = w'_i =
\pair{u,v_i} =0$ for $i = 1,\ldots,k$ and $w_i = \pair{u,v_i} > 0$, $w'_i =
\pair{u,v'_i} < 0$ for $i > k$.

Then $\chi_{12}^\ast (t')^{w'} = t^w$, and $p_1$ and $p_2$ can be separated by
the explicit open sets $\big\{ t^w < \epsilon\big\} \subset U_{\sigma_1}$ and
$\big\{ (t')^{-w'}\big\} \subset U_{\sigma_2}$ since
\[
	\chi_{12}\pns{\set{t^w < \epsilon} \cap U_{\sigma_1,\tau}}
	= \big\{(t')^{w'} < \epsilon\} \cap U_{\sigma_2,\tau}
	= \big\{(t')^{-w'} > 1/\epsilon\} \cap U_{\sigma_2,\tau},
\]
so these sets do not meet in the quotient by $\chi_{12}$.

Thus $[\R^n_+; R]$ defined by the quotient \eqref{E:localblow_quotient}  has the
structure of a Hausdorff, paracompact smooth manifold with corners arising
from the covering by the coordinate charts which are the images in $[\R^n_+,R]$
of the $U_{\sigma}.$ Smooth functions on $[\R^n_+;R]$ are those which are
smooth in each of the coordinate patches and the blow-down map 
\[
	\beta:[\R^n_+;R]\to \R^n_+
\]
is a well defined smooth b-map since it is such on each coordinate patch
and these maps form, by construction, a commutative diagram with the
transition maps.
\end{proof}

\noindent From the point of view of algebraic geometry, the chart $U_\sigma$
consists of the `$\R_+$-points' of the dual monoid $\widehat \sigma$, given by
$\Hom_{\mon}\pns{\widehat \sigma, \R_+}$ (where $\R_+$ is considered as a
multiplicative monoid), though our construction equips this set with a much
stronger topology and smooth structure. In this sense $[\R^n_+; R]$ is related
to the ``singular manifold with corners'' \cite{fulton1993introduction} of a
toric variety, though our setup is more closely related to toroidal embeddings
\cite{kempf1973toroidal} and logarithmic geometry \cite{kato1994toric},
\cite{ogus2006lectures}. It is important to point out however that the gluings
between the affine charts $U_\sigma$ above are extremely ill-behaved over $\R$
or $\C$ (as opposed to $\R_+$ here) owing to the presence of radicals (i.e.
fractional powers of the coordinates), so what results in a smooth manifold
with corners above results in a much more singular object (like a stack) in the
conventional algebraic setting. This is the ultimate reason for our nonstandard
definition of monoid refinement in the previous section and for the greater
flexibility of our methods.

Next we observe that the boundary faces of $[\R^n_+;R]$ are in bijection with
the monoids in the refinement $R.$ If $\tau \in R$ and $\sigma \geq \tau$ is a
maximal monoid of which it is a face then $\tau$ defines the boundary face
$F_{\sigma,\tau}\subset U_{\sigma}=\R^n_+$ which is the closure of
$U_{\sigma,\tau}\setminus(0,\infty)^n,$ the part of the boundary where the
coordinates corresponding to the generators of $\tau$ vanish (and the other
coordinates may or may not vanish).

\begin{prop}
\label{P:local_faces_cones} 
The manifold $[\R^n_+; R],$ determined by a smooth resolution of
$\sigma_{\R^n_+},$ has interior diffeomorphic to $(0,\infty)^n$ and its
boundary faces are in 1-1 correspondence with the monoids $\tau \in R$ where the
face $F_\tau$ corresponding to $\tau$ is the quotient in
\eqref{E:localblow_quotient} of the union of the $F_{\sigma,\tau}\subset
U_{\sigma}$ for the maximal monoids $\sigma$ with $\tau \leq \sigma;$ this
correspondence satisfies $\codim(F_\tau) = \dim(\tau)$ and is
inclusion-reversing:
\[
	F_{\tau'} \subseteq F_\tau \iff \tau \leq \tau'.
\]
\end{prop}

\begin{proof} 
The boundary faces of each $U_\sigma=\R^n_+$ are, as follows from
\eqref{E:relopen_subchart}, precisely the $F_{\sigma ,\tau}$ for all faces
$\tau\leq \sigma,$ since these correspond to all subsets of the coordinate
functions.  Under the transition maps $\chi_{12}$ defined above and
corresponding to two monoids, $\sigma_1$ and $\sigma_2,$ of maximal dimension,
the interior of each $F_{\sigma_1,\tau}$ is mapped onto the interior of
$F_{\sigma_2,\tau},$ if $\tau\leq\sigma_1\cap\sigma_2.$ Otherwise if $\tau$ is
not contained in either $\sigma_1$ or $\sigma_2$ then $F_{\sigma_1,\tau}$ does
not meet the domain of definition of $\chi_{12}$ or does not meet the range.

Thus the interiors of the $F_{\sigma,\tau}$ are globally well-defined subsets
of $[\R^n_+;R].$ For each $\tau\in R,$ the closure in $[\R^n_+; R]$ of the
interior of $F_{\sigma,\tau}$ is the image of the union of the closures in the
$U_{\sigma}$ that it meets.  It is therefore everywhere locally a manifold with
corners, and hence is globally a manifold with corners. In particular each
boundary face is embedded in $[\R^n_+;R].$  Finally, the fact that $\tau \leq
\tau' \iff F_{\tau'}\subseteq F_\tau$ is evident in $U_{\sigma}$ for a maximal
monoid $\sigma \geq \tau' \geq \tau$.
\end{proof}

Now we establish a local version of a result we shall prove in more generality in 
Section \ref{S:global} about lifting b-maps to a generalized blow-up. Suppose
\[
	f : O \subset \R^m_+ \to \R^n_+
\]
is an interior b-map with domain an open set $O \subset \R^m_+$. Without loss of generality
we may assume that $0 \in O$ and that $f(0) = 0.$ Then $f$ has the explicit
coordinate expression
\[
	f : x \mapsto x' = a(x)\,x^\delta = \Bigg(a_1(x) \prod_{j=1}^m x_j^{\delta_{j1}}, \ldots,a_n(x) \prod_{j=1}^m x_j^{\delta_{jn}}\Bigg)
\]
where $\delta \in \Mat(m\times n, \Zp )$ and $0 < a_i(x) \in C^\infty(\R^m_+)$.
Because of the non-negativity and integrality of the $\delta_{ij}$, the
b-differential $\bd f_\ast : \bN_0\set{0}\to \bN_0 \set{0}$ (which is represented by
the matrix $\delta^\transpose$ with respect to the bases $\set{x_i\pa_{x_i}}$ and $\set{x'_i \pa_{x'_i}}$)
restricts to a monoid homomorphism
\[
	\bd f_\ast = \delta^\transpose : \sigma_{\R^m_+} \to \sigma_{\R^n_+}.
\]
(This is a preliminary version of the morphism of monoidal complexes associated
to a b-map in Section \ref{S:global}.)

\begin{prop}
If $f$ is an interior b-map as above such that $\bd f_\ast :
\sigma_{\R^m_+} \to \sigma_{\R^n_+}$ factors through a monoid homomorphism
$\phi : \sigma_{\R^m_+} \to \tau$ for some $\tau \in R$ where $R$ is a smooth
refinement of $\sigma_{\R^n_+}$ then there exists a unique b-map $f' : O \subset\R^m_+
\to [\R^n_+; R]$ such that
\[
\begin{tikzpicture}[->,>=to,auto]
\matrix (m) [matrix of math nodes, column sep=1cm, row sep=1cm,  text depth=0.25ex]
{  & {[\R^n_+; R]} \\ O\subset\R^m_+ & \R^n_+ \\};
\path (m-1-2) edge node  {$\beta$} (m-2-2); 
\path (m-2-1) edge node {$f$} (m-2-2); 
\path (m-2-1) edge node {$f'$} (m-1-2); 
\end{tikzpicture}
\]
commutes and the range of $f'$ is contained in the coordinate chart $U_\tau
\subset [\R^n_+; R].$ Moreover this construction is functorial in that for
a b-map $g : U \subset \R^l_+ \to O \subset\R^m_+,$ $\bd \pns{f \circ g}_\ast$ factors through
$\tau \in R$ if and only if $\bd f_\ast$ does, and then $(f \circ g)' = f' \circ g.$
\label{P:local_lifting}
\end{prop}

\begin{proof}
Taking a maximal dimension monoid containing $\tau$ if necessary, it suffices
to assume that $\dim(\tau) = n$.  Then let $\pns{U_\tau, t =
(t_1,\ldots,t_n)}$ be the coordinate chart associated to $\tau =
\Zp \pair{v_1,\ldots,v_n} \in R$, and let $\nu \in \GL(n,\Zp )$ denote the matrix
such that $v_i = \sum_j \nu_{ij} x'_j\pa_{x'_j}$ as in the construction of
generalized blow-up, so that the blow-down acts by $\beta : t \mapsto x =
t^\nu$ on $U_\tau$.  Alternatively, one can view $\nu^\transpose$ as the matrix
defining the monoid inclusion $\nu^\transpose : \tau \hookrightarrow
\sigma_{\R^n_+}$ with respect to the bases $\set{v_i}$ and
$\set{x'_i\pa_{x'_i}}$.

In a similar manner, let $\mu \in \Mat(m\times n, \Zp )$ denote the matrix whose transpose
represents the given homomorphism
\[
	\phi = \mu^\transpose : \sigma_{\R^m_+} \to \tau = \Zp \pair{v_1,\ldots,v_n}
\]
with respect to the bases $\set{x_i\pa_{x_i}}$ and $\set{v_i}$.  From the
assumption that $\bd f_\ast$ factors through $\phi$ it follows that
\[
	\delta^\transpose = \nu^\transpose \mu^\transpose = (\mu \nu)^\transpose.
\]
where $\delta$ is the matrix of coefficients of $\bd f_\ast$ as above.

Define $f' : \R^m_+ \to U_\tau$ in coordinates by 
\[
	f' : x \mapsto a(x)^{\nu^{-1}}x^{\mu} = t.
\]
Then observe that $\beta \circ f'$ has the form
\[
	\beta \circ f' : x \mapsto \big(a(x)^{\nu^{-1}}x^\mu\big)^\nu = a(x) x^{\mu\nu} = a(x) x^\delta,
\]
and is therefore equal to $f.$  The form of any $f' : \R^m_+ \to U_\tau$
such that $\beta \circ f' = f$ is determined in these local coordinates,
giving uniqueness.

To show functoriality, suppose that $g : \R^l_+ \to \R^m_+$ acts by $z
\mapsto b(z){z}^\kappa = x.$ We have the commutative diagram
\[
\begin{tikzpicture}[->,>=to,auto]
\matrix (m) [matrix of math nodes, column sep=1cm, row sep=1cm, text depth=0.25ex]
{& & \tau \\ \sigma_{\R^l_+} & \sigma_{\R^m_+} & \sigma_{\R^n_+}. \\};
\path (m-2-1) edge node {$\kappa^\transpose$} (m-2-2);
\path (m-2-2) edge node {$\delta^\transpose$} (m-2-3);
\path (m-2-2) edge node {$\mu^\transpose$} (m-1-3);
\path (m-1-3) edge node {$\nu^\transpose$} (m-2-3);
\end{tikzpicture}
\]
The composition $f \circ g$ acts by
\[
	f \circ g : z \mapsto a\big(b(z){z}^\kappa\big) b^\delta {z}^{\kappa\delta} =: c(z) {z}^{\kappa \delta}.
\]
Then on one hand $\pns{f \circ g}'$ has the form
\[
	\pns{f \circ g}' : z \mapsto c(z)^{\nu^{-1}} {z}^{\kappa \mu} = a\big(b(z){z}^{\kappa}\big)^{\nu^{-1}} b^{\delta \nu^{-1}} {z}^{\kappa \mu},
\]
while on the other hand $f' \circ g$ acts by
\[
	f' \circ g : z \mapsto a\big(b(z){z}^\kappa\big)^{\nu^{-1}} \pns{b {z}^{\kappa}}^{\mu},
\]
which is the same since $\delta \nu^{-1} = \mu$.
\end{proof}

This result has an important corollary which is fundamental to the global 
construction of generalized blow-up in Section \ref{S:global}.

\begin{cor}
Any diffeomorphism of open submanifolds of $\R^n_+,$ $f : O_1 \subset
\R^n_+ \stackrel{\cong}{\to} O_2 \subset \R^n_+$ which maps boundary hypersurfaces onto themselves 
lifts to a unique diffeomorphism,
\[
	f' : O'_1 \subset [\R^n_+; R] \to O'_2 \subset [\R^n_+; R]
\]
where $O'_i = \beta^{-1}(O_i)$, such that the diagram
\[
\begin{tikzpicture}[->,>=to,auto]
\matrix (m) [matrix of math nodes, column sep=1cm, text depth = 0.25ex, row sep=1cm]
{ O'_1 & O'_2 \\ O_1 & O_2 \\};
\path (m-1-1) edge node {$f'$}    (m-1-2);
\path (m-1-2) edge node {$\beta$} (m-2-2);
\path (m-2-1) edge node {$f$}     (m-2-2);
\path (m-1-1) edge node {$\beta$} (m-2-1);
\end{tikzpicture}
\]
commutes and then $f'$ also maps boundary hypersurfaces onto themselves.
\label{C:diffeos_lift}
\end{cor}
\begin{proof}
The lift $f'$ will be defined locally on coordinate patches, so let $\tau \in R$ be a maximal
monoid and consider the b-map
\[
	f \circ \beta : O_1' \cap U_\tau \subset \R^n_+ \to \R^n_+.
\]
It is evident that $\bd (f \circ \beta)_\ast$ sends $\tau$ (identified here
with $\sigma_{\R^n_+}$ where $\R^n_+ = U_\tau$) isomorphically to $\tau$ and so
it follows from Proposition~\ref{P:local_lifting} that there is a unique
\begin{equation}
	f_\tau' = (f \circ \beta)' : O_1' \cap U_\tau \to U_\tau 
	\label{E:f_tau_prime}
\end{equation}
such that $\beta \circ f_\tau' = f \circ \beta.$ Proceeding this way for each maximal
$\tau \in R$, we obtain b-maps \eqref{E:f_tau_prime} defined locally on the covering
$\set{O'_1 \cap U_\tau}$ of $O'_1$, and these maps patch together to give a map
$f' : O'_1 \to [\R^n_+; R]$ by uniqueness, with range contained in $O'_2$. 

Before proving that $f'$ is a diffeomorphism, consider the functoriality of
these lifts. Thus if $g : O_2 \subset \R^n_+ \to O_3\subset \R^n_+$ is another
such diffeomorphism, it follows by the uniqueness and functoriality results in
Proposition~\ref{P:local_lifting} that
\[
	(g\circ f)' = (g \circ f \circ \beta)' = (g \circ \beta \circ f')' = g' \circ f' : O'_1 \to O'_3.
\]
It is also evident from the proof of Proposition~\ref{P:local_lifting} that
$\id' = \id$ and it then follows that $f'$ is a diffeomorphism with inverse $(f^{-1})'.$
\end{proof}

The following is also a straightforward consequence of Proposition~\ref{P:local_lifting}.

\begin{cor} \label{C:boundary_fans}
If $R(\tau)$ is the refinement of $\tau = \Zp \pair{x_i\pa_{x_i}}_{i \in I}
\leq \sigma_{\R^n_+}$ obtained from a smooth refinement $R$ of $\sigma_{\R^n_+}$,
where $I \subset \set{1,\ldots,n}$, and $\dim(\tau) = \#I = k,$ there is
an injection
\begin{equation}
	[\R^k_+; R(\tau)]\times (0,\infty)^{n-k} \hookrightarrow [\R^n_+; R]
	\label{E:boundary_subfan_blowup_inclusion}
\end{equation}
giving a commutative diagram
\[
\begin{tikzpicture}[->,>=to,auto]
\matrix (m) [matrix of math nodes, column sep=1cm, row sep=1cm] { 
[\R^k_+; R(\tau)] \times (0,\infty)^{n-k} & {[\R^n_+; R] }
\\ \R^k_+ \times (0,\infty)^{n-k} & \R^n_+ \\ };
\draw[right hook->] (m-1-1) to (m-1-2); 
\path (m-1-2) edge node {$\beta$} (m-2-2) ; 
\draw[right hook->] (m-2-1) to (m-2-2); 
\path (m-1-1) edge node {$\beta$} (m-2-1) ; 
\end{tikzpicture}
\]
where the bottom map is the inclusion 
\begin{equation}
	\R^k_+ \times (0,\infty)^{n-k} \cong \set{x_i \neq 0 \;;\; i \notin I} \subset \R^n_+.
	\label{E:boundary_subfan_inclusion}
\end{equation}
\end{cor}
\begin{proof}
The proof is similar to the previous one: the composition of the blow down and
inclusion from $[\R^k_+; R(\tau)]\times (0,\infty)^{n-k}$ to $\R^n_+$,
considered locally on charts in the domain, sends monoids in $R(\tau)$ to the
same monoids in $R$, and so factors uniquely (and injectively) through
$[\R^n_+; R]$.
\end{proof}

Finally we note the following versions of Proposition~\ref{P:local_lifting} and
Corollary~\ref{C:diffeos_lift} involving additional factors without boundary.
The proofs differ only in notation, where the additional factors are simply
carried along.

\begin{prop}
If $f : O \subset \R^m_+ \times \R^{m'} \to \R^n_+ \times \R^{n'}$ is an interior b-map 
from an open set containing $0$ such that $\bd f_\ast :
\sigma_{\R^m_+} \to \sigma_{\R^n_+}$ factors through a monoid homomorphism
$\phi : \sigma_{\R^m_+} \to \tau$ for some $\tau \in R$ where $R$ is a smooth
refinement of $\sigma_{\R^n_+}$, then there exists a unique b-map $f'$ lifting $f$
such that
\[
\begin{tikzpicture}[->,>=to,auto]
\matrix (m) [matrix of math nodes, column sep=1cm, row sep=1cm,  text depth=0.25ex]
{  & {[\R^n_+; R]\times \R^{n'}} \\ O\subset\R^m_+\times \R^{m'} & \R^n_+\times \R^{n'} \\};
\path (m-1-2) edge node  {$\beta$} (m-2-2); 
\path (m-2-1) edge node {$f$} (m-2-2); 
\path (m-2-1) edge node {$f'$} (m-1-2); 
\end{tikzpicture}
\]
commutes, and these lifts are functorial.
\label{P:local_lifting_wfactors}
\end{prop}

\begin{cor}
Any diffeomorphism of open submanifolds $f : O_1 \subset \R^n_+\times \R^{n'}
\stackrel{\cong}{\to} O_2 \subset \R^n_+\times \R^{n'}$ which maps boundary
hypersurfaces onto themselves lifts to a unique diffeomorphism
\[
	f' : O'_1 \subset [\R^n_+; R]\times \R^{n'} \to O'_2 \subset [\R^n_+; R]\times \R^{n'},
\]
where $O'_i = (\beta\times \id)^{-1}(O_i)$, and then $(\beta\times \id)\circ f' = f
\circ (\beta \times \id)$.
\label{C:diffeos_lift_wfactors}
\end{cor}

\section{Monoidal complexes}\label{S:complexes}

In Section~\ref{S:global} we show how to associate a natural monoid to each
boundary face of a manifold with corners, with the compatibility condition that
whenever $F \subseteq G$, the monoid associated to $G$ is a face of the one
associated to $F$, i.e. that face relations persist in an inclusion-reversing
sense.  We have already seen such an example, which is the association between
monoids in $R$ and faces in $[\R^n_+; R]$, where $R$ is a refinement of
$\sigma_{\R^n_+}.$ The notion of a refinement (where all the monoids reside in
the same vector space) is not sufficiently flexible to capture the ways in
which boundary faces may be related for general $X$ so we introduce the concept
of a {\em monoidal complex}; this should be thought of by analogy to a
simplicial or CW complex. It is a generalization of the structure of a
refinement in which the monoids are still ``attached together nicely along
faces'', but in which the monoids reside in separate vector spaces and there is
no `base' monoid (as when $R$ is a refinement ``of $\sigma$'').  The contents
of this section are all that is needed for the discussion of the generalized
blow-up of a manifold in Section~\ref{S:global}, apart from
Theorem~\ref{T:blowing_up_domain} which uses the fiber product of monoidal 
complexes discussed in the next section.

Consider a finite partially ordered set (poset) by $A = (A,\leq).$ It is convenient
to think of $A$ as a {\em category}, with objects the elements $a \in A$
and morphisms given by the order relation: $a \smallto b \in \Hom(a,b) \iff
a \leq b.$

\begin{defn} 
Let $(A,\leq)$ be a poset. A {\em monoidal complex} $\cQ$ over $A$ is a
covariant functor from $A$ to the category of monoids, where all morphisms are
isomorphisms onto faces, so $\cQ$ consists of a monoid $\sigma_a$ for each $a
\in A,$ along with isomorphisms called {\em face maps}
\begin{equation}
	i_{a b} : \sigma_a \stackrel{\cong}{\to} \tau \leq \sigma_b\ \text{whenever}\  a \leq b,
	\label{E:face_maps}
\end{equation}
for some $\tau \leq \sigma_b.$ This relation will be denoted
\begin{equation}
	\sigma_a \leq \sigma_b,\ \text{for}\ a \leq b.
	\label{E:complex_order}
\end{equation}

We say $\cQ$ is {\em complete}, respectively {\em reduced}, if for every $b
\in A$ and every face $\tau \leq \sigma_b$, there exists at least one (resp.\
at most one) $a \in A$ such that $a \leq b$ and $i_{ab} : \sigma_a \cong \tau$.
\label{D:monoidal_complex}
\end{defn}

\begin{lem}
If $\cQ$ is complete and reduced, then the posets $(A,\leq)$ and
$\Gamma_{\cQ} = \set{\tau \leq \sigma \;;\; \sigma \in \cQ} \big / i_\ast$
with the order \eqref{E:complex_order} are isomorphic.
\label{L:complete_complex_directed}
\end{lem}
\begin{proof}
Consider the composition of the map $A \ni a \mapsto \sigma_a \in \cQ$ with
the quotient by the face maps \eqref{E:face_maps}. This map $A \to \Gamma_{\cQ}$ 
intertwines the orders by \eqref{E:complex_order}; it must be
injective by the assumption that $\cQ$ is reduced, and surjective by the
assumption that it is complete.
\end{proof}
\noindent 
From this it follows that for a complete, reduced monoidal complex, $A$ is
entirely determined by the set of $\sigma \in \cQ$ and the face relations
\eqref{E:complex_order}.  Nevertheless, it is often convenient to be able to refer
explicitly to an indexing set $A$, as in the case $(A,\leq) =
(\cM(X),\leq)$ below, which we will most often encounter.

There are evidently some obstructions for a general poset to index
a complete, reduced monoidal complex, but we shall not concern ourselves 
with them here; all complete reduced monoidal complexes will arise naturally.

From now on, {\em monoidal complex} will mean {\em complete and reduced
monoidal complex}.

\begin{ex} 
The faces of a single monoid $\sigma$ form a monoidal
complex over the poset $\set{\tau \;;\; \tau \leq \sigma}$.

More generally, any refinement $R$ of $\sigma$ forms a monoidal complex over
the set $R$ with the order coming from the face relations.
\end{ex}

For an example which is not a refinement, consider the following.
\begin{ex}
Let $\sigma_0 = \set{0}$, $\sigma_a = \Zp (1,0) \subset \R^2$,
$\sigma_b = \Zp (0,1) \subset \R^2$, and let $\sigma_c$ and $\sigma_d$
be two distinct copies of the monoid $\Zp \pair{(1,0),(0,1)} \subset \R^2$.  Let the
face maps be the obvious ones in the diagram
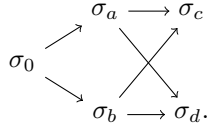
\[
\begin{tikzpicture}[->,>=to,auto]
\matrix (m) [matrix of math nodes, column sep=0.5cm, row sep=0.25cm]
{ &\sigma_a & \sigma_c \\ \sigma_0 & & \\ & \sigma_b & \sigma_d. \\ };
\draw (m-1-2) to (m-1-3);
\draw (m-3-2) to (m-3-3);
\draw (m-2-1) to (m-1-2);
\draw (m-2-1) to (m-3-2);
\draw (m-1-2) to (m-3-3);
\draw (m-3-2) to (m-1-3);
\end{tikzpicture}
\]
This complex cannot be realized as a refinement, since otherwise $\sigma_c$ and
$\sigma_d$ would have to be identified; here they remain distinct. 
The poset underlying this complex is the same as the one given by the
boundary faces (with the order of reverse inclusion) of the 2-dimensional manifold with corners
pictured in Fig.~\ref{F:football}.
\end{ex}
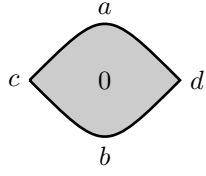
\begin{figure}[tb]
\[
\begin{tikzpicture}[fill opacity=0.2, line width=1pt]
\filldraw (0,0) node[left,opacity=1] {$c$} .. controls (1,1) .. node[midway,opacity=1,above] {$a$} (2,0) 
	node[right,opacity=1] {$d$} .. controls (1,-1) .. node[midway,opacity=1,below] {$b$} (0,0);
\node[fill opacity=1] at (1,0)  {$0$};
\end{tikzpicture}
\]
\caption{A manifold whose boundary face relations do not index a monoid refinement.}
\label{F:football}
\end{figure}

\begin{defn}
A {\em morphism} of monoidal complexes $\phi : \cQ_A \to \cQ_B$ consists of
a map of posets $\phi_\# : (A,\leq) \to (B,\leq)$ and monoid homomorphisms
$\phi_{ab} : \cQ_A \ni \sigma_a \to \sigma_b \in \cQ_B$ for all $a \in A$,
where $b = \phi_\#(a) \in B$.  These are required to commute with the homomorphisms 
$i_{aa'}$ for $a \leq a'$, so that
\[
\begin{tikzpicture}[->,>=to,auto]
\matrix (m) [matrix of math nodes, column sep=1cm, row sep=1cm]
{ \sigma_a & \sigma_b \\ \sigma_{a'} & \sigma_{b'} \\};
\path (m-1-1) edge node {$\phi_{ab}$} (m-1-2);
\path (m-1-2) edge node {$i_{bb'}$} (m-2-2);
\path (m-1-2) edge node {$i_{bb'}$} (m-2-2);
\path (m-2-1) edge node {$\phi_{a'b'}$} (m-2-2);
\path (m-1-1) edge node {$i_{aa'}$} (m-2-1);

\end{tikzpicture}
\]
commutes, where $a \leq a'$, $\phi_\#(a) = b$, $\phi_\#(a') = b'$, and thus $b
\leq b'$.  We say $\phi$ is {\em injective} if all the morphisms $\phi_{ab}$
are injective, though we do not necessarily require that $\phi_\# : A \to
B$ be injective.
\label{D:morphism_of_monoidal_complex}
\end{defn}

An elementary example of a morphism is the inclusion of a subcomplex:

\begin{defn} Let $\cQ$ be a monoidal complex over $(A,\leq)$.  A {\em
    monoidal subcomplex} of $\cQ$  
is a complex $\cQ_0$ obtained by restricting $\cQ$ to a subset $A_0 \subset A$:
\[
	\cQ_0 = \set{\sigma_a \;;\; a \in A_0}
\]
such that $\cQ_0$ is complete and reduced. There is then an injective morphism
\[
	\cQ_0 \longhookrightarrow \cQ
\]
consisting of the identity homomorphisms over $A_0 \subset A$.
\label{D:monoidal_subcomplex}
\end{defn}

In Section~\ref{S:global}, we will show that a b-map $f : X \to Y$ gives rise to a
morphism $f_\natural : \cP_X \to \cP_Y$ of the basic monoidal complexes associated to
$X$ and $Y$.

Another important class of morphisms consists of refinements.  
\begin{defn} Let
$\cQ$ be a monoidal complex.  A {\em refinement} of $\cQ$ is a morphism $\phi :
\cR \to \cQ$ all of whose homomorphisms are injective and for all $\sigma \in
\cQ$, 
\begin{enumerate}[{\normalfont (i)}] 
\item $\bigcup_{\tau \in \phi_\#^{-1}(\sigma)} \phi\big(\supp(\tau)\big) =
\supp(\sigma)$, and
\label{I:ref_one} 
\item for $\tau_1,\tau_2 \in \phi_\#^{-1}(\sigma)$, $\relint\big[\phi\big(\supp(\tau_1)\big)\big] \cap
\relint\big[\phi\big(\supp(\tau_2)\big)\big] = \emptyset$ unless $\tau_1 =
\tau_2$, 
\label{I:ref_two} 
\end{enumerate} 
where the {\em relative interior} of a cone $C =
\sspan_{\R_+}\set{v_1,\ldots,v_k}$ is the set $\relint[C] =
\sspan_{(0,\infty)}\set{v_1,\ldots,v_k}$.
\label{D:refinement_of_monoidal_complex} 
\end{defn} 
\noindent Condition \eqref{I:ref_one} requires that the support of each monoid in
$\cQ$ is covered by the supports of monoids in $\cR$, and condition
\eqref{I:ref_two} demands that these supports only intersect along mutual
faces. These, along with our assumption that $\cR$ is complete and reduced,
are analogous to the conditions in Definition~\ref{D:monoid_refinement} for the
refinement of a monoid.  Figure \ref{F:refinement} depicts a refinement.

\begin{figure}[tb]
\begin{tikzpicture}
\matrix (m) [matrix of nodes, anchor=base, column sep=2cm, row sep=0.5cm]
{ 
 \begin{tikzpicture}[scale=0.55]
 \draw (0,1) -- (0,5.5);
 \draw (0,1) -- (4.5,5.5);
 \fill (0,1) circle (3pt);
 \fill (0,3) circle (3pt);
 \fill (0,5) circle (3pt);
 \fill (2,3) circle (3pt);
 \fill (2,5) circle (3pt);
 \fill (4,5) circle (3pt);

 \draw (1,0) -- (5.5,0);
 \draw (1,0) -- (5.5,4.5);
 \foreach \x in {0,2,...,4} {
	 \foreach \y in {0,...,\x} {
		 \fill (1,0)+(\x,\y) circle (3pt);
		 }
	 }
  \draw (2.5,3.9) -- (2.7,3.3);
  \draw (2.4,3.8) -- (2.6,3.2);

  \draw (3.7,2.3) -- (3.5,2.9);
  \draw (3.6,2.2) -- (3.4,2.8);

 \end{tikzpicture}
&
 \begin{tikzpicture}[scale=0.6]
 \draw (0,0) -- (0,4.5);
 \draw (0,0) -- (4.5,0);
 \draw[dashed,opacity=0.5] (0,0) -- (4.5,4.5);
 \foreach \x in {0,1,...,4} {
	 \foreach \y in {0,1,...,4} {
		 \fill (\x,\y) circle (2pt);
		 }
	 }
 \fill (0,0) circle (3pt);
 \fill (0,2) circle (3pt);
 \fill (0,4) circle (3pt);
 \fill (2,2) circle (3pt);
 \fill (2,4) circle (3pt);
 \fill (4,4) circle (3pt);
 \foreach \x in {0,2,...,4} {
	 \foreach \y in {0,...,\x} {
		 \fill (\x,\y) circle (3pt);
		 }
	 }
  
  \end{tikzpicture}
\\
$\cR$ & $\cQ$
\\};
\path (m-1-1) edge[->,>=to,bend left=10] (m-1-2);
\end{tikzpicture}
\caption{A refinement of monoidal complexes.  Only maximal dimension monoids are 
pictured.  In this example, $\cR \smallto \cQ$ is injective, though $\cR$ is not 
smooth, and the individual monoid homomorphisms do not cover their targets.}
\label{F:refinement}
\end{figure}
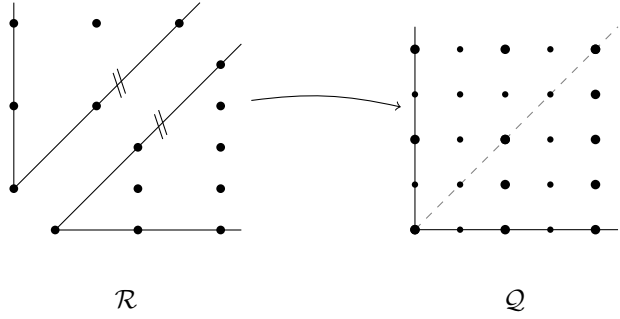

\begin{prop} If $\phi : \cR \to \cQ$ is a refinement then for each $\sigma \in
\cQ$, the collection 
\begin{equation}
	\cR(\sigma) := \set{\phi(\sigma') \;;\; \sigma' \in
		\phi_\#^{-1}(\tau),\ \tau \leq \sigma}
	\label{E:refinement_set}
\end{equation}
is a refinement of $\sigma.$ Conversely, let $\set{R\pns{\sigma} \;;\;
  \sigma \in \cQ}$ be a collection of refinements which are compatible in
that whenever $\tau \leq \sigma$, the refinement $R(\tau)$ is identical to
the induced refinement $\big(R(\sigma)\big)(\tau)$ as in
Lemma~\ref{L:refinements_and_faces}.  Then the quotient $\cR = \bigcup
R(\sigma) \big / i_\ast$ of the set of all monoids in the collection by the
face maps $i_\ast$ of $\cQ$ forms a refinement $\phi : \cR \to \cQ$ of
monoidal complexes.
\label{P:refinement_of_monoidal_complex}
\end{prop}

\begin{proof}
Let $\cR(\sigma)$ be the set in \eqref{E:refinement_set}.  Certainly 
\[
	\cR(\sigma) \supseteq \set{\phi(\sigma') \;;\; \sigma' \in \phi_\#^{-1}(\sigma)}
\]
and it follows that $\cR(\sigma)$ satisfies property
\ref{D:monoid_refinement}.\eqref{I:monref_3} for a refinement of $\sigma$ by
property \eqref{I:ref_one} in Definition~\ref{D:refinement_of_monoidal_complex}
above.

To see that $\cR(\sigma)$ satisfies the other two properties of a refinement,
let $\tau'_i \in \cR$, $i = 1,2$ with $\phi(\tau'_i) \in \cR(\sigma)$, and for
convenience of notation, identify these with their images in $\sigma$.  Let
$\tau \leq \sigma$ be the smallest face such that $\tau'_1 \cap \tau'_2 \subset
\tau$.  Since $\phi$ is a morphism it follows that each $\tau'_i \cap \tau$ is
a face of $\tau'_i$, hence also in $\cR$, and it must be that $\tau'_i \cap
\tau \in \phi_\#^{-1}(\tau)$ (or else there would be a smaller $\tau$), so it
follows that the face $\tau'_i \cap \tau \leq \tau$ is also in $\cR(\sigma)$.

By property \eqref{I:ref_two} above for $\tau'_i\cap \tau \in
\phi_\#^{-1}(\tau)$, it follows that $\tau'_1 \cap \tau'_2 = \pns{\tau'_1 \cap
\tau} \cap \pns{\tau'_2 \cap \tau}$ must be a face of each. To see that $\cR(\sigma)$
contains the faces of all its monoids, simply let $\tau'_1$ be the face in
question of $\tau'_2$; we conclude that $\tau'_1 = \tau'_1 \cap \tau'_2 \in
\cR(\sigma)$.  Thus $\cR(\sigma)$ is a refinement of $\sigma$.

For the converse, let $\cR$ be the set of all monoids in the refinements
$R(\sigma)$, $\sigma \in \cQ$, modulo the identification of monoids in
$(R(\sigma))(\tau)$ with those in  $R(\tau)$ for $\tau \leq \sigma$, as above.
That $\cR$ is a complex follows easily from the fact that each $R(\sigma)$ forms
a monoidal complex.

For notational clarity, denote the posets for $\cR$ and $\cQ$ by
$(A,\leq)$ and $(B,\leq)$, respectively.  We take $\phi_\# : A \to B$ to be the
map which sends $a$ to the smallest $b \in B$ such that $\sigma'_a \in \cR$
is in the refinement $R_{\sigma_b}$, and then $\phi_{ab} :
\sigma'_a \to \sigma_b$ is given by the inclusion $R_{\sigma_b} \ni
\sigma'_a \subset \sigma_b$ which is of course injective.

Finally, it follows directly from the fact that the $R(\sigma)$ are refinements
that $\phi : \cR \to \cQ$ satisfies the properties in
Definition~\ref{D:refinement_of_monoidal_complex}.
\end{proof}

The remainder of this section is devoted to specific algorithms for obtaining
and extending various refinements of complexes. 

A primary means of obtaining refinements is star subdivision; we extend
this to a monoidal complex. The proof of the following
follows directly from Propositions~\ref{P:refinement_of_monoidal_complex} and
\ref{P:star_subdivision}.

\begin{prop}[Star subdivision]
Let $\cQ$ be a monoidal complex, $\sigma_a \in \cQ$, and $v \in \sigma_a$
the monoids
\[
\cS(\cQ,v)(\sigma_b) = \begin{cases} S\big(\sigma_b,(i_{ab}v)\big) & \text{if $\sigma_a \leq \sigma_b$, and} \\ 
		\sigma_b & \text{otherwise}\end{cases}
\]
form a monoidal complex which refines $\cQ;$ if $\cQ$ is smooth, then $\cS(\cQ,v)$ is
smooth.
\label{P:star_subdivn_complex}
\end{prop}

Recall that in Section~\ref{S:monoids} an extension of star subdivision is
discussed, giving a refinement $S(\sigma,\mu)$ of a smooth monoid,
$\sigma,$ with respect a monoid $\mu$ given by the intersection of $\sigma$
with an integral subspace. We now extend this to monoidal complexes, which
will be of use in Section \ref{S:binvarres}.

\begin{prop}[Planar refinement]\label{P:planar_refinement_complex}
Let $\cP$ a smooth monoidal complex and suppose $i : \cQ \to
\cP$ is an injective morphism such that for all $\sigma \in \cP$,
$(i_\#)^{-1}(\sigma) \subset \cQ$ contains at most one monoid $\mu,$ and
$i(\mu) \subset \sigma$ is the intersection of $\sigma$ with an integral
subspace, then
\[
	\cS(\cP,\cQ)(\sigma) = S(\sigma,\mu), \quad \mu = (i_\#)^{-1}(\sigma)
\]
form a complex, containing $\cQ$ as a subcomplex, which refines $\cP.$
\end{prop}
\begin{proof}
This follows directly from Proposition \ref{P:planar_refinement}, using the fact
that $S(\tau, \mu \cap \tau) = S(\sigma, \mu)(\tau).$
\end{proof}

The smoothing of a simplicial monoid also extends to monoidal complexes.

\begin{prop}[Smoothing]
Let $\cQ$ be a simplicial complex.  Then the collection of monoid refinements
$\sigma_\mathrm{sm} \to \sigma$, where $\sigma_\mathrm{sm} = \Zp \pair{v_1,\ldots,v_n}$
is freely generated by the extremals of $\sigma$ as in Example \ref{X:smoothing},
forms a smooth refinement $\cQ_\mathrm{sm} \to \cQ$, in which the indexing posets
are the same.
\label{P:smoothing_complex}
\end{prop}
\begin{proof}
This follows directly from the fact that the faces of $\sigma_\mathrm{sm}$, namely 
the monoids $\Zp \pair{v_i}_{i \in I}$, $I \subset \set{1,\ldots,n}$ are the smoothings of the 
corresponding faces of $\sigma$.
\end{proof}
\section{Fiber products of complexes}\label{S:fib_prod_complexes}
Below we frequently need to address the problem of finding a refinement $\cR_1 \to \cQ_1$
of a complex which is compatible with a given refinement $\cR_2 \to \cQ_2$
with respect to a morphism $\cQ_1 \to \cQ_2$, meaning that 
\[
\begin{tikzpicture}[->,>=to,auto]
\matrix (m) [matrix of math nodes, column sep=1cm, row sep=1cm,  text depth=0.25ex]
{ \cR_1 & \cR_2 \\ \cQ_1 & \cQ_2 \\};
\path (m-1-1) edge  (m-1-2); 
\path (m-1-2) edge (m-2-2); 
\path (m-2-1) edge (m-2-2); 
\path (m-1-1) edge (m-2-1); 
\end{tikzpicture}
\]
commutes. This section is therefore devoted to the basic constructions and
properties of fiber products in the category of monoidal complexes.
Apart from the proof of Theorem~\ref{T:blowing_up_domain}, the material here
is not used until Section~\ref{S:binvarres}.

First consider fiber products in the categories of posets and monoids.
Thus, let
$(A,\leq)$, $(B,\leq)$ and $(C,\leq)$ be posets, with order preserving
maps $(\phi_1)_\# : A \to C$ and $(\phi_2)_\# : B \to C$.

The product $A\times B$ is an poset with elements $(a,b) : a \in A, b \in
B$ and order given by 
\[
	(a_1,b_1) \leq (a_2,b_2) \iff a_1 \leq a_2\ \text{and}\ b_1 \leq b_2,
\]
and the fiber product is the subset 
\[
	A\times_C B = \set{(a,b)\in A\times B \;;\; (\phi_1)_\#(a) = (\phi_2)_\#(b) = c} \subset A\times B,
\]
with the order induced from $A\times B$, which is well-defined since the
$(\phi_i)_\#$ are order preserving. 

\begin{prop}[Fiber products of monoids]
If $\phi_i : \sigma_i \to \sigma$, $i = 1,2$ are (toric) monoid homomorphisms, then
there is a unique toric monoid fiber product $\sigma_1 \times_\sigma \sigma_2$, so with homomorphisms
to $\sigma_i$ such that 
\[
\begin{tikzpicture}[->,>=to,auto]
\matrix (m) [matrix of math nodes, text depth = 0.25ex, column sep=1cm, row sep=1cm] 
{ \sigma_1\times_\sigma\sigma_2 & \sigma_2 \\ \sigma_1 & \sigma \\};
\path (m-1-1) edge node {$\phi'_2$} (m-1-2);
\path (m-1-2) edge node {$\phi_2$} (m-2-2);
\path (m-2-1) edge node {$\phi_1$} (m-2-2);
\path (m-1-1) edge node {$\phi'_1$} (m-2-1);

\end{tikzpicture}
\]
commutes, and which has the usual universal property.

Its faces can be described as follows. Let $A_1$, $A_2$ and $A$ be the posets
whose elements are faces of $\sigma_1$, $\sigma_2$, and $\sigma$, respectively,
and define $(\phi_i)_\# : A_i \to A$ by taking $(\phi_i)_\#(\tau_i)$ to be the
smallest $\tau \in A$ such that $\phi_i(\tau_i)\subset \tau$.  Then faces of
$\sigma_1\times_\sigma \sigma_2$ are indexed by an ordered subset of $A_1
\times_A A_2$.
\label{P:fiber_products_of_monoids}
\end{prop}
\noindent
Faces of $\sigma_1 \times_\sigma \sigma_2$ are not necessarily in bijection
with $A_1 \times_A A_2$, since there will generally be distinct pairs
$(\tau_1,\tau_2) \neq (\tau'_1,\tau'_2) \in A_1\times_A A_2$ (with $\tau'_i
\leq \tau_i$ and $\phi_\#(\tau_i) = \phi_\#(\tau'_i) = \tau$, $i = 1,2$) such
that $\tau_1 \times_\tau \tau_2 = \tau'_1 \times_\tau \tau'_2$.

\begin{proof}
$\sigma_1\times \sigma_2$ is a toric monoid in $N_{\sigma_1}\times
N_{\sigma_2}$, generated by $\set{(v_i,0), (0,w_j)}$ where
$\set{v_1,\ldots,v_n}$ generate $\sigma_1$ and $\set{w_1,\ldots,w_k}$ generate
$\sigma_2$, and the set
\[
	\sigma_1\times_\sigma\sigma_2 = \set{(v,w) \in \sigma_1\times \sigma_2
		\;;\; \phi_1(v) = \phi_2(w)}
\]
(which might be trivial) is evidently closed under addition.  It is the
intersection of $N_{\sigma_1}\times N_{\sigma_2}$ with the cone
\[
	\supp(\sigma_1) \times_{\supp(\sigma)}\supp(\sigma_2)=
	 \supp(\sigma_1)\times\supp(\sigma_2)
	 \cap N^\R_{\sigma_1} \times_{N_\sigma^\R} N_{\sigma_2}^\R,
\]
and hence is a toric monoid.

We will construct a map $I$ from faces of $\sigma_1\times_\sigma \sigma_2$ to
$A_1\times_A A_2$.  Let $\tau \leq \sigma_1\times_\sigma \sigma_2$ be a face.
Since $\sigma_1\times_\sigma \sigma_2 \subset \sigma_1\times \sigma_2$, there
is a smallest face of $\sigma_1\times\sigma_2$ containing $\tau$, which must be
of the form $\tau_1 \times \tau_2$ for some $\tau_i \leq \sigma_i$, $i = 1,2$.
Observe that $(\phi_1)_\#(\tau_1) = (\phi_2)_\#(\tau_2) \in A$: indeed,
$(\phi_1)_\#(\tau_1)$ and $(\phi_2)_\#(\tau_2)$ must have a common face
(containing the image of $\tau$), and then passing to the inverse images of
this face with respect to the $(\phi_i)_\#$ would give a smaller face of
$\sigma_1\times \sigma_2$ containing $\tau$ unless $(\phi_1)_\#(\tau_1) =
(\phi_2)_\#(\tau_2)$.  Thus $(\tau_1,\tau_2) \in A_1\times_A A_2$ and we set $I(\tau) =
(\tau_1,\tau_2)$. 

Next observe that $\tau \leq \tau_1 \times_\sigma \tau_2$; however $\tau$ meets the
interior of $\tau_1 \times \tau_2$ (otherwise $\tau_1 \times \tau_2$ would not
be minimal) and any proper face of $\tau_1 \times_\sigma \tau_2$ would have to
lie in a proper face of $\tau_1\times \tau_2$ (which follows from the fact that
$\tau_1\times_\sigma \tau_2$ is the intersection of the lattice
$N_{\tau_1}\times N_{\tau_2}$, the proper convex cone $C = \supp_{\tau_1}\times
\supp_{\tau_2}$ and the subspace $N^\R_{\tau_1} \times_{N^\R_\sigma}
N^\R_{\tau_2}$), so in fact $\tau = \tau_1 \times_\sigma \tau_2$. It follows
that $I$ is injective.

The universal property of $\sigma_1 \times_\sigma \sigma_2$ in the category of
monoids follows from the analogous property in the category of sets, since the
maps factoring through $\sigma_1 \times_\sigma \sigma_2$ are additive.
\end{proof}


Note that $\sigma_1\times_\sigma \sigma_2$ need not be smooth (or even
simplicial) even if $\sigma_1$ and $\sigma_2$ are, as the example $\phi_1 =
\phi_2 : \Zp^2 \to \Zp$, $(m,n) \mapsto m + n$ shows.



\begin{prop}[Fiber product of monoidal complexes]
Products and fiber products exist in the category of monoidal complexes.
Thus if $\cQ_1$, $\cQ_2$ are monoidal complexes over $A$ and $B$,
respectively, and if $\phi_i : \cQ_i \to \cQ$ are morphisms to a complex
$\cQ$ over $C$, then there exist (complete and reduced) monoidal complexes $\cQ_1 \times \cQ_2$
over $A\times B$ and $\cQ_1 \times_\cQ \cQ_2$ over 
$A\times_C B$ with the requisite universal properties.
\label{P:fiber_products_complexes}
\end{prop}
\begin{proof}
To define $\cQ_1\times \cQ_2$, let $\sigma_{(a,b)} = \sigma_a \times \sigma_b$
for each $(a,b) \in A\times B$.  Likewise, define $\cQ_1\times_\cQ \cQ_2$ to
consist of the distinct $\sigma_{(a,b)} = \sigma_a \times_{\sigma_c}\sigma_b$
for $(a,b) \in A\times_C B$. The only issue is to show that these are complete
and reduced.

Consider first $\cQ_1\times \cQ_2$.  The faces of $\sigma_a \times
\sigma_b$ are the monoids $\sigma_{a'}\times \sigma_{b'}$ where $\sigma_{a'}
\leq \sigma_a$ and $\sigma_{b'}\leq \sigma_b$.  Since the $\cQ_i$ are
complete and reduced, these are in bijection with the elements
$(a',b') \in A\times B$ such that $(a',b') \leq (a,b)$.

Next, it follows from Proposition \ref{P:fiber_products_of_monoids}, that the
faces of $\sigma_a \times_{\sigma_c} \sigma_b$ are monoids of the form
$\sigma_{a'}\times_{\sigma_c} \sigma_{b'} =
\sigma_{a'}\times_{\sigma_{c'}}\sigma_{b'}$ where $(\phi_1)_\#(a') =
(\phi_2)_\#(b') = c' \leq c$, and for such a face $(a',b') \leq (a,b) \in A
\times_C B$ is unique, so $\cQ_1 \times_\cQ \cQ_2$ is complete and reduced.

The universal properties of $\cQ_1\times\cQ_2$ and $\cQ_1\times_\cQ \cQ_2$
follow from the corresponding universal properties of $A\times B$ and
$A\times_C B$ among posets, and of $\sigma_1\times \sigma_2$ and
$\sigma_1\times_\sigma \sigma_2$ among monoids.
\end{proof}

\begin{prop}[Refinements pull back]
If $\phi : \cR \to \cQ$ is a refinement, and $\psi : \cQ_1 \to \cQ$ is any morphism of complexes, then 
\[
	\phi' : \cQ_1 \times_\cQ \cR \to \cQ_1 \ \text{is a refinement.}
\]
In particular, the fiber product $\cR_1 \times_\cQ \cR_2$ of two refinements is
a mutual refinement of each.
\label{P:refinements_pull_back}
\end{prop}
\begin{proof}
Fix $\sigma \in \cQ$, and consider $\tau \in \cR(\sigma)$ and $\sigma_1 \in
\cQ_1$ such that $\psi : \sigma_1 \to \sigma$.  If we identify $\tau$ with its
image in $\sigma$, then as noted above, 
\[
	\phi' : \sigma_1 \times_\sigma \tau \cong \psi^{-1}(\tau) \hookrightarrow \sigma_1
\]
is identified with an inclusion and therefore injective.  From the identity
\[
	\psi^{-1}(\tau_1 \cap \tau_2) = \psi^{-1}(\tau_1) \cap \psi^{-1}(\tau_2)
\]
it then follows that $\sigma_1 \times_\sigma \cR(\sigma) = \set{\sigma_1
\times_\sigma \tau \;;\; \tau \in \cR(\sigma)}$ is a refinement of $\sigma_1$,
and by commutativity of $\phi$ and $\psi$ with the face maps that
$\cQ_1\times_\cQ \cR \to \cQ_1$ forms a refinement of complexes.  
\end{proof}

The complex $\cQ_1\times_\cQ \cQ_2$ need not be smooth even if $\cQ_1$ and
$\cQ_2$ are smooth, and even if they are refinements. For this and other
reasons, it is desirable to know that smooth refinements exist.  

Like the classical algorithms for obtaining smooth refinements in toric
geometry (\cite{fulton1993introduction}, \cite{de6complete}), the following
algorithm consists of two steps. The first step results in a simplicial
refinement and is similar in spirit to the usual algorithms except for the
novelty of subdividing monoids from largest dimension downwards rather than the
other way --- this strategy has the desirable feature of being functorial with
respect to inclusion of complexes. The second step is to produce a smooth
refinement from a simplicial one, and here we make use of the smoothing
operation of Proposition~\ref{P:smoothing_complex} which, as previously noted,
is not available in the conventional algebraeo-geometric settings.

\begin{thm}[Natural smooth refinement]
Let $\cQ$ be a monoidal complex.  Then there exists a {\em natural smooth refinement} $\ns(\cQ) \to \cQ$ with
the following properties.
\begin{enumerate}[{\normalfont (i)}]
\item If $\cQ$ is smooth then $\ns(\cQ) = \cQ$.
\item If $\cQ_0 \subset \cQ$ is a subcomplex, then the corresponding
  subcomplex of $\ns(\cQ)$ is the natural smooth refinement of $\cQ_0$, i.e.
\[
	\ns(\cQ_0) = \ns(\cQ)_0 := \set{\ns(\cQ)(\sigma) \;;\; \sigma \in \cQ_0}.
\]
\end{enumerate}
\label{T:smooth_refinement}
\end{thm}

\begin{proof}

Consider a monoid $\sigma$ with extremals $V = \set{v_1,\ldots,v_n}$.  Let
$\Lambda$ be the set of those $v_i$ which are linearly independent from all of
the others, i.e.  such that $\sspan\set{v_i} \cap \sspan(V \setminus v_i) =
\set{0}.$ Then the monoid $\tau$ which is the largest face of $\sigma$ lying in
the span of $V \setminus \Lambda$ is uniquely determined, and is either
$\set{0}$ or non-simplicial.  It represents the `essential' non-simpliciality
of $\sigma$, since $\sigma$ consists of the join of $\tau$ and the smooth face
generated by $\Lambda$.

We define the {\em non-simplicial dimension} of $\sigma$ by
\[
	\nsdim(\sigma) = \dim(\tau) 
\]
so in particular, $\sigma$ is simplicial if and only if $\nsdim(\sigma) = 0$.
If $\nsdim(\sigma) = \dim(\sigma)$ (i.e.\ $\Lambda = \emptyset$, so $\sigma = \tau$), we 
call $\sigma$ {\em fully non-simplicial.}  

For $j \in \N$, let
\[
	M_j(\cQ) = \#\set{\sigma\;;\; \nsdim(\sigma) = j}.
\]
Since $\cQ$ has a finite number of monoids which are each finite dimensional,
each $M_j(\cQ)$ is finite, and $M_j(\cQ) = 0$ for all $j \geq N$, for some
$N$.

Proceeding by induction, assume that $M_j(\cQ) = 0$ for $j > k$. If there are
no fully non-simplicial monoids of dimension $k$, then $M_k(\cQ) = 0$ (indeed,
if $\nsdim(\sigma) = k$ then it has a fully non-simplicial face $\tau \leq
\sigma$ as above of dimension $k$), and the induction is complete.  Otherwise,
let $\sigma$ with extremals $\set{v_1,\ldots,v_n}$ be a fully non-simplicial
monoid of dimension $k$, and consider the star subdivision $\cS(\cQ,v)$, where
$v = v_1+ \cdots + v_n$.  

If $v$ lies in a monoid $\mu \in \cQ$, then necessarily $\sigma \leq \mu$, and by the induction
hypothesis $k \geq \nsdim(\mu) \geq \nsdim(\sigma) = k$, so equality holds.
Moreover, $\sigma$ must be the {\em unique} fully non-simplicial face of $\mu$
with dimension $k$, so any face $\tau \leq \mu$ with $\nsdim(\tau) = k$ must
have $\sigma \leq \tau$.  
Thus if $v \notin \tau \leq \mu$, then $\sigma \not \leq \tau$
and therefore $\nsdim(\tau + \Zp v) < k$ since $v$ is
independent from the generators of $\tau$ and $\nsdim(\tau) < k$.  

Since all monoids in $\cS(\cQ,v)$ which are not in $\cQ$ are of the form $\tau
+ \Zp v$ where $v \notin \tau$ and $\tau \leq \mu \ni v$, it follows that
$M_j(\cS(\cQ,v)) \leq M_j(\cQ)$ for $j \geq k$.  On the other hand, $M_k$ must
actually decrease by at least one, since $\cQ$ contains a monoid of
non-simplicial dimension $k$ which is not in $\cS(\cQ,v)$, namely $\sigma$.
Thus
\[
	M_k(\cS(\cQ,v)) < M_k(\cQ)
\]
and since $M_k(\cQ) < \infty$, after a finite number of such subdivisions,
we obtain a complex $\cQ'$ for which $M_j(\cQ') = 0$
for $j \geq k$ and $M_j(\cQ') < \infty$ for $j < k$, completing
the inductive step.

Note that the inductive step does not depend on the order of the subdivisions.
Indeed, the result of subdividing $\sigma_1$ and $\sigma_2$ of dimension $k$ in
either order is the same, unless there is a monoid $\mu$ with $\sigma_1
\leq \mu$ and $\sigma_2 \leq \mu$.  But such a $\mu$ would then have
$\nsdim(\mu) > k$, contradicting the induction hypothesis. Upon completion
of the induction, we obtain a complex $\ns'(\cQ)$ such that 
$M_j\pns{\ns'(\cQ)}$ vanishes for all $j$, thus $\ns'(\cQ)$ is simplicial.
Since star subdivisions are refinements and refinements compose, $\ns'(\cQ) \to
\cQ$ is a simplicial refinement.

Finally set $\ns(\cQ) = \ns'(\cQ)_\mathrm{sm}$.  This involves a local operation
on each monoid, and does not depend on any choice of order of the monoids.  The
first property of $\ns(\cQ)$ is clear, since if $\cQ$ is smooth then
$M_j(\cQ) = 0$ for all $j$, so $\ns'(\cQ) = \cQ$ and $\cQ_\mathrm{sm} = \cQ$.  

The second property follows from the fact that either $\big(\cS(\cQ,v)\big)_0 =
\cQ_0$ if $v \in \sigma \notin \cQ_0$, or $\big(\cS(\cQ,v)\big)_0 = \cS(\cQ_0,v)$
otherwise, along with the fact that $\ns'(\cQ)$ depends only on the set of
fully non-simplicial monoids, independent of any choice of order.
\end{proof}

The following is an immediate corollary of
Proposition~\ref{P:fiber_products_complexes} and
Theorem~\ref{T:smooth_refinement}.

\begin{cor}
Let $\cR_i$ $i = 1,2$ be refinements of $\cQ$.  Then a mutual smooth
refinement exists, namely $\ns(\cR_1\times_{\cQ} \cR_2) \to \cQ$.
\label{C:mutual_smooth_refinement}
\end{cor}

Finally, we include here a lemma which will be of use in
Section~\ref{S:binvars}.
\begin{lem}
Let $\cQ_0 \subset \cQ$ be a monoidal subcomplex, and $\cR_0 \to \cQ_0$ a
refinement.  Then there exists a refinement $\cR \to \cQ$ which extends
$\cR_0$, meaning that it contains $\cR_0$ as a subcomplex.
If $\cR_0$ is smooth, then a smooth extension exists.
\label{L:extension_of_refinement}
\end{lem}
\begin{proof}
Say a monoid $\tau \in \cQ$ is ``refined by $\cR_0$'' if $\tau \in
\cQ_0$ and $\cR_0(\tau) \to \tau$ is a nontrivial refinement (i.e.
$\cR_0(\tau) \neq \tau$).  We then say that $\sigma \in \cQ$ is
``damaged by $\cR_0$'' if some face $\tau \le \sigma$ is refined by
$\cR_0$, but $\sigma$ itself is not.  In particular, note that if $\sigma$
is damaged, then $\sigma \notin \cQ_0$, otherwise $\cR_0 \to \cQ_0$
would fail to be a refinement.

If there are no monoids which are damaged by $\cR_0$, then 
\[
	\cR := \cR_0 \cup \cQ \setminus\cQ_0 \to \cQ
\]
is an extension of $\cR_0$.

Let $d(\cQ, \cR_0)$ denote the minimum dimension of $\sigma \in \cQ$ such that
$\sigma$ is damaged by $\cR_0$.  We will produce a refinement
$\cR_1$ of a subcomplex $\cQ_1 \subset \cQ$ which extends $\cR_0$,
and for which $d(\cQ, \cR_1) > d(\cQ,\cR_0)$.  Proceeding by
induction, we eventually obtain a refinement $\cR_k$ of $\cQ_k \subset
\cQ$ extending $\cR_0$ which damages no monoids on $\cQ$, since
$d(\cQ,\cR_k)$ is bounded by the maximum dimension of a monoid in $\cQ$,
and then we can take $\cR = \cR_k \cup \cQ \setminus \cQ_k$ as
above.

For the induction, assume a refinement $\cR_{j-1} \to \cQ_{j-1} \subset \cQ$ is
given, which extends $\cR_0$.  Let $\cQ_j$ be the subcomplex of $\cQ$
consisting of $\cQ_{j-1}$ and all monoids of dimension $d(\cQ, \cR_{j-1})$
which are damaged by $\cR_{j-1}$ along with their faces.  Note that none of
their proper faces are damaged by definition of $d(\cQ,\cR_{j-1})$.  Let
$\Lambda_j$ be the set of all monoids of the form $\tau + \Zp  v$ where $v =
v_1 +\cdots v_n$ is the sum over extremals of a damaged monoid $\sigma$ in
$\cQ_j$, and either 
\begin{enumerate}[{\normalfont (i)}]
\item $\tau \leq \sigma$, such that $\tau \notin \cQ_{j-1}$, or
\item $\tau \in \cR_{j-1}(\tau')$ for some $\tau' \leq \sigma$.
\end{enumerate}
Then set
\[
	\cR_j = \cR_{j-1} \cup \Lambda_j.
\]
We claim $\cR_j$ refines $\cQ_j$, and $d(\cQ, \cR_j) > d(\cQ, \cR_{j-1})$.

Indeed, identifying monoids in $\cR_j$ with their images in monoids in
$\cQ$, the intersection of a monoid in $\cR_{j-1}$ and one in $\Lambda_j$
must be a face of each in $\cR_{j-1}$; the intersection of two monoids in
$\Lambda_j$ is a face of each in $\Lambda_j$; and it is clear that the support
of any $\sigma \in \cQ_j$ is covered by the supports of monoids in
$\cR_j$, so $\cR_j \to \cQ_j$ is a refinement.

For the second claim, suppose there was a monoid $\sigma \in \cQ$ with
$\dim(\sigma) \leq d(\cQ, \cR_{j-1})$ which was damaged by $\cR_j$.  As
noted above, $\sigma \notin \cQ_j$, which means that $\sigma$ is not damaged
by $\cR_{j-1}$ and therefore must be damaged by $\cR_j \setminus
\cR_{j-1} = \Lambda_j$.  In other words, $\sigma$ has a proper face $\tau$
which is non-trivially refined by $\Lambda_j$, but then $\tau \in \cQ_j$ and
$\dim(\sigma) > \dim(\tau) = d(\cQ, \cR_{j-1})$, a contradiction.

Finally, if $\cR_0$ is smooth, then we can replace $\cR$ by
$\ns(\cR)$ since $\pns{\ns(\cR)}_0 = \ns(\cR_0) = \cR_0$ by Theorem
\ref{T:smooth_refinement}.
\end{proof}

\section{Generalized blow-up of a manifold with corners}\label{S:global}

We first describe a functor $X \mapsto \cP_X$ which assigns to a manifold with
corners its `basic smooth monoidal complex' over the poset of the boundary faces of
$X$; and assigns to a b-map $f : X \to Y$ a morphism $f_\natural : \cP_X \to
\cP_Y.$  Next we show that for any smooth refinement $\cR \to \cP_X$, the local
construction of generalized blow-up in Section \ref{S:local} extends to
give a new manifold with corners $[X; \cR]$ with basic complex realizing $\cR.$

Let $X$ be a compact manifold with corners, and consider the set $\cM(X) =
\bigcup_{k=0}^n\cM_k(X)$ of its boundary faces.  It is partially ordered with
respect to inclusion, but here we will equip it with the {\em reverse order}
$(\cM(X),\leq)$, in which 
\begin{equation}
	(\cM(X), \leq) \ni G \leq F \iff G \supseteq F.
	\label{E:boundary_reverse_order}
\end{equation}

\begin{defn}[Basic monoidal complex]
The {\em basic monoidal complex of $X$}, denoted $\cP_X$, consists of the
smooth monoids
\[
	\sigma_F = \bigoplus_{\substack{G \in \cM_1(X)\\F \leq G}} \Zp e_G, \quad F \in \cM(X)
\]
freely generated by the boundary hypersurfaces containing a given face, with
the obvious morphisms
\[
	\sigma_{F'} \hookrightarrow \sigma_{F}, \quad F' \leq F \in \cM(X).
\]
It is smooth, complete and reduced, and indexed by the poset $(\cM(X),\leq)$.

If $f : X \to Y$ is a b-map
\begin{equation}
	f_\# : (\cM(X),\leq) \to (\cM(Y),\leq)
\label{E:f_hash}\end{equation}
is order preserving, and $f$ induces a morphism of monoidal complexes
\begin{equation}
	f_\natural : \cP_X \to \cP_Y
	\label{E:f_natural}
\end{equation}
where $f_\natural : \sigma_F \to \sigma_{f_\#(F)}$ is generated by the boundary exponents 
\eqref{E:f_on_ideals} of $f$:
\begin{equation}
\begin{gathered}
	f_\natural : \Zp\pair{e_{G_1},\ldots,e_{G_k}} \to \Zp \pair{e_{H_1},\ldots,e_{H_l}} \\
	e_{G_i} \mapsto \sum_j \alpha(G_i,H_j)e_{H_j}
	\label{E:f_natural_locally}
\end{gathered}
\end{equation}
\label{D:basic_monoidal_complex}
\end{defn}

The complex $\cP_X$ has a canonical embedding into global sections of the
b-normal bundles over the faces of $X$ as follows.  If $F \in \cM_k(X)$, recall
that $\bN F \to F$ is canonically trivialized by the frame
$\pns{x_1\pa_{x_1},\ldots,x_k\pa_{x_k}}$ where the $x_i$ are any boundary
defining functions for the hypersurfaces $G_i$ containing $F$.  These elements
are well-defined independent of the choices of the $x_i$, and they therefore
generate a smooth monoid which may be identified with $\sigma_F:$
\[
	\sigma_F \cong \Zp \pair{x_1\pa_{x_1},\ldots,x_k\pa_{x_k}} \subset \bN F, \quad F \in \cM(X).
\]
If $F \subseteq F'$, then $\bN F'$ is trivialized by a frame 
$\pns{x_{i}\pa_{x_i}}_{i \in I}$, $I \subset \set{1,\ldots,k}$, and
at any $p \in F$, there is a natural inclusion $\bN_p F' \subset \bN_p F$ which
induces over $F' \leq F$ the natural homomorphism
\[
	i_{F'F} : \sigma_{F'} \cong \Zp \pair{x_{i}\pa_{x_{i}}}_{i \in I}
	 \hookrightarrow \sigma_F \cong \Zp \pair{x_1\pa_{x_1},\ldots,x_k\pa_{x_k}}
\]
which is an isomorphism onto the corresponding face of $\sigma_F$ and agrees
with the previous definition. Furthermore, if $f : X
\to Y$ is a b-map then under this identification \eqref{E:f_natural_locally} is
given by the b-differential:
\[
	f_\natural = \bd f_\ast : \sigma_F \subset \bN F \to \sigma_{f_\#(F)} \subset \bN f_{\#}(F)
\]
since the latter is integral with respect to the bases  
$\big\{x_i\pa_{x_i}\big\}$ for $\bN F$ and
$\big\{x'_j \pa_{x'_j}\big\}$ for $\bN f_\#(F)$, and $\bd f_\ast$ intertwines
the inclusion $\bN F' \subseteq \bN F$ (generating the face map $i_{F'F}$) 
with the corresponding inclusion $\bN f_\#(F') \subseteq \bN f_\#(F).$

Hereafter we will identify $\cP_X$ with its image in the b-normal spaces of $X.$

\begin{thm}[Generalized blow-up of a manifold] 
If $X$ is a manifold with corners then any smooth refinement, $\cR$ of $\cP_X,$
defines a manifold $[X; \cR],$ the {\em generalized blow-up} of $X$ with
respect to $\cR,$ equipped with a {\em blow-down map} 
\[
	\beta : [X; \cR]\longrightarrow X,
\]
restricting to a diffeomorphism of the interiors and such that
$\beta_{\natural}:\cP_{[X; \cR]}\longrightarrow \cP_X$ factors through an
isomorphism 
\[
	\cP_{[X; \cR]} \stackrel{\cong}{\to} \cR
\]
of monoidal complexes.
\label{T:genblow} 
\end{thm}
\noindent
In particular $\cM([X; \cR])$ is determined as an poset by $\cR,$
so $[X;\cR]$ has a unique boundary face $F_\tau$ for each $\tau \in
\cR,$ with $\codim(F_\tau) = \dim(\tau)$ and 
\[
	F_\sigma \subseteq F_\tau \iff \tau \leq \sigma.
\]

\begin{proof} 
As a manifold $X = \bigcup W_{i}$ has a locally finite open covering by
coordinate charts $\phi_i : W_i \cong V_i \subset \R^{k(i)}_+\times \R^{n -
k(i)}.$ It can further be arranged that the cover is `good' in the sense that
all intersections of the coordinate sets are contractible and that the origin
is in the image of each coordinate chart, so that the codimension $k(i)$ is
achieved. Then the image of each coordinate chart is actually diffeomorphic to
$\R^{k(i)}_+\times \R^{n - k(i)}$ so, by composing with such a diffeomorphism,
it can be assumed that each coordinate chart is surjective. The manifold is
then recovered, up to global diffeomorphism, by gluing
\[
X\cong \bigsqcup_i V_{i} / \sim
\]
where the equivalence relation is generated by the transition maps, $p \sim q$
if and only if $f_{ij}(p) = q$ for some $i$ and $j.$ Here, $f_{ij} = \phi_j
\circ \phi_i^{-1} : O_{ij} \to O_{ji}$ is a diffeomorphism on the sets
$O_{ij}= \phi_i(W_i \cap W_j),$ whenever $W_i \cap W_j \neq \emptyset$.

In essence the blown-up manifold is obtained by blowing up each coordinate
chart and showing that the transition maps lift to be smooth. 

For each $i,$ let $F_i \in \cM_{k(i)}(X)$ be the unique boundary face of
(maximal) codimension $k(i)$ such that $F_i \cap W_i \neq \emptyset$ and let $\cR(F_i)=
\cR(\sigma_{F_i})$ be the induced refinement of $\sigma_{F_i}$, interpreted as
a refinement of the basic monoid $\sigma_{F_i} \cong \sigma_{\R^n_+}$ of $\R^n_+$.
Whenever $W_i\cap W_j\neq \emptyset$, there is a unique smallest boundary
face $G_{ij}\in \cM(X)$ which meets $O_{ij}$ and contains both $F_{i}$ and
$F_{j}$, so that $\cR(F_i) = \pns{\cR(G_{ij})}(\sigma_{F_i}).$  Then set
\[
	V'_i = [\R^{k(i)}_+; \cR(F_i)] \times \bbR^{n-k(i)}
\]
and for each pair $(i,j)$, set
\[
	O'_{ij} = (\beta_{i}\times \id)^{-1}(O_{ij}) \subset [\R^{k(ij)}_+; \cR(G_{ij})]\times \R^{n - k(ij)}
	 \subset [\R^{k(i)}_+; \cR(F_i)] \times\bbR^{n-k(i)}.
\]
In light of Corollary~\ref{C:diffeos_lift_wfactors}, there are unique
diffeomorphisms
\[
	f'_{ij} : O'_{ij} \stackrel{\cong}{\to} O'_{ji}
\]
lifting $f_{ij}$, and we can therefore define
\[
	[X; \cR] = \bigsqcup_i V'_i / \sim
\]
where $p \sim q \iff f'_{ij}(p) = q$ for some pair $(i,j)$.  

Since they commute with the transition maps, $f'_{ij},$ the local blow-down maps
$\beta_i : V'_i \to V_i$ patch together to define the global blow-down map
\[
	\beta : [X; \cR] \to X.
\]

Clearly $[X;\cR]$ is paracompact. To verify that it is Hausdorff, let $p$
and $q$ be distinct points. If $\beta(p) \neq \beta(q)$, then they can be
separated by sets of the form $\beta^{-1}(O_i)$ where $O_i$ for $i = 1,2$
are open sets in $X$ separating $\beta(p)$ and $\beta(q).$  On the other
hand, if $\beta(p) = \beta(q)$, then $p$ and $q$ can be separated inside
some set $V'_i$, as in the proof of Proposition~\ref{P:localblow}.

For a fixed $F \in \cM(X)$, it follows from Proposition \ref{P:localblow} that
each $V'_i$ for which $F_i = F$ has boundary faces $(F_i)_\sigma$ in
correspondence with monoids $\sigma \in \cR(F)$.  These are connected in
$[X;\cR]$ for adjacent pairs such that $V_i \cap V_j \cap F \neq \emptyset$
since the diffeomorphisms $f'_{ij}$ preserve the identification of boundary
hypersurfaces with monoids in $\cR(F)$. Thus for each $\sigma \in \cR$ there is
a unique boundary face $F_\sigma \in \cM([X;\cR])$ given by the quotient of the
union of the local boundary faces $(F_i)_\sigma$, and hence the blow-down maps
give an identification
\[
	\cP_{[X; \cR]} \cong \cR. \qedhere
\]
\end{proof}

In fact the blow-up can be defined globally near each boundary face since
$F\in\cM_k(X)$ has a neighborhood in $X$ which is diffeomorphic to
$\R^{k}_+\times F$ and then the preimage of this open set in $[X;\cR]$
is diffeomorphic to $[\R^k_+;\cR(F)]\times F$ obtained by localizing the
resolution to boundary faces containing $F.$

If $f : X \to Y$ is a b-map and $\beta : [Y; \cR] \to Y$ is a generalized
blow-down map we say $f$ is {\em compatible with} $\beta$ if the morphism
$f_\natural : \cP_{X} \to \cP_{Y}$ factors through a morphism $\phi : \cP_X \to
\cR$:	
\begin{equation}
\begin{tikzpicture}[->,>=to,auto]
\matrix [matrix of math nodes, column sep=1cm, row sep=1cm, text depth = 0.25ex] 
{ & |(M_0)| {\cR} \\ |(M_1)| {\cP_X} & |(M_2)| {\cP_Y} \\ };
\path (M_1) edge node {$f_{\natural}$} (M_2) ;
\path (M_0) edge node {$\beta_{\natural}$} (M_2) ;
\path (M_1) edge node {$\phi$} (M_0) ;

\end{tikzpicture}
\label{E:compatible_morphism}
\end{equation}

\begin{thm}[Lifting b-maps]  
If $f:X\longrightarrow Y$ is an interior b-map compatible with a generalized blow-up in
the sense of \eqref{E:compatible_morphism} then $f$ lifts uniquely to a b-map
$f' : X \to [Y; \cR]$ such that
\[
\begin{tikzpicture}[->,>=to,auto]
\matrix [matrix of math nodes, column sep=1cm, row sep=1cm, text depth = 0.25ex] { & |(M_0)| {[Y; \cR]} \\ |(M_1)| X & |(M_2)| Y \\ };
\path (M_1) edge node {$f$} (M_2) ;
\path (M_0) edge node {$\beta$} (M_2) ;
\path (M_1) edge node {$f'$} (M_0) ;

\end{tikzpicture}
\]
commutes, and such that $f'_\natural = \phi : \cP_X \to \cR \cong
\cP_{[Y; \cR]}$.
\label{T:lifting_b_maps}
\end{thm}
\begin{proof}
Again, the construction is local. We consider a covering of $Y$ by
coordinate charts $V_i \cong \R^{k(i)}_+\times \R^{n - k(i)}$ 
and, refining if necessary, a covering of
$X$ by coordinate charts $W_j \cong \R^{l(j)}_+\times \R^{m - l(j)}$
such that for all $j$, $f(W_j) \subset V_i$ for some $i$.  
As in the previous proof, for each $i$ there is a maximal codimension face $F_i
\in \cM_{k(i)}(Y)$ such that $F_i \cap V_i \neq \emptyset$ and similarly for
each $j$ a face $G_j \in \cM_{l(j)}(X)$ such that $W_j \cap G_j \neq
\emptyset$.  Shrinking the coordinate charts if necessary, we can assume
without loss of generality that $f_\#(G_j) = F_i$.

Locally, $f$ has the form
\[
	f_j= f_{|W_j} : \R^{l(j)}_+ \times \R^{m-l(j)} \to \R^{k(i)}_+\times \R^{n-k(i)}
\]
which lifts by Proposition~\ref{P:local_lifting_wfactors} to a b-map
\[
	f'_j : W_j \cong \R^{l(j)}_+\times \R^{m-l(j)} \to V'_i \cong [\R^{k(i)}_+; \cR(F_i)]\times \R^{n-k(i)}.
\]
It follows from the functoriality of these local lifted maps that they are
compatible with the transition maps $g'_{ij} : O'_{ij} \to O'_{ji}$ used to
construct $[Y; \cR]$, and so the $f'_j$ patch together to form a b-map
\[
	f' : X \to [Y; \cR]
\]
as claimed.
\end{proof}

\begin{cor} 
If $\cR_i,\ i = 1,2$ are two refinements of $\cP_X,$ then $[X; \cR_1] \cong [X;
\cR_2]$ over $X$ if and only if $\cR_1 \cong \cR_2.$ as refinements of $\cP_X.$
\label{C:uniqueness_of_blowup}
\end{cor}
\begin{proof}
The `if' direction is clear. For the converse, suppose $\phi : [X; \cR_1] \to
[X; \cR_2]$ is a diffeomorphism (which is necessarily a b-map) which
intertwines the blow-down maps to $X$. Since $P_{[X; \cR_i]} = \cR_i$ it
follows that $\phi_\# : \cR_1 \to \cR_2$ is an isomorphism of complexes which
intertwines the refinement morphisms $\cR_i \to \cP_X$.
\end{proof}

Even if $f : X \to Y$ is not necessarily compatible with the refinement giving the
blow-up $[Y; \cR]$, there are generalized blow-ups of $X$ through which $f$
does lift.  Indeed, from Theorem~\ref{T:lifting_b_maps} above a
generalized blow-up $[X; \cS]$ admits a map to $[Y; \cR]$ over $f : X \to
Y$ precisely when $f_\natural \circ \beta_\natural : \cS \to \cP_Y$ factors
through $\cR \to \cP_Y$.  Such a blow-up always exits.

\begin{thm}[Blowing up the domain]
Let $f : X \to Y$ be an interior b-map and $[Y; \cR] \to Y$ a generalized
blow-up.  Then there exists a generalized blow-up $[X; \cS] \to X$ and a map
$f' : [X; \cS] \to [Y; \cR]$ such that
\[
\begin{tikzpicture}[->,>=to,auto]
\matrix (m) [matrix of math nodes, column sep=1cm, row sep=1cm,  text depth=0.25ex]
{ {[X; \cS]} & {[Y; \cR]} \\ X & Y \\};
\path (m-1-1) edge node {$f'$} (m-1-2); 
\path (m-1-2) edge node {$\beta_\cR$} (m-2-2); 
\path (m-2-1) edge node {$f$} (m-2-2); 
\path (m-1-1) edge node {$\beta_\cS$} (m-2-1); 
\end{tikzpicture}
\]
commutes.  
\label{T:blowing_up_domain}
\end{thm}
\begin{proof}
First we consider $\cP_X \times_{\cP_Y} \cR \to \cP_X$, which is a refinement
by Proposition \ref{P:refinements_pull_back}.  If it is a smooth refinement,
then we take $\cS = \cP_X \times_{\cP_Y} \cR$ and we are done.  Note that in
this case $[X; \cS]$ is the unique ``minimal'' blow-up, meaning it is universal
among blow-ups of $X$ which lift $f$: any other blow-up which lifts $f$ must
factor through $[X; \cS]$ by the universality of fiber products of monoidal
complexes and Theorem~\ref{T:lifting_b_maps}.  In general however, $\cP_X
\times_{\cP_Y} \cR$ is not smooth.  We let $\cS$ be any smooth refinement of
$\cP_X \times_{\cP_Y} \cR$, for instance the natural smooth refinement
$\ns(\cP_X \times_{\cP_Y} \cR)$ of Theorem \ref{T:smooth_refinement}.  There
are many other choices, none of which is universal.
\end{proof}

Note that this includes Theorem~\ref{T:lifting_b_maps} as a special case,
as follows from the fact that $f_\natural : \cP_X \to \cP_Y$ is compatible
with $\cR \to \cP_Y$ if and only if $\cP_X \times_{\cP_Y} \cR \to \cP_X$ is
an isomorphism, which we leave as an exercise for the reader.

\section{Ordinary blow-up and examples}\label{S:blowup}

Recall the ordinary blow-up of $[X; F]$ of a boundary face $F \in \cM_k(X)$.  
As a set this is 
\[
	[X; F] = X \setminus F\ \cup\ SN_+ F
\]
where $SN_+ F \stackrel{\pi}{\to} F$ denotes the inward-pointing spherical normal bundle, 
and the blow-down map $\beta : [X; F] \to X$ is given by the identity on $X \setminus F$
and by $\pi$ on $SN_+ F$ (see Figure~\ref{F:blowup}).  The smooth structure on $[X; F]$ is generated by 
$\beta^\ast C^\infty(X)$ as well as the quotients $x_i/x_j$ (where they are finite)
of boundary defining functions for the boundary hypersurfaces $H_i$ such that
$F$ is a component of $\bigcap_i H_i.$

\begin{prop} The ordinary blow-up of $X$ at the boundary face $F$ is the
generalized blow-up corresponding to the star subdivision along the sum over
generators of $\sigma_F.$  That is, 
\[
	[X; F] \cong [X; \cS(\cP_X,v_F)],
\]
where $v_F = \sum_i x_i \pa_{x_i}$ and $\sigma_F =
\Zp \pair{x_i\pa_{x_i}}_{1 \leq i \leq \codim(F)}.$ 
\label{P:ordinary_blow_up}
\end{prop}

\begin{proof}
With the smooth structure above, the ``front face'' $SN_+ F \in \cM([X; F])$
is a boundary hypersurface which is fibered over $F$, whose fiber is a 
$(k-1)$-simplex where $k = \codim(F)$.  In fact, for any coordinate chart
$(x, x',y) : U \cong \R^{k}_+\times \R^{l-k}_+ \times \R^{n-l}$ in $X$ such that 
$F \cap U \cong \set{x_1 = \cdots = x_k = 0}$, there are $k$ coordinate charts $\wt U_i$, $i = 1,\ldots,k$
covering $\beta^{-1}(U)$ in $[X; F]$:
\[
	(t_i, x', y) : \wt U_i \cong \R^k_+ \times \R^{l-k}_+\times \R^{n-l}
\]
where 
\[
	t_{i,j} = \begin{cases} x_i & \text{if $i = j$}, \\ x_j/x_i & \text{otherwise.}\end{cases}
\]
Thus $\beta : \wt U_i \to U$ has the form
\[
	\beta : (t_i, x', y) \mapsto (t_i^{\mu_i}, x', y) = (x,x',y)
\]
where $\mu_i$ is the identity matrix with its $i$th row replaced by ones:
\[
	\mu_i = \begin{pmatrix} 1 & 0 & \hdotsfor{3} & 0 \\
			      0 & 1 & \hdotsfor{3} & 0 \\
			     \vdots && \ddots &&& \vdots \\
			      1 & \cdots & 1 & 1 & \cdots & 1 \\
			     \vdots & & & & \ddots & \vdots \\
			      0 & \hdotsfor{4} & 1\end{pmatrix}.
\]
Observe that $\mu_i^\transpose$ represents the monoid inclusion
\[
	\Zp \pair{x_1\pa_{x_1}, \ldots, x_{i-1}\pa_{x_{i-1}}, v_F, \ldots, x_k\pa_{x_k}}  \hookrightarrow
	\Zp \pair{x_1\pa_{x_1},\ldots,x_k\pa_{x_k}},
\]
where $v_F = x_1\pa_{x_1} + \cdots +x_k\pa_{x_k}$.  The collection of these
$\mu_i^\transpose$, $i=1,\ldots,k$ therefore give the monoid homomorphisms
from the maximal dimension $\tau \in \cS(\cP,v_F)\pns{\sigma_F}$ to $\sigma_F$.

All coordinate charts meeting $F$ are blown-up in this way, according to the
star subdivision $\cS(\cP_X,v_F)$, and this is precisely the construction of
the generalized blow-up $[X; \cS(\cP_X,v_F)] \to X.$
\end{proof}

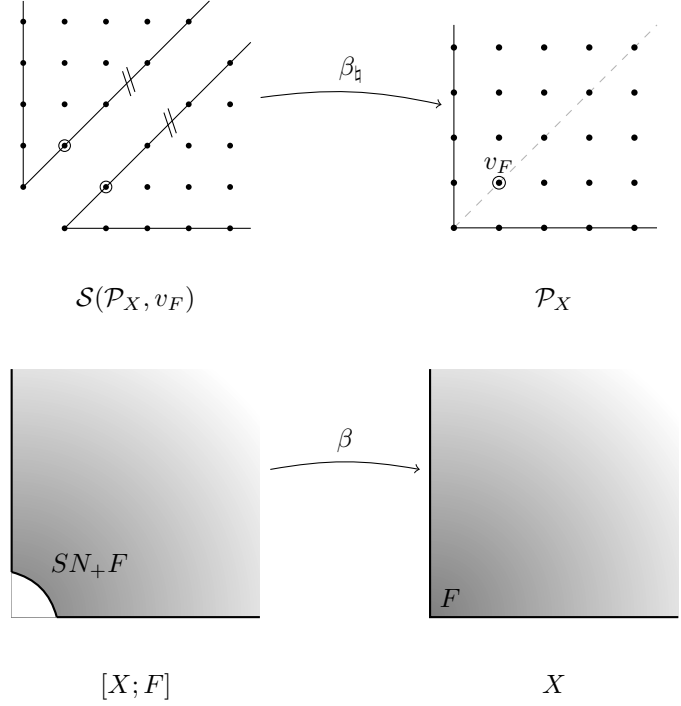
\begin{figure}[tb]
\begin{tikzpicture}
\matrix (m) [matrix of nodes, anchor=base, column sep=2cm, row sep=0.5cm]
{ 
 \begin{tikzpicture}[scale=0.55]
 \draw[-] (0,1) -- (0,5.5);
 \draw[-] (0,1) -- (4.5,5.5);
 \foreach \y in {0,1,...,4} {
	 \foreach \x in {0,...,\y} {
		 \fill (0,1)+(\x,\y) circle (2pt);
		 }
	 }

 \draw[-] (1,0) -- (5.5,0);
 \draw[-] (1,0) -- (5.5,4.5);
 \foreach \x in {0,1,...,4} {
	 \foreach \y in {0,...,\x} {
		 \fill (1,0)+(\x,\y) circle (2pt);
		 }
	 }
  \draw (2.5,3.9) -- (2.7,3.3);
  \draw (2.4,3.8) -- (2.6,3.2);

  \draw (3.7,2.3) -- (3.5,2.9);
  \draw (3.6,2.2) -- (3.4,2.8);

  \draw (2,1) circle (4pt);
  \draw (1,2) circle (4pt);

 \end{tikzpicture}
&
 \begin{tikzpicture}[scale=0.6]
 \draw[-] (2,2) -- (2,6.5);
 \draw[-] (2,2) -- (6.5,2);
 \draw[-,dashed,opacity=0.3] (2,2) -- (6.5,6.5);
 \foreach \x in {0,1,...,4} {
	 \foreach \y in {0,1,...,4} {
		 \fill (2,2)+(\x,\y) circle (2pt);
		 }
	 }
  
  \draw (3,3) circle (4pt);
  \path (3,3) node[above] {$v_F$};
  \end{tikzpicture}
\\
$\cS(\cP_X, v_F)$ & $\cP_X$
\\
 \begin{tikzpicture}[thick,scale=0.6]
  \begin{scope} \clip (0,0) rectangle (5.5,5.5);
      \shade[shading=radial] (-7,-7) rectangle (7,7);
  \end{scope}
  \fill[color=white] (1,0) to[bend right] (0,1) -- (0,0) --cycle;
  \draw (0,1) -- (0,5.5);
  \draw (1,0) -- (5.5,0);
  \draw (1,0) to[bend right] node[midway,auto,swap] {$SN_+F$} (0,1);
  \end{tikzpicture}
& 
 \begin{tikzpicture}[thick,scale=0.6]
  \begin{scope} \clip (0,0) rectangle (5.5,5.5);
      \shade[shading=radial] (-7,-7) rectangle (7,7);
  \end{scope}
  \draw (0,5.5) -- (0,0) -- (5.5,0);
  \path (0,0) node[above right] {$F$};
  \end{tikzpicture}
\\
$[X; F]$ & $X$
\\};
\path (m-1-1) edge[->,>=to,bend left=10] node[midway,auto] {$\beta_\natural$} (m-1-2);
\path (m-3-1) edge[->,>=to,bend left=10] node[midway,auto] {$\beta$} (m-3-2);
\end{tikzpicture}
\caption{The ordinary blow-up of $F$ in $X$, and the associated star subdivision.  
Only the monoids of maximal dimension are pictured.}
\label{F:blowup}
\end{figure}

The {\em inhomogeneous blow-up} of $F \in \cM(X)$ is similar to ordinary
blow-up but more general.  It consists again of the set $X \setminus F \cup
SN_+ F$, but the smooth structure is generated over $\beta^\ast C^\infty(X)$ by
quotients of the form $x_i^{1/n(i)}/x_j^{1/n(j)}$, where the $n(i) \in \Zp $
are consistently associated with the boundary hypersurfaces through $F$.

The proof of Proposition~\ref{P:ordinary_blow_up} can be modified in a
straightforward manner to give the following result.

\begin{prop}
Let $n : \set{H \in \cM_1(X)\;;\; F \subseteq H} \to \Zp $ be an assignment of
integer roots to the boundary hypersurfaces through $F$.  Then the
inhomogeneous blow-up of $F$ with respect to $n$ is realized by the generalized
blow-up by the weighted star subdivision $[X; \cS(\cP_X, v_{F,n})]$, where 
\[
	v_{F,n} = \sum_i n(i) x_i \pa_{x_i}
\]
is the corresponding weighted sum of the generators of $\sigma_F =
\Zp \pair{x_1\pa_{x_1},\ldots,x_n\pa_{x_n}}$.
\label{P:inhomogeneous_blow_up}
\end{prop}

Iterating either of these constructions, we find that iterated boundary blow-up
is also a special case of generalized blow-up.  Recall that the {\em lift} (or
{\em proper transform}) of a submanifold $Y \subset X$ under a blow-up $\beta :
[X; F] \to X$ is the set
\[
	\beta^\#(Y) = \begin{cases} \beta^{-1}(Y) & \text{if}\ Y \subseteq F, \\ 
	\clos\pns{\beta^{-1}\pns{Y \setminus F}}, & \text{otherwise.} \end{cases}
\]
The {\em iterated boundary blow-up} $[X; F_N,\ldots,F_1]$, $F_i \in \cM(X)$ is
defined by successive lifting:
\[
	[X; F_1, \ldots,F_N] = [\cdots[[X; F_1], \beta_1^\#(F_2)], \cdots, \beta_1^\# \circ \cdots \circ \beta_{N-1}^\#(F_N)]
\]
where $\beta_i : [X; F_1, \ldots, F_i] \to [X; F_1, \ldots,F_{i-1}]$, and can be
extended similarly to the case of inhomogeneous blow-up.

Observe that the lift of a boundary face $G \in \cM(X)$ to the blow-up $\beta :
[X; F]  \to X$ is again a boundary face of $[X; F]$.  In light of the
identification of boundary faces of a generalized blow-up with monoids in the
refinement, it follows directly that iterated boundary blow-up is realized by 
iterated star subdivisions.

\begin{cor} 
Let $F_1,\ldots,F_N \in \cM(X)$, and let $\cR_1, \ldots,\cR_N$ be the sequence
of refinements of $\cP_X$ obtained by iteratively defining $\cR_j =
\cS(\cR_{j-1},v_{\wt F_j}),$ where $\wt F_j = \beta_1^\# \circ \cdots \circ \beta_{j-1}^\#(F_j)$.
\[
	[X; F_N, \ldots, F_1] = [X; \cR_N].
\]
\label{C:iterated_boundary_blowup}
\end{cor}

One might wonder at this point if there are generalized blow-ups of $X$ which
are not of the above type.  In fact we will prove in
Section~\ref{S:characterization} that generalized blow-ups are determined
uniquely up to diffeomorphism by their refinements.  It follows that any smooth
refinement of $\cP_X$ which cannot be obtained by iterated (possibly weighted)
star subdivision gives a generalized blow-up which is not of this classical
type.  Examples of these are easy to construct provided $X$ has corners with
codimension at least $3$.

\section{Characterization of generalized blow-down maps}\label{S:characterization}
In this section, which is independent from the remainder of the paper, we
complete our treatment of generalized blow-up, showing that blow-down maps from
a generalized blow-up of the target are characterized analytically among b-maps
in general by two properties. 

\begin{defn}
A {\em generalized blow-down map} between manifolds with corners, $\beta
:X\longrightarrow Y,$ is a proper b-map which is a diffeomorphism of the
interiors and which has b-differential, $\bd \beta_\ast: \bT_xX\longrightarrow
\bT_{\beta (x)}Y,$ an isomorphism for each $x\in X.$
\label{D:gen_blowdown}
\end{defn}

\noindent
We prove below that such a map induces a smooth refinement as its associated morphism
of monoidal complexes.

The range of a continuous proper map is closed, so it follows directly from
the definition that a generalized blow-down map is surjective. It is convenient,
and no restriction, to assume in the subsequent discussion that $X,$ and
hence $Y,$ is connected.  We continue to follow the convention for local coordinates
$(x,y)$, where $x_i \in \R_+$ are local boundary defining
functions and $y_i \in \R$ are tangential variables.

First we find a local normal form for a generalized blow-down map.

\begin{lem} Let $\beta:X\longrightarrow Y$ be a generalized
blow-down map between compact manifolds, $p\in G\setminus\pa
G$ a point in the interior of a boundary face $G\in\cM_{k}(X),$ and
$(x',y') = (x'_1,\ldots,x'_{k'},y'_1,\ldots,y'_{n-k'})$ local coordinates
near $q=f(p)\in F=\beta_\#(G)\in\cM_{k'}(Y).$ Then there exist local coordinates
$(x,y,z)$ near $p$, with the $z_i > 0$, such that, after perhaps
renumbering the $x'$ coordinates near $q$, $\beta$ has the local form
\begin{equation}
	\beta(x,z,y)=(x^{\nu_1},\dots,x^{\nu_{k}},z_1x^{\nu_{k+1}},\dots,
	z_{k'-k}x^{\nu_{k'}},y_{1},\dots,y_{n-k'})=  (x',y'),
	\label{E:blow_down_normal_form}
\end{equation}
where the first $k$ of the $\nu_i\in \Z^k$ are linearly independent.
\label{L:normal_form_genblowdown}
\end{lem}

\begin{proof}
Let $\overline x_j$, $j = 1,\ldots,k$ be local boundary defining functions for
hypersurfaces through $p$.  Then, since $\beta$ is an interior b-map,
\begin{equation}
\beta^\ast(x'_i) = a_i\, \overline x^{\nu_i} = a_i\prod_{j = 1}^k \overline x_j^{\nu_{ji}},\;
0 < a_i \in \CI(X)
\label{E:beta_pullback}
\end{equation}
where $\nu_{ji} \in \Zp .$ Indeed, $\nu^\transpose \in \Mat(k'\times k, \Zp )$
is the matrix representing $\bd \beta_\ast :\bN G \smallto \bN F$ with respect to the bases
$\big\{\overline x_i\pa_{\overline x_i}\big\}$ and $\big\{x'_j\pa_{x'_j}\big\}.$

Since $\nu$ must have full rank (or else $\bd \beta_\ast$ could not 
be bijective), relabeling the $x'_i$ appropriately ensures that the $k\times
k$ matrix formed by the first $k$ entries in the $\nu_{ji}$ is
invertible and these give the vectors $\nu_i := (\nu_{1i},\ldots,\nu_{ki})$ 
for $i=1,\dots,k.$ 

Changing the $\overline x_j$ by positive smooth factors, $x_j = b_j\overline
x_j,$ $0 < b_j \in \CI(X),$ multiplies the coefficient functions $a_i$ in
\eqref{E:beta_pullback} by the monomials $b^{\nu_i}.$  The independence of
$\nu_i$, $i \leq k$ means that the $b_j$ can be chosen so each $a_i\equiv1$ for
$i\le k.$ This gives the first $k$ equations in
\eqref{E:blow_down_normal_form}.

The tangential coordinates $y'_i$ pull back under $\beta$ to be
smooth and independent at $p,$ so we take $y_i = \beta^\ast y'_i$ 
for $i = 1,\ldots,n-k'$, ensuring the last $n-k'$ equations in
\eqref{E:blow_down_normal_form}, without affecting the first $k.$

Finally then consider the pull-back of the last $k'-k$ boundary defining
functions. The logarithmic differentials, $\bd \beta^*(dx_i'/x'_i)$ must be
independent at $p,$ and be independent of the $dx_j/x_j,$ $j=1,\dots k,$
and $dy_l,$ $l=1,\dots,n-k'.$ In view of \eqref{E:beta_pullback} this means
precisely that the smooth differentials $da_i/a_i$ must be linearly
independent at $p$ for $i=k+1,\dots, k'$ and since the $a_i>0$ this in turn
is equivalent to the independence of the corresponding $da_i,$ so $z_i := a_{k-i}$,
$i = 1,\ldots,k-k'$, can be introduced as additional tangential coordinates giving
\eqref{E:blow_down_normal_form}.
\end{proof}

In fact, one can take the $x'_i$ to be globally defined boundary defining
functions on a neighborhood of $F \setminus \pa F$, and then it follows from the
proof that the $x_i$ can be taken to be global on a neighborhood of $G
\setminus \pa G$.  It follows similarly that the $y_i$ and $z_j$ are globally
defined on each fiber $\beta^{-1}(q') \cap G \setminus \pa G$ for $q'$ in a
neighborhood of $q$, and it follows that these fibers are contractible.  In
fact, more is true.

\begin{lem}
The functions $z_i$, $i = 1,\ldots,k'-k$ in \eqref{E:blow_down_normal_form} are
globally defined on the fibers $\beta^{-1}(q') \cap G \setminus \pa G$ for $q'$
in a neighborhood of $q,$ and the map
\begin{equation}
	z : \beta^{-1}(q) \cap G \setminus \pa G \to (0,\infty)^{k'-k}
	\label{E:blow_down_z_map}
\end{equation}
is surjective.
\label{L:blow_down_fiber_variables}
\end{lem}

\begin{proof}
As noted, $z$ is a globally defined map on the fiber in light of
\eqref{E:beta_pullback} and the fact that $z_i = a_{k-i}$.  Suppose then that
\eqref{E:blow_down_z_map} is not surjective, and consider a point in the
closure of the image under $z$ of a component of $\beta^{-1}(q) \cap G\setminus
\pa G$.  By compactness of $G$, this must be the image of a point $p'$ in $\pa
G$, in the interior of a boundary face $H \subseteq G$ with $H \in \cM_{r}(X)$,
say.  The assumption that $z$ is not surjective means that $z(p')$ lies inside
$(0,\infty)^{k-k'}$, so that $0 < z_i(p') < \infty$ for all $i \leq k-k'$.   

By continuity $p'$ is mapped by $\beta$ to $q$.  Consider the construction of
the $x_i$ and $z_i$ in the proof of Lemma \ref{L:normal_form_genblowdown},
where we now begin with additional boundary defining functions $\overline
x'_i$, $i = k+1,\ldots,r \leq k'$ for $H$.  The equations
\eqref{E:beta_pullback} become
\[
	\beta^\ast(x'_i) = a_i\,\overline x^{\nu_j} 
	 = a'_i\, {\overline x'}^{\gamma_i}\, \overline x^{\nu_i},\; 0 < a'_i,
\]
where not all the $\gamma_{j} \in \Zp ^{r-k}$ can vanish without violating the
isomorphism condition on $\bd\beta_\ast$.  It follows by going through the
construction in the proof that $z_i = e_i {\overline x'}^\mu_i$, where $\mu_i \in
\Q^{r}$ cannot all vanish (though $\mu_{ji}$ may be non-integral and/or
negative since the $b_j$ may now involve the $\overline x'_i$), so at least one
$z_i$ tends to either zero or infinity at $H$, giving a contradiction.
\end{proof}

From the normal form \eqref{E:blow_down_normal_form} we derive the following
path lifting result, which determines in which face $G \in \beta^{-1}(F)$ a
lifted path will hit the boundary of $X$ from its b-tangent at $t = 0.$

\begin{lem} Let $\beta : X \to Y$ be a generalized blow-down, $G \in \cM(X),$ 
$F = \beta_\#(G) \in \cM(Y)$ and $q \in F \setminus \pa F,$ with coordinates
$(x',y')$ centered at $q$. If
\[
	\gamma : [0,\epsilon) \ni t \mapsto (t^\kappa, 0) \in Y
\]
is a path with endpoint at $q$ and initial b-tangent vector $\kappa = \sum
\kappa_i x'_i \pa_{x'_i} \in \bN_+ F$, with all $\kappa_i > 0$ then the lift
$\clos(\beta^{-1}(\gamma((0,\epsilon)))$ of the image of $\gamma$ to $X$ meets $\beta^{-1}(q)
\cap G \setminus \pa G$ if and only if there exists $\lambda = \sum_i \lambda_i
x_i\pa_{x_i} \in \bN_+ G$ such that $\bd f_\ast(\lambda) = \kappa.$
\label{L:genblow_path_lifting}
\end{lem}

\begin{proof} Since $\beta$ is a diffeomorphism on the interiors,
the smooth path $\gamma : (0,\epsilon) \to Y \setminus \pa Y$ lifts to
$X \setminus \pa X.$  If such a $\lambda \in \bN_+ G$ does exist, then 
near a component of $\beta^{-1}(q) \cap G \setminus \pa G$ it follows
directly that $\gamma$ has a lift
\[
	[0,\epsilon) \ni t \mapsto (t^\lambda, 0, 1) = (x,y,z)
\]
extending to $t = 0,$ with endpoint $p = (0,0,1) \in G\setminus \pa G.$ 
(Observe that $p$ lies in the domain of the coordinates $(x,y,z)$ in light of
Lemma \ref{L:blow_down_fiber_variables}.)

If no such $\lambda$ exists, then every point $p$ of 
$\beta^{-1}(q)\cap G \setminus \pa G$ has a neighborhood 
\[
	D(\epsilon,p) 
	 = \set{(x,y,z) \;;\; x_i < \epsilon, i = 1,\ldots,k, \abs{y} < \epsilon},\; 
	 \epsilon > 0
\]
which does not meet the lift of $\gamma$ to $X \setminus \pa X.$  

Indeed, the image of such a neighborhood under $\beta$ contains a point of
$\gamma$ if and only if
\[
	(t^{\kappa_1},\ldots,t^{\kappa_{k'}},0) 
	 = (x^{\nu_1},\ldots,x^{\nu_k},z_1x^{\nu_{k+1}},\ldots,z_{k'-k}x^{\nu_{k'}},y).
\]
Taking the logarithm of the first $k$ conditions gives
\begin{equation}
	\kappa_i\log t = \sum_j\nu_{ji}\log x_j=(\bd \beta_\ast\log x)_i
	\label{E:log_path_equation}
\end{equation}
since $\nu^\transpose$ is the matrix representing $\bd\beta_\ast.$
Now, in $D(\epsilon,p)$, the vector $(\log x)\in(-A,-\infty)^k$ where
$A=-\log\epsilon >0.$ Thus, right side of \eqref{E:log_path_equation} lies in
$-\bd \beta_\ast (\bN_+G)$, but by assumption $\kappa \notin\bd \beta_\ast(\bN_+G)$
so \eqref{E:log_path_equation} can have no solution with $0<t<1.$
\end{proof}

In fact, though we do not use this directly below, Lemmas~\ref{L:normal_form_genblowdown} 
and \ref{L:genblow_path_lifting} show that for all $G$ such that $\beta_\#(G) = F$, 
the map $\beta : G\setminus \pa G \to F \setminus \pa F$ is a fibration, with
fibers diffeomorphic to $(0,\infty)^{k'-k}.$

To see this note that since $\beta : G \to F$ is an interior b-map it follows from
\eqref{E:blow_down_normal_form} that
\begin{equation}
	\beta : G\setminus\pa G \to F\setminus\pa F
	\label{E:beta_G_to_F}
\end{equation}
is a submersion. By compactness of $G$, any limit point of the range of
\eqref{E:beta_G_to_F} must be the image of a point in $G$, so it must be
surjective, or else an interior point of $F$ would only meet the boundary
of $G$ which is inconsistent with the defining properties of an interior
b-map.  By Lemma \ref{L:normal_form_genblowdown} it follows that $z :
\beta^{-1}(q) \cap G\setminus \pa G \to (0,\infty)^{k'-k}$ is a covering
map, hence the fibers of \eqref{E:beta_G_to_F} are a disjoint union of
components diffeomorphic to the latter space, but by
Lemma~\ref{L:genblow_path_lifting} these fibers must be connected, since
otherwise there would be multiple preimages of $\gamma(t)$ for small $t$.


\begin{prop} If $\beta : X \to Y$ is a generalized blow-down map then the morphism 
of monoidal complexes
\[
	\beta_\natural : \cP_X \to \cP_Y
\]
is a smooth refinement.
\label{P:refinement_from_blowdown}
\end{prop}

\begin{proof} Since $\bd \beta_\ast$ is bijective, $\beta_\natural$ is
  necessarily injective, and it suffices to verify that the cones
\[
\set{\supp\big(\beta_\natural(\sigma_G)\big) =
\beta_\ast(\bN_+ G) \;;\; \beta_\#(G) = F} \subset \supp(\sigma_F) = \bN_+F
\]
have union equal to $\bN_+ F$, and have no common interior vectors.

First observe that the union of these cones is indeed $\bN_+ F$.  If not, the
complement, which is open, would contain an interior point of $\bN_+ F$.  Thus,
proceeding by contradiction, we can suppose that there is a vector $\kappa =
(\kappa_1,\ldots,\kappa_k)\in \bN_+ F$, with positive integer entries, which is
disjoint from all the $\beta_\ast(\bN_+ G)$.  By Lemma
\ref{L:genblow_path_lifting} there is a path $\gamma(t) \in Y$ with initial
b-differential equal to $\kappa$ and endpoint at $q \in F\setminus \pa F$ whose
lift to $X$ does not meet any $G$ with $\beta_\#(G) = F$.  On the other hand,
by the properness of $\beta$ there is a sequence $0 < t_j \smallto 0$ such that
the points $\beta^{-1}(\gamma(t_j))$ converge in $X$.  By continuity, the limit
must be in $\beta^{-1}(q)$.  However, since $q$ lies in the interior of $F$
such a point must lie in the interior of one of the $G$ with $\beta_\#(G) = F$
which is a contradiction.

Next consider two of the cones $\beta_\ast(\bN_+ G_i)$, $i = 1,2$ which contain
an interior vector $\kappa$ of $\bN_+ F$ in both their relative interiors,
i.e.\ this point is the image of an interior vector $\lambda_i$ of each of the
$\bN_+ G_i$.  Lemma \ref{L:genblow_path_lifting} applies to both faces, hence
the curve $\gamma$ with endpoint $q \in F\setminus \pa F$ and initial b-tangent
vector $\kappa$ has a lift with endpoint in the interior of each boundary face.
The assumption that $\beta$ is a diffeomorphism in the interior therefore
ensures that these two boundary faces have interiors which intersect and they
are therefore equal.

We conclude that $\beta_\natural : \cP_X \smallto \cP_Y$ is a refinement,
which must be smooth since each $\sigma_G \in \cP_X$ is smooth.
\end{proof}

\begin{prop} \label{P:genblowdown_identity} A generalized blow-down map
  $\beta : X \to Y$ is a diffeomorphism if and only if
\[
	\beta_\natural : \cP_X \to \cP_Y
\]
is invertible.
\end{prop}

\begin{proof} That $\cP_Y \cong \beta_\natural(\cP_X)$ if $\beta$ is a diffeomorphism is
clear, since $\beta^{-1}_\natural$ furnishes an inverse.

Assume then that $\beta_\natural: \cP_X \cong \cP_Y$.  For any $F \in
\cM_{k'}(Y)$, there is therefore a unique $G \in \cM_k(X)$ with $\beta_\#(G) =
F$ and $\beta_\natural : \sigma_G \cong \sigma_F$ (it follows that $k = k'$).
Since the $\nu_j$, $j = 1,\ldots,k$ in Lemma \ref{L:normal_form_genblowdown}
are precisely the coordinates for the generators of $\sigma_G$ in terms of
those of $\sigma_F$, we can arrange that $\nu = \id$, and there are therefore
local coordinates as in \eqref{E:blow_down_normal_form} near each point in
which $\beta = \id$.
\end{proof}

\begin{thm} If $\beta : X \to Y$ is a generalized blow-down map then
\[
	X \cong [Y; \cP_X]
\]
with respect to the simplicial refinement $\beta_\natural : \cP_X \to \cP_Y.$
\label{T:characterization}
\end{thm}

\begin{proof}
Let $Y_0 = [Y; \cP_X]$, and let $\beta_0 : Y_0 \to Y$ be the blow-down.  The
morphism $\beta_\natural : \cP_X \to \cP_Y$ is tautologically compatible with
the refinement $\cP_X \to \cP_Y$, since $\beta_\natural$ factors through the
identity morphism $\id : \cP_X \to \cP_X$.

From Theorem \ref{T:lifting_b_maps} then, $\beta$ lifts to a b-map
\[
	\beta' : X \to Y_0
\]
which is easily seen to be a generalized blow-down since $\beta : X \setminus
\pa X \to Y \setminus \pa Y$ factors through $\beta' : X \setminus \pa X \to
Y_0 \setminus \pa Y_0$, which must therefore be a diffeomorphism.  Since the lifted
map on monoidal complexes is just the identity,
\[
	\beta'_\natural = \id : \cP_X \to \cP_X,
\]
$X \cong Y_0$ by Proposition~\ref{P:genblowdown_identity}.
\end{proof}

\section{Binomial subvarieties}\label{S:binvars}
We consider subvarieties of a manifold which near the boundary
have the local form 
\begin{equation}
	a_i x^{\alpha_i} = b_i x^{\beta_i},\; y_j = 0,\ a_i,\ b_i>0,
	\label{E:binvar_idea}
\end{equation}
which is to say they are given by the vanishing of some binomial equations with
smooth positive coefficients in the boundary defining variables, and the
vanishing of some interior variables. Such objects occur naturally in the
setting of fiber products and in other contexts, for instance the embedding
under a b-map of one manifold into another.

After verifying that such subvarieties are well-defined and have a standard
form, we show that, although they are not in general smoothly embedded
manifolds with corners, they have enough structure to support the machinery
of monoidal complexes and b-maps. This allow us to develop their resolution
theory in the next section. While it would be possible to define an
intrinsic category of abstract `binomial varieties', since it suffices in
our later treatment of fiber products, attention here is restricted to the
case of a binomial variety explicitly embedded in a manifold.

Observe in \eqref{E:binvar_idea} that by dividing by $b_ix^{\beta_i}$ and
also exponentiating the interior coordinates, the equations take the
unified form $a'_i x^{\gamma_i} = 1$, with $\gamma_i = \alpha_i - \beta_i$,
$a'_i = a_i/b_i>0$ in the first case and $\gamma_i = 0$, $a_j = \exp(y_j)$ in
the second. This may involve the cancellation of factors of $x_i$ which
appear on both sides of the equation and hence the loss of some solutions
contained entirely in the boundary. In view of this we will define binomial
structures on sets which are the closure of their intersection with the
interior as in \eqref{E:loc_bin_str} below.

With this motivation in mind, on any manifold consider the set of functions 
\begin{equation}
\begin{gathered}
\cG(X) = \bigcup_{\gamma: \cM_1(X)\smallto \Z}\cG_\gamma(X),\ \text{where}\\
\cG_\gamma (X)=\set{u\in\CI(X\setminus\pa X)\;;\;u=a\rho^\gamma ,\ 0<a\in\CI(X)};
	\label{E:G_gamma}
\end{gathered}
\end{equation}
here $\rho = (\rho_H)_{H \in \cM_1(X)}$ is a collection of global boundary
defining functions.  Observe that $\gamma : \cM_1(X) \to \Z$ induces a
functional $\bd \gamma : \bN G \to \C$ for any $G \in \cM(X)$ by setting 
\begin{equation}
	\bd \gamma = \sum_{\cM_1(X) \ni H \supseteq G} \gamma(H) \frac{d\rho_H}{\rho_H}.
	\label{E:b_gamma}
\end{equation}

The $\cG_{\gamma}(X)$ are independent of the defining functions used in
\eqref{E:G_gamma} and pull back under any interior b-map $f:Z\longrightarrow
Y$ giving an inclusion $f^*\cG(Y)\subset\cG(X).$ Clearly $\cG(X)$ is an abelian
group under pointwise multiplication, and $\cG_{\gamma_1}(X)\cdot
\cG_{\gamma_2}(X) \subset \cG_{\gamma_1+\gamma_2}(X).$

A function $f \in \cG(X)$ can be extended by continuity to points in $\pa X$ at
which it has finite limits but the logarithmic differential, $df/f,$ extends by
continuity to a smooth global section of $\bT^\ast X$ since near the boundary it
reduces to 
\[
	\frac{df}f=d\log a+\sum\limits_{H}\gamma(H)\frac{d\rho _H}{\rho _H}.
\]

\begin{defn} A {\em local binomial structure} on a closed subset $D \subset X$ near a point
$p\in D$ consists of a coordinate neighborhood $U \ni p$ and functions $f_i \in
\cG_{\gamma_i}(U)$, $i=1,\ldots,d$ which have independent logarithmic
differentials $d f_i/f_i \in \bT^\ast_q U$ at each $q \in D \cap U$, and which
define $D$ locally in the sense that
\begin{equation}
	D \cap U = \clos_U \set{q \in U \setminus \pa U\;;\;f_i(q)=1,\ i = 1,\ldots,d}.
	\label{E:loc_bin_str}
\end{equation}
The \emph{codimension} of $D$ at $p$ is $d.$ 
\end{defn}

An exponent vector $\gamma \in \cM_1(X) \to \Z$ is said to be {\em
non-negative} (resp. {\em non-positive}) if $\gamma(H) \geq 0$ (resp. $\leq
0$) for all $H \in \cM_1(X)$ and $\gamma$ is {\em indefinite} if it is nonzero,
and neither non-negative nor non-positive, so an indefinite $\gamma$ must have
at least one positive and at least one negative coefficient.  More locally,
$\gamma$ is non-negative {\em with respect to $F \in \cM(X)$} if $\gamma(H)
\geq 0$ for all $H \in \cM_1(X)$ such that $F \subseteq H$ and non-positivity
and indefiniteness with respect to $F$ are defined similarly.

For a local binomial structure $U$ on $D \subset X$, the boundary faces of
$X \cap U$ which are met by $D$ can be seen by examining the indefiniteness of
the exponent vectors.

\begin{lem}
If $U,$ $f_{i}\in\cG_{\gamma_i}(U)$ is a local binomial structure on $D$ near
$p$ and $U$ only meets boundary hypersurfaces which pass through $p$ then every
exponent vector $\gamma_i$ is either zero or indefinite.  Similarly if $D$
meets the interior of $F\in\cM(U)$ then each $\gamma_i$ is either zero or
indefinite with respect to $F.$ 
\label{L:faces_met_by_D}
\end{lem}

\begin{proof} 
By assumption all the local boundary defining functions vanish at $p$ so if
$\gamma_i$ is non-zero and has all entries of a fixed sign then
$a_ix^{\gamma_i}$ cannot be equal to $1$ near $p.$ The same argument applies to
other boundary points of $D$ where fewer of the boundary defining functions
vanish.
\end{proof}

Next we establish a local normal form for local binomial structures.

\begin{lem} \label{L:bin_str_normal_form}
If $U,$ $f_i \in \cG_{\delta_i}(U)$ is a local binomial structure
on $D$ then near any boundary point $q\in D$ in $U$ there are local
coordinates $(x,y)$ in a (possibly smaller) neighborhood $U'$ of $q$ in terms of which
\begin{equation}
D\cap U'=\set{x^{\gamma_i}=1,\  i=1,\ldots,d',\ y_j=0,\ j=d'+1,\ldots,d}
\label{E:bin_str_normal_form}
\end{equation}
where the $\gamma_i$ are linearly independent and indefinite vectors with respect
to the maximal codimension boundary face through $q.$ 
\end{lem}

\begin{proof} 
By definition of local binomial structure, $D\cap U$ is the closure of its
intersection with the interior, so there are interior points of $D$ near $q.$
Let $F$ be the boundary face of maximal codimension containing $q;$ so $q$ lies
in the interior of $F.$ The b-cotangent space $\bT^\ast_{q} U$, of which the
$d\log f_i$ are sections, has a `smooth subspace' $\bN^\perp_q F$ consisting of the
differentials of smooth functions; it is the image in $\bT^\ast_{q} U$ of the
natural map $T^\ast_{q} U \to \bT^\ast_{q} U$, and is the annihilator
of the b-normal space $\bN_{q} F.$

The assumed independence of the logarithmic differentials $df_i/f_i$ in $U$ 
ensures that they span a linear space of dimension $d$ at $q,$
\begin{equation}
	A_q(D) := \sspan\set{df_i/f_i} \subset \bT^\ast_q U.
	\label{E:conormal_to_binvar}
\end{equation}
Within this space consider the intersection
\[
	A^\mathrm{sm}_q(D) = A_q(D) \cap \bN^\perp_q F
\]
with the smooth subspace and set $d-d' =
\dim(A^\mathrm{sm}_q(D)).$ In terms of local coordinates $(x,y)$ at $q,$ 
the functions $f_i=a_i\,x^{\delta'_i}$ where the $\delta'_i$ are
the restrictions of the $\delta_i$ to the boundary hypersurfaces through
$q,$ and $\bN^\perp_q F$ is the span of the $dy_k$.  
Thus $A^\mathrm{sm}_q(D)$ is spanned by those linear combinations of
logarithmic differentials
\[
	\sum\limits_{i}t_idf_i/f_i=\sum\limits_{i,j}t_i\delta'_{ij}dx_j/x_j+\sum\limits_{i}t_id\log a_i
\]
for which $\sum_i t_i \delta'_i \equiv 0$.

Thus, after some renumbering, an independent set of the $\delta'_i$ can be
chosen and renamed $\gamma_i,$ $i = 1,\ldots,d'.$ The remaining $f_j$, $j =
d'+1,\ldots,d$ can be replaced by products $f'_j = \prod_i f_i^{t_i}$
corresponding to nontrivial independent relations $\sum_i t_i \delta'_i = 0.$
Then the set $\set{df_i/f_i, df'_j/f'_j}$ is independent and spans $A_q(D)$ and
$D$ is given locally by the equations $f_i = f'_j = 1$, but the $f'_j =
a'_j\,x^0$ are smooth and positive.  

The functions $\log a'_j$, $j = d'+1,\ldots,d$ can then be introduced in place
of some of the tangential variables $y_j$ and their differentials span
$A^\mathrm{sm}_q(D)$.  Since the $\gamma_i$ are now linearly independent, after
changing the boundary variables from $x_k$ to $g_k\,x_k$ with $g_k>0$ and
renumbering, $a_i\,x^{\gamma_i}$ is reduced to the desired form $x^{\gamma_i}$
giving \eqref{E:bin_str_normal_form}.
\end{proof}

\begin{defn} 
A connected, closed subset $D \subset X$ of a manifold with corners is an
\emph{interior binomial subvariety} if it has a covering by local binomial
structures.  
\label{D:int_bin_subvar} 
\end{defn}
The codimension of $D$ is well-defined as the local codimension by
connectivity, and $D \cap \pns{X\setminus \pa X}$ is a smooth manifold of
dimension $\dim(D)= \dim(X) - \codim(D)$. 

It follows from the proof of Lemma \ref{L:bin_str_normal_form} that the
`b-conormal spaces' $A_{q}(D)\subset \bT^\ast_{q}X$ in
\eqref{E:conormal_to_binvar} are well-defined at each point $q \in D$ and
independent of the local binomial structure used -- for points in the
interior $A_q(D)$ is just the ordinary conormal space to the smooth manifold $D
\cap X \setminus \pa X$, and its extension by continuity to $D \cap \pa X$ is
unique.  So we may proceed as for a smooth submanifold of a manifold and set 
\begin{equation}
\begin{gathered}
	\bT_{q}D=(A_{q}(D)\subset \bT^\ast_{q}X)^\perp\subset\bT_{q}X,\\
	\bN_{q}D_G=\bT_{q}D\cap\bN_{q}G, \quad p \in D \cap G,\ G \in \cM(X).
	\label{E:binvar_bundles}
\end{gathered}
\end{equation}
The b-tangent bundle $\bT D \to D$ is actually independent of the `binomial
embedding' of $D$ in $X$ once this is understood correctly, but here we persist
with the extrinsic discussion.  Observe that each $\bN_q D_G$ is just the
nullspace of the maps $\bd \gamma_i$ as in \eqref{E:b_gamma} for the $\gamma_i$
occuring in any local binomial structure for $D$ near $p$.  In particular it is
invariant with respect to replacing the $\gamma_i$ by linear combinations with
the same span.

From the local normal form \eqref{E:bin_str_normal_form} it follows that the
intersection of $D$ with a boundary face through $p$ is again a binomial
subvariety, as we now show.  By way of motivation, notice that the behavior of
the b-tangent spaces on passage to a boundary face of $X$: for $p\in
G\in\cM(X),$ $\bT_pG=\bT_pX/\bN_pG.$

\begin{lem}\label{L:binvar_boundary} 
Suppose $D \subset X$ is an interior binomial variety with codimension $d,$
then for any $G\in \cM(X),$ such that $D \cap (G \setminus \pa G) \neq
\emptyset$, a component $G'$ of the closure $D_G := \clos\pns{D \cap G
\setminus \pa G}$ is an interior binomial subvariety of $G$ with codimension $d
- \dim(\bN_p D_G)$, for any $p \in G'$ and
\begin{equation}
	\bT_pD_G=\bT_pD/\bN_pD_G.
\label{E:binvar_boundary_tangent}
\end{equation}
\end{lem}
\noindent It may well happen that $D\cap(G\setminus\pa G)=\emptyset$ but $D
\cap G \neq \emptyset$ in which case we do not regard $D\cap G$ as an {\em
interior} binomial subvariety of $G$ since it lies only in the boundary. 

\begin{proof} 
Let $p$ be a limit point of $G' \cap G\setminus \pa G$, and consider the local
form for $D$ near $p$ given on a neighborhood $U'$ by
\eqref{E:bin_str_normal_form}; we can assume without loss of generality that
$U'$ does not meet any other component of $D \cap G.$ Thus the $x$'s are local
defining functions for the boundary face $F \subseteq G$ of maximal codimension
through $p.$  Divide them as $x = (x',x'')$  where the $x''$ define $G$
locally. 

Following the proof of Lemma~\ref{L:bin_str_normal_form} above, a maximal
subset of the $\gamma_i$ may be chosen so that their restrictions to the
boundary faces containing $G$ are independent, i.e.\ the projections $\gamma_i
= (\gamma'_i,\gamma''_i) \smallmapsto \gamma''_i$ onto the subspace
corresponding to the $x''$ are independent.  The remaining $\gamma_j$ may be
replaced by independent linear combinations $\delta_j=\sum_i t_i\gamma_i$ which
are non-zero only with respect to the boundary hypersurfaces which do not
contain $G,$ i.e.\ their projections $\delta_j = (\delta'_j,\delta''_j)
\smallmapsto \delta''_j$ vanish.  The defining conditions
\eqref{E:bin_str_normal_form} may therefore be rewritten, after renumbering,
\begin{equation}
\begin{gathered}
	(x')^{\gamma'_i}(x'')^{\gamma''_i} = 1,\; i = 1,\ldots,k,\
	(x')^{\delta'_i} = 1,\; i = k+1,\ldots,d',\\ y_j = 0,\; j = d'+1,\ldots,d,
	\label{E:bin_str_split_normal_form}
\end{gathered}
\end{equation}
where the $\gamma''_i$ and $\delta'_i$ are separately independent, and the
$\gamma''_i dx''_i/x''_i$ span $\bN_p D_G$.

Then
\begin{equation}
	G' \cap U' = \set{(x')^{\gamma'_i} = 1, i = k+1,\ldots,d',\; y_j = 0,\ j = d'+1,\ldots,d}
	\label{E:binvar_boundary_form}
\end{equation}
is a local binomial structure on $G'.$ Certainly the right side
of \eqref{E:binvar_boundary_form} is included in the left. To see the converse, observe
that, by the assumption that $D$ meets the interior of $G$ near $p,$ the
system \eqref{E:bin_str_split_normal_form} must have a sequence of solutions with $x''\smallto0$
and all entries of $x'$ positive but small. 

Writing out first set of equations as the linear system $\gamma''_i \cdot \log
x'' = c_i$ for the vector $\log x'',$ with entries the logarithms of the $x'',$
we see that this system must, for each $N\in\R,$ have a solution with all
entries less than $N$ with the $c_i$ bounded. Since the $\gamma''_i$ are
linearly independent, it follows that the same is true for \emph{any} $c_i.$
Thus in fact any solution of the equations in \eqref{E:binvar_boundary_form}
corresponds to a point in $G'.$ 
\end{proof}

In light of this result, we define the {\em boundary faces} of $D$ to be the
components $G'$ of the nonempty $D_G = \clos(D \cap G \setminus \pa G)$.  The
linear space $\bN_p G' := \bN_pD_G = \nullspace\set{\bd \gamma_i} \subset \bN_p G$
as in \eqref{E:binvar_bundles} can then be identified as the `b-normal space to
$G'$ as a boundary face of $D,$' consistent with the smooth case.  The
codimension of a boundary face $G' \subset D_G$ with respect to $D$ is
given by $\dim(\bN G')$ as expected, and we let $\cM_k(D)$ be the boundary
faces of codimension $k$ in this sense, equipping $\cM(D) = \bigcup_k \cM_k(D)$
with the order of reverse inclusion.

The intersection of the rational subspaces $\bN G',$ $G' \in \cM(D)$ with 
the monoids $\sigma_G$, $G \in \cM(X)$ gives the `basic monoidal complex' of
$D$.

\begin{prop} If $D \subset X$ is an interior binomial subvariety the monoids 
\begin{equation}
	\sigma_{G'} = \bN G' \cap \sigma_G ,\ G' \in \cM(D)
\label{E:binvar_monoids}
\end{equation}
where $G'$ is a component of $\clos(D \cap G \setminus \pa G)$
form a (not necessarily smooth) complex $\cP_D$ over $(\cM(D),
\leq)$ and there is a natural, injective morphism of complexes
\begin{equation}
	i_\natural : \cP_D \to \cP_X
	\label{E:bin_suvar_mon_complex}
\end{equation}
over $i_\# : \cM(D) \to \cM(X),$ where $i_\#(G') = G$ such that $G'
\subset D_G$ as above.
\label{P:bin_subvar_mon_complex}
\end{prop}

\begin{proof}
The monoids $\sigma_{G'}$ are clearly well-defined and toric; it suffices to
verify that they form a complex.  If $F' \subseteq G' \in \cM(D)$ and $p \in
F'$, then $F \subseteq G \in \cM(X)$.  Since $\bN_p G' = \bT_p D \cap \bN_p G$
and similarly for $\bN_p F'$, and since $\bN_p G \subseteq \bN_p F$, it follows
that $\bN_p G' \subseteq \bN_p F',$ giving an inclusion $\sigma_{G'} \subset
\sigma_{F'}$, which must be an isomorphism onto a face since it is the
intersection of $\sigma_{F}$ with a subspace.  Thus 
\[
	\cP_D = \set{\sigma_{G} \;;\; G \in \cM(D)}
\]
is a (complete, reduced) monoidal complex over $\cM(D)$, and the inclusions
$\sigma_{G} \subset \sigma_{i_\#(G)}$ produce a necessarily injective morphism
$i : \cP_D \to \cP_X$.
\end{proof}

A very special instance of an interior binomial subvariety is a `product-' or
{\em p-submanifold} (see \cite{melrosedifferential} for background).  This is a
smooth submanifold $D \subset X$ which meets all boundary faces of $X$
transversally, and which is covered by coordinate neighborhoods
$\big(U,(x,y)\big)$ in $X$ such that
\[
	D\cap U = \set{y_j = 0\;;\; j = 1,\ldots,\codim(Y)}.
\]
For a p-submanifold, it is evident that $\bN_p G' \equiv \bN_p G$ whenever $G'
\in \cM(D)$ is a component of $D_G$, $G \in \cM(X)$, and so the morphism
$i_\natural : \cP_D \to \cP_X$ consists of monoid isomorphisms $\sigma_{G'}
\stackrel{\cong}{\to} \sigma_G$.  In other words $i_\natural$ is a {\em local}
isomorphism, though it need not be a global one, since the $D_G$ may consist of
multiple components, and therefore  $i_\# : \cM(D) \to \cM(X)$ need not be
injective.

In fact this global issue of the failure of $i_\# : \cM(D) \to \cM(X)$ to be
injective arises as a technical obstruction to the resolution of a general
interior binomial variety $D \subset X$ by generalized blow-up of $X$ in the
next section.  Indeed, if there are multiple components $G'_i \subset D_G$,
there may be no way to refine $\sigma_G \in \cP_X$ in a way which appropriately
resolves the images $i_{\natural}(\sigma_{G'_i})$.  Fortunately one can always
pass to a `collar neighborhood' $X' \supset D$ in $X$ for which this
obstruction does not arise.  The following lemma guarantees the existence of
such a neighborhood; the proof follows directly from the local normal form
\eqref{E:bin_str_normal_form}. 

\begin{lem}\label{L:binvar_neighborhood} 
If $D \subset X$ is an interior binomial subvariety then there exists an open
submanifold $X' \subset X$ containing $D$ such that $i_\# : \cM(D) \to \cM(X')$
is injective so for each $G \in \cM(X')$ there is at most one connected
component of $D_G.$ 
\end{lem}

Finally, observe that other concepts can be extended from the `smooth' case of
manifolds to binomial subvarieties.

\begin{defn} 
If $D \subset X$ is a binomial subvariety, and $Y$ a manifold with corners then
a map $f : Y \to D$ is a b-map if it is a b-map in the smooth sense, i.e.\ as a
map $f : Y \to X$ with $f(Y)\subset D.$
\end{defn}

For such a b-map, the range of the b-differential $\bd f_\ast$ will lie in $\bT
D$ and if $G' \in \cM(D)$ is the highest codimension face such that $f(F)
\subset G'$ it follows that
\begin{equation}
	\bd f_\ast : \bN_p F \to \bN_{f(p)} G'.
	\label{E:f_normal_onto_binomial}
\end{equation}
Just as in Definition~\ref{D:basic_monoidal_complex} there is an induced map
\begin{equation}
	f_\natural : \cP_Y \to \cP_D
	\label{E:f_natural_onto_binomial}
\end{equation}
of monoidal complexes.

\section{Resolution of binomial subvarieties} \label{S:binvarres}

In this section we show that a carefully chosen smooth refinement of the monoidal
complex of the ambient manifold $X$ which also refines the monoidal complex
of an interior binomial subvariety $D \subset X$ leads to a blow-up under
which $D$ lifts to p-submanifold, in particular this resolves $D.$ The
resolution of $D$ so obtained depends essentially only on the choice of the
smooth refinement of the monoidal complex $\cP_D.$

The notion of a resolution is analogous to that of a generalized blow-up
defined earlier.

\begin{defn}\label{D:binvar_blowdown} If $D \subset X$ is an interior
  binomial subvariety then a manifold $Y$ with a b-map $f:Y\longrightarrow
  X$ with $f(Y)\subset D$ is a {\em resolution of} $D$ if $\bd f_\ast :
  \bT_p Y \to \bT_{f(p)}D$ is a bijection for all $p\in Y$ and $f :
  Y\setminus \pa Y \to D\setminus \pa D$ is a diffeomorphism.
\end{defn}

\begin{prop}
If $\beta : X_1 \to X$ is a generalized blow-down map between manifolds and
$D\subset X$ is an interior binomial subvariety then the lift (or proper
transform)
\begin{equation}
	\beta^{\#}(D)=\clos_{X_1}\big(\beta^{-1}(D\setminus\pa X)\big)
	\label{E:lift_of_D}
\end{equation}
is an interior binomial subvariety of $X_1.$ 
\label{P:binvar_lift}
\end{prop}

\begin{proof} 
Certainly $\beta^{\#}(D)\subset X_1$ is a closed subset. It is a smooth
embedded submanifold in the interior, so it suffices to show that it has a
local binomial structure at each boundary point. If $p\in\pa X_1\cap
\beta^{\#}(D)$ then by definition, $\beta(p)\in\pa X\cap D$ and $D$ has a local
binomial structure given by $f_i\in\cG(X).$ The pull-backs $f^*f_i\in\cG(X_1)$
define $\beta^{\#}(D)$ locally and their logarithmic differentials are
independent since the b-differential of $\beta$ is an isomorphism at each point.
\end{proof}

Let $\cR\longrightarrow \cP_X$ be a smooth refinement.  If $D\subset X$ is an
interior binomial subvariety then $\cR$ is {\em compatible with $\cP_D$} if
there is a subcomplex $\cR_D \subset \cR$ such that $\cR_D\longrightarrow
\cP_X$ factors through $\cP_D$ giving a commutative diagram
\begin{equation}
\begin{tikzpicture}[->,>=to,auto]
\matrix (m) [matrix of math nodes, column sep=1cm, row sep=1cm,  text depth=0.25ex]
{ \cR_D & \cR \\ \cP_D & \cP_X \\};
\path (m-1-1) edge node {$\subset$} (m-1-2); 
\path (m-1-2) edge  (m-2-2); 
\path (m-2-1) edge node {$i_\natural$} (m-2-2); 
\path (m-1-1) edge (m-2-1); 
\end{tikzpicture}
\label{E:compatible_refinement}
\end{equation}
where the vertical arrows are smooth refinements and the top is the
inclusion of a subcomplex.

\begin{prop} 
If $\cR\longrightarrow \cP_X$ is a smooth refinement which
is compatible with an interior binomial subvariety $D$ in the sense of
\eqref{E:compatible_refinement} then $D$ lifts to a p-submanifold $\wt D =
\beta^\#(D)$ of the generalized blow-up
\[
	\beta : [X; \cR] \to X.
\]
In particular $\beta : \wt D\longrightarrow D$ is a resolution of $D$ and
$\beta_\natural : \cP_{\wt D} \to \cP_D$ factors through an isomorphism
$\cP_{\wt D} \cong \cR_D$ of monoidal complexes.
\label{P:binomial_lift_psub}
\end{prop}

\begin{proof}
Consider an arbitrary $p' \in \tD = \beta^\#(D)$, and let $p = \beta(p') \in
D$.  There is a unique face $G' \in \cM(D)$ such that $p$ lies in the interior
of $G'$, and we set $G = i_\#(G') \in \cM(X)$.  Since $i_\natural : \cP_D \smallto \cP_X$ in \eqref{E:bin_suvar_mon_complex} and
$\beta_\natural : \cR \to \cP_X$ are both injective, we identify monoids $\sigma_{G'}$ and
$\tau \in \cR\pns{G}$ with their respective images in $\sigma_G$.

The point $p' \in \wt D \subset [X; \cR]$ lies in some coordinate chart
$U_{\tau} \subset \R^k_+\times (0,\infty)^{n-k}$ with coordinates $(t,y) =
(t_1,\ldots,t_k,y_1,\ldots,y_{n-k})$, for some $\tau \in \cR\pns{G}$, and we
can assume that the coordinates $y$ are pulled back identically from those on
$X$ in which $D$ has the local normal form \eqref{E:bin_str_normal_form}.  Thus
the equations $\set{y_j = 0\;;\; j =d'+1,\ldots,d}$ are the same in $U_{\tau}$.
We consider the lift of the other equations, $x^{\gamma_i} = 1$ under $\beta$.
On $U_{\tau}$, $\beta$ has the form $\beta : t \to t^\mu = x,$ and we obtain
\[
	\beta^\ast\pns{x^{\gamma_i}} 
	 = \pns{t^\mu}^{\gamma_i} 
	 = t^{\beta_i} = 1,
\]
where $\beta_i = \mu \gamma_i$.  Thus near $p'$, $\wt D \subset [X; \cR]$ has
the local binomial structure
\[
	\wt D = \set{t^{\beta_i} = 1, y_j = 0}.
\]
We will show that the $\beta_i$ are each non-negative or non-positive.

Indeed, the compatibility assumption implies that, for all $\tau \in
\cR\pns{G}$, either $\tau \in \cR_D(G')$ and hence $\tau \subset \sigma_{G'}$,
or $\tau \in \cR \setminus \cR_D$; in either case the intersection $\tau \cap
\sigma_{G'}$ must be a face of $\tau$.  Since $\sigma_{G'} = \sigma_G \cap
\nullspace\set{\bd \gamma_i}$, it follows that
\[
	\pair{\gamma_i dx/x, \tau} \geq 0 \ \text{or}\ \pair{\gamma_i dx/x,\tau} \leq 0, \quad \text{for each}\  \gamma_i.
\]
In other words, each vector $\gamma_i dx/x$ is either non-positive or non-negative
with respect to $\tau \in \cR(G)$.  Let us assume non-negative; the other case
is similar.

From this it follows that no $\beta_i$ is indefinite, since for any
$a \in \R^k$, $a_i > 0$, 
\[
	\pair{\beta_i, a} = \pair{\beta_i\, dt/t, a\, t\pa_t} = \pair{\mu\, \gamma_i\, dx/x, a\, t\pa_t} = \pair{\gamma_i dx/x, \mu^\transpose(a\,t \pa_t)} \geq 0
\]
as $\mu^\transpose(a t\pa_t) \in \tau$.  In light of
Lemma~\ref{L:faces_met_by_D}, $\wt D$ only meets boundary faces with respect to
which the $\beta_i$ are zero, hence $\wt D$ has a covering by binomial
structures such that
\[
	\wt D = \set{y_j = 0\;;\;  j=1,\ldots,d}
\]
and is therefore a p-submanifold.

Since $\wt D$ is a p-submanifold, for any $\wt G \in \cM(\wt D)$, $\sigma_{\wt
G} \cong \tau$ for some $\tau \in \cR \cong \cP_{[X; \cR]}$.  Finally, it
follows from the fact that $\wt D$ is the lift of $D$ that $\sigma_{\wt G} =
\tau$ is actually in $\cR_D$; hence $\cP_{\wt D} \to \cR_D$ is a local
isomorphism onto its image and $\beta : \wt D \to D$ is a resolution since it
is a diffeomorphism on interiors.
\end{proof}
\noindent The resolution of binomial ideals in polynomial and power series
rings by toric methods are well-known, see \cite{teissier51monomial} for a good
overview.  The previous proposition can be seen as an extension of this theory
to the interior binomial subvarieties we have been discussing.

Next we show that we can obtain a unique resolution of $D$ realizing any smooth
refinement $\cR_D \to \cP_D$; in particular the resolution so obtained is essentially
independent of the ambient manifold $X$.

\begin{thm}[Resolution of binomial varieties]
If $D \subset X$ is an interior binomial subvariety, then for every smooth
refinement $\cR_D \to \cP_D$ there exists a resolution
\[
	\beta : [D; \cR_D] \to D
\]
which realizes the refinement in the sense that
\[
	\beta_\natural : \cP_{[D; \cR_D]} \cong \cR_D \to \cP_D,
\]
and $[D; \cR_D]$ is unique up to diffeomorphism.

If $f : Y \to D$ is a b-map from a smooth manifold, and if
$f_\natural : \cP_Y \to \cP_D$ factors through $\cR_D$, then $f$ factors
through a unique b-map $\wt f : Y \to [D; \cR_D]$.
\label{T:binvar_resolution}
\end{thm}
\noindent
The notation is meant to suggest that this is in some sense the generalized
blow-up in the category of (differentiable) `binomial varieties' where the
objects are treated intrinsically.  Indeed, we believe that such a category
exists and that generalized blow-up extends to include arbitrary refinements,
not necessarily smooth.  
\begin{proof}
Assume that $i_\# : \cM(D) \to \cM(X)$ is injective, passing if necessary to a
collar neighborhood $X' \subset D$ as in Lemma \ref{L:binvar_neighborhood}. By
Proposition \ref{P:binomial_lift_psub} it suffices to show that, given $\cR_D
\to \cP_D$, there exists a refinement $\cR_X \to \cP_X$ extending $\cR_D$ (that
is, containing $\cR_D$ as a subcomplex), and that the resulting resolution
$\beta^\#(D) \subset [X; \cR_X]$ is well-defined, independent of the choice of
such extension.

For the first step the planar refinement of
Proposition~\ref{P:planar_refinement_complex} $\cS(\cP_X, \cP_D) \to \cP_X$
is a (not necessarily smooth) refinement containing $\cP_D$ as a subcomplex.
Then by Lemma \ref{L:extension_of_refinement} the refinement $\cR_D \to \cP_D$
can be extended to a smooth refinement $\cR_X \to \cS(\cP_X, \cP_D)$, and the
composition $\cR_X \smallto \cS(\cP_X, \cP_D) \smallto \cP_X$ is therefore an
extension of $\cR_D$.

Next suppose $\cR_1$ and $\cR_2$ are smooth refinements of $\cP_X$ extending
$\cR_D$, and set
\[
	\wt D_i := \beta_i^\#(D) \subset [X; \cR_i]\quad i = 1,2.
\]
The blow-down $\beta_1 : \wt D_1 \to D$, considered as a map to $X$, lifts by
Theorem \ref{T:lifting_b_maps} to a b-map to $[X; \cR_2]$ whose range lies in
$\wt D_2$, and vice versa.  Thus we obtain b-maps
\[
	\wt D_1 \leftrightarrow \wt D_2
\]
which are generalized blow-down maps and in fact diffeomorphisms by Proposition
\ref{P:genblowdown_identity} since $\cP_{\wt D_1} \cong \cR_D \cong \cP_{\wt
D_2}$.  

Thus the lift $\wt D \subset [X; \cR_X]$ is independent of the extension
$\cR_X$ of $\cR_D$ up to diffeomorphism, and we define $[D; \cR_D] = \wt D$ to
be any such lift.

Recall that a b-map $f : Y \to D$ induces a monoidal complex morphism
$f_\natural : \cP_Y \to \cP_X$ which factors through $\cP_D$ as in
\eqref{E:f_natural_onto_binomial}.  If in addition $f_\natural$ factors through
$\cR_D \to \cP_D$, then it follows that $f_\natural : \cP_Y \to \cP_X$ factors
through any extension $\cR_X$, and $f$ admits a unique lift $f' : Y \to [X;
\cR_X]$ by Theorem~\ref{T:lifting_b_maps}.  The image of $Y \setminus \pa Y$
under $f'$ lies in $\wt D \setminus \pa [X; \cR_X]$, and therefore by
continuity $f'(Y) \subset \wt D$ so $f' : Y \to \wt D$ is a b-map.  Since
the resolutions coming from different extensions are diffeomorphic, it follows
from the naturality of the lifted b-maps that 
\[
	f' : Y \to [D; \cR_D]
\]
is well-defined, independent of the choice of extension.
\end{proof}

Observe that if $\cP_D$ is already smooth, then the above suggests that $D$ is
in some sense already a smooth manifold, though it may not be nicely embedded
in $X$.  Indeed, it follows from the proof of Theorem~\ref{T:binvar_resolution} that 
there exists a refinement $\cR_X \to \cP_X$ which is trivial on $\cP_D$, and this
gives a `minimal' resolution of $D$ which is universal in this case.

\begin{thm}
If $D \subset X$ is an interior binomial variety and $\cP_D$ is smooth, then
there exists a {\em universal} resolution $[D; \cP_D] \to D$, with the property
that any b-map $f : Y \to D$ factors uniquely through $[D; \cP_D]$.  In
particular, any other resolution of $D$ is a generalized blow-up of this
universal resolution.
\label{T:universal_resolution}
\end{thm}
\begin{proof}
From Theorem~\ref{T:binvar_resolution}, there exists a unique resolution $[D;
\cP_D]$ coming from the trivial refinement $\id : \cP_D \to \cP_D$.  Since any
b-map $f : Y \to D$ induces a morphism $f_\natural : \cP_Y \to \cP_D$ which
necessarily factors through this trivial refinement, $f$ necessarily factors
through $[D; \cP_D]$.  If $f$ itself is a resolution, it follows that the lift
$f' : Y \to [D; \cP_D]$ is also diffeomorphic on the interiors, with
bijective b-differential; in other words, it is a generalized blow-down map
onto the manifold $[D; \cP_D]$, and hence a blow-up of this space by Theorem
\ref{T:characterization}.
\end{proof}

In fact, though we shall not use this below, it is possible to show that in
case $\cP_D$ is smooth, the spaces $D$ and $[D; \cP_D]$ are actually
homeomorphic.  Indeed, it is straightforward to show that the map $[D; \cP_D]$
is a diffeomorphism on the interiors of boundary faces, hence globally
bijective, then since it is continuous and proper it has a continuous inverse.
Thus one can regard Theorem~\ref{T:universal_resolution} as giving a natural
smooth structure on $D$ itself, though this is {\em not} generally equal to the
restriction of the smooth structure on $X$, as illustrated by the example $D =
\set{x_1^2 = x_2^3} \subset \R^2_+$, whose universal resolution is the usual
one: $[D; \cP_D] = \R_+ \ni t \mapsto (t^3,t^2) \subset \R^2_+$.

\section{Fiber products}\label{S:fiber}

We finally bring the theory of the last two sections to bear on the question of
fiber products of manifolds with corners.  Recall that in the category of
manifolds without boundary, smooth fiber products do not generally exist.  A
sufficient condition in this context is {\em transversality}; namely two smooth
maps $f_i : X_i \to Y$, $i = 1,2$ are transversal if $(f_1)_\ast(T_{p_1} X_1) +
(f_2)_\ast (T_{p_2} X_2) = T_q Y$ for all pairs $(p_1,p_2) \in X_1\times X_2$
such that $f_1(p_1) = f_2(p_2) = q$, and then 
\begin{equation}
	X_1 \times_Y X_2 = \set{(p_1,p_2) \in X_1\times X_2 \;;\; f_1(p_1) = f_2(p_2)}  \subset X_1 \times X_2
	\label{E:set_fiber_product}
\end{equation}
is a smooth manifold.  

We show that the analogous condition of b-transversality in the category of
manifolds with corners implies that $X_1 \times_Y X_2 \subset X_1 \times X_2$
is a union of binomial subvarieties, each of which is interior to some product
of faces.  The theory we have developed then gives sufficient conditions for a
fiber product (with the required universal properties) to exist in the category
of manifolds with corners, and gives a coherent system of resolution by
generalized blow-up even when these conditions are not satisfied.

Note that even when a fiber product does exist in the category of manifolds
with corners, it is not generally equal as a set to \eqref{E:set_fiber_product}.
Indeed, one step in the construction of the fiber product is to separate the
binomial subvarieties of \eqref{E:set_fiber_product} by taking their
{\em disjoint} union, which is already a kind of resolution.  Furthermore, as
noted in the last section, even when a binomial subvariety has a smooth
monoidal complex, it need not be smoothly embedded in its ambient manifold
$X_1\times X_2$.

\begin{defn}
We say two b-maps $f_i : X_i \to Y$, $i = 1,2$ are {\em b-transversal} if for
all points $p_i \in X_i$ such that $f_1(p_1) = f_2(p_2) = q \in Y$, 
\[
	\bd (f_1)_\ast\pns{\bT_{p_1} X_1} + \bd (f_2)_\ast\pns{\bT_{p_2} X_2} = \bT_q Y.
\]
\label{D:b-transversal}
\end{defn}

\begin{prop}[Iterated transversality]
If $f_i : X_i \to Y$, $i = 1,2$ are b-transversal, then for every pair of $F_i \in \cM(X_i)$
such that $(f_1)_\#(F_1) = (f_2)_\#(F_2) = F \in \cM(Y)$, the induced b-maps
\[
	(f_i)_{|F_i} : F_i \to F \quad \text{are b-transversal.}
\]
\label{P:iterated_b_transversality}
\end{prop}
\begin{proof}
Choose $p_i \in F_i$ such that $f_1(p_1) = f_2(p_2) = q$.  Suppose $G_i
\subseteq F_i$ are the maximal codimension boundary faces containing $p_i$ (in
particular $p_i \in G_i \setminus \pa G_i$), and set $\bN_{p_i}(G_i; F_i) =
\bN_{p_i} G_i / \bN_{p_i} F_i$.  These are the b-normal bundles to the $G_i$ in
the manifolds $F_i$.  The b-tangent spaces of the $G_i$ are the quotients
$\bT_{p_i} G_i = \bT_{p_i} X_i/\bN_{p_i} G_i$ with respect to the natural
inclusions $\bN_{p_i} G_i \hookrightarrow \bT_{p_i} X_i$, and since the $p_i$
lie in the interior of the $G_i$, $\bT_{p_i} G_i = T_{p_i} G_i$.

Using a metric to replace the quotients by orthogonal decompositions,
\[
	\bT_{p_i} X_i = T_{p_i} G_i \oplus \bN_{p_i} (G_i; F_i) \oplus \bN_{p_i} F_i.
\]
The last two factors constitute $\bN_{p_i} G_i$, while the first two constitute
$\bT_{p_i} F_i$.  The b-differentials $\bd (f_i)_\ast$ have the form
\[
	\bd (f_i)_\ast = \begin{pmatrix} \ast & 0 & 0 \\\ast & \ast & 0 \\ \ast & 0 & \ast  \end{pmatrix} : 
	\begin{array}{l} T_{p_i} G_i \\\oplus\bN_{p_i}(G_i; F_i)\\\oplus \bN_{p_i} F_i \end{array}
	\to \begin{array}{l} T_q G \\\oplus \bN_q (G; F) \\\oplus \bN_q F\end{array}
\]
with respect to the quotients, where $G = (f_i)_\#(G_i)$, $i = 1,2$.  In
particular, the only vectors with image in $\bT_q F = T_q G \oplus \bN_q (G;
F)$ must lie in the first two factors, namely $T_{p_i} G_i\oplus \bN_{p_i}(G_i;
F_i) \equiv \bT_{p_i} F_i$, thus if 
\[
	\bd (f_1)_\ast + \bd (f_2)_\ast : \bT_{p_1} X_1 \times \bT_{p_2} X_2 \to T_q Y 
\]
is surjective, then 
\[
	\bd \pns{(f_1)_{|F_1}}_\ast + \bd \pns{(f_2)_{|F_2}}_\ast : \bT_{p_1} F_1 \times \bT_{p_2} F_2 \to \bT_q F 
\]
must be surjective.
\end{proof}

\begin{prop}
If $f_i : X_i \to Y$, $i = 1,2$ are b-transversal maps then the set-theoretic
fiber product \eqref{E:set_fiber_product} is a union of interior binomial subvarieties
\[
D(F_1,F_2)= \clos\pns{F_1\times_Y F_2 \setminus \pa \pns{ F_1 \times F_2}} \subset F_1\times F_2 \quad F_i \in \cM(X_i).
\]
\label{P:fiber_prod_is_binvar}
\end{prop}

\begin{proof} Every point in $X_1 \times_Y X_2$ lies in the interior of some $F_1\times F_2$,
hence in $D(F_1,F_2).$  It suffices to verify that the $D(F_1,F_2) \subset
F_1 \times F_2$ are either empty or interior binomial subvarieties. So we
may restrict attention to the case $F_i = X_i$, and assume that $D(X_1,X_2)$ is nonempty.

Suppose $X_i \ni p_i \mapsto q \in Y$, and choose coordinates $(x',y')$
centered at $p_1$, $(x'',y'')$ centered at $p_2$, and $(\overline x,\overline
y) = (\overline x_1,\ldots,\overline x_k,\overline y_{k+1},\ldots,\overline
y_n)$ centered at $q$.  The maps $f_i$ have the local form
\[
\begin{aligned}
	f_1 &: (x',y') \mapsto \big(a_1(x',y') (x')^{\nu_1}, b_1(x',y')\big) = (\overline x,\overline y), \ \text{and}\\
	f_2 &: (x'',y'') \mapsto \big(a_2(x'',y'') (x'')^{\nu_2}, b_2(x'',y'')\big) = (\overline x,\overline y).
\end{aligned}
\]
Near $(p_1,p_2) \in X_1 \times X_2$, $(x,y) = (x',x'', y', y'')$ are local
coordinates in terms of which
\[
D(X_1,X_2) \subset \set{a_1(x',y')\,\pns{x'}^{\nu_1} =
a_2(x'',y'')\,\pns{x''}^{\nu_2},\; b_1(x',y') = b_2(x'',y'')}
\]
which can be written in the form
\[
	D(X_1,X_2) \subset \set{c_i\,x^{\gamma_i} = 1\;;\; i =1,\ldots,n}, \quad 0 < c_i\in \CI(X_1\times X_2)
\]
where 
\[
	(c_i, \gamma_i) = \begin{cases} \big((a_1)_i/(a_2)_i, (\nu_1\oplus 0)_i - (0 \oplus \nu_2)_i\big) & i =1,\ldots,k, \\
					\Big(\exp\big((b_1)_i - (b_2)_i\big), 0\Big) & i =k+1,\ldots,n. \end{cases}
\]

Thus it only remains to check the independence of the logarithmic differentials of the $c_i
x^{\gamma_i}.$ Consider the b-map $f_1\times f_2 : X_1 \times X_2 \to Y\times
Y.$ There is an exact sequence
\[
0 \smallto \bT_q Y \stackrel{\bd \Delta_\ast}{\smallto} \bT_{(q,q)}\pns{Y\times Y} \smallto \bT_q Y  \smallto 0
\]
where $\Delta : Y \to Y\times Y$ is the diagonal inclusion and the
subsequent map is the difference from $\bT_{(q,q)}\pns{Y\times Y}=\bT_q Y\times\bT_q Y.$
The condition of
b-transversality condition means that $\bd\pns{f_1\times f_2}^\ast$ is
injective as a map
\[
	\bd\pns{f_1\times f_2}^\ast : \bT^*_q Y \to \bT_{(p_1,p_2)}^\ast(X_1 \times X_2).
\]
This can be identified with the map
\[
	\bd f_1^\ast - \bd f_2^\ast : \bT^\ast_q Y \to \bT^\ast_{(p_1,p_2)}(X_1\times X_2)
\]
which is similarly injective.  Taking the coordinate basis $\set{d\overline
x_i/\overline x_i, d \overline y_j}$ for $\bT^\ast_q Y$, we obtain that
\[
	\set{(\bd f_1^\ast - \bd f_2^\ast) d\overline x_i/\overline x_i, (\bd f_1^\ast - \bd f_2^\ast) d \overline y_j}
\]
is independent.  Observe however that
\[
\begin{aligned}
	(\bd f_1^\ast - \bd f_2^\ast) d\overline x_i/\overline x_i &= d\log \pns{(a_1 {x'}^{\nu_1})_i} - d\log\pns{(a_2{x''}^{\nu_2})_i} 
	\\&= d\log \pns{c_i x^{\gamma_i}},\quad  i =1,\ldots,k
\end{aligned}
\]
and
\[
\begin{aligned}
	(\bd f_1^\ast - \bd f_2^\ast) d\overline y_i &= d(b_1)_i - d (b_2)_i \\&= d\log\Big(\exp\big((b_1)_i - (b_2)_i\big)\Big) 
	\\&= d\log \pns{c_i x^{\gamma_i}}, \quad  i =k+1,\ldots,n.
\end{aligned}
\]
We conclude that $D(X_1,X_2)$ has a covering by local binomial structures and
is therefore an interior binomial subvariety.
\end{proof}

While $X_1\times_Y X_2$ may therefore be a complicated and quite singular
space, any smooth maps factoring through it must actually factor through one of
the subvarieties $D(F_1,F_2)$.

\begin{prop} If $g_i : Z \to X_i$ are b-maps from a connected smooth manifold such that
$f_1 \circ g_1 = f_2 \circ g_2$ then the maps $g_i : Z \to X_i$
factor through a unique b-map $h : Z \to D(F_1,F_2)$ for some $D(F_1,F_2)$:
\[
\begin{tikzpicture}[->,>=to,auto]
\matrix (m) [matrix of math nodes, column sep=1cm, row sep=1cm, text depth=0.25ex]
{ Z & & \\  &D(F_1,F_2) & X_2 \\ & X_1 & Y. \\};
\path (m-2-2) edge node {$\pi_2$} (m-2-3);
\path (m-2-3) edge node  {$f_2$} (m-3-3);
\path (m-2-2) edge node  {$\pi_1$} (m-3-2);
\path (m-3-2) edge node  {$f_1$} (m-3-3);
\path (m-1-1) edge [bend left] node  {$g_2$} (m-2-3);
\path (m-1-1) edge [bend right] node  {$g_1$} (m-3-2);
\path (m-1-1) edge [dashed] node  {$h$} (m-2-2);
\end{tikzpicture}
\]

\label{P:fiber_components_mapping}
\end{prop}
\begin{proof}
Let $F_1$, $F_2$ be the minimal (largest codimension) faces such that $g_i : Z
\to F_i$ are {\em interior} b-maps.  Then $g_1\times g_2 : Z \to F_1\times F_2
\subset X_1 \times X_2$ is an interior b-map, and as such
\[
	g_1 \times g_2 : Z \setminus \pa Z \to F_1\times F_2 \setminus \pa\pns{F_1\times F_2}.
\]

On the other hand, as a map of sets, $g_1\times g_2$ factors through (the
set-theoretic) $F_1\times_Y F_2$ by the assumption that $f_1\circ g_1 = f_2
\circ g_2$.  Finally, it follows by continuity and taking the closure of the
intersection with the interior of $F_1 \times F_2$ that
\[
	h= g_1\times g_2 : Z \to D(F_1,F_2) \subset X_1\times X_2
\]
which is by definition a b-map into an interior binomial subvariety.
\end{proof}

Note that the monoidal complex $\cP_{D(F_1,F_2)}$ consists of monoids of the
form $\sigma_{G_1} \times_{\sigma_G} \sigma_{G_2}$, where $\sigma_{G_i} \in
\cP_{F_i}$, and $G = (f_1)_\#(G_1) \cap (f_2)_\#(G_2).$  Indeed, if $F$ is a
boundary face of $D(F_1,F_2)$, given by a component of $D(F_1,F_2) \cap
G_1\times G_2$ for $G_1\times G_2 \in \cM(X_1\times X_2)$, then 
\[
	\sigma_F = \sigma_{G_1\times G_2} \cap \bN F = \sigma_{G_1} \times_{\sigma_G} \sigma_{G_2}.
\]
It is generally {\em not} true that $\cP_{D(F_1,F_2)}$ is equal to
$\cP_{F_1}\times_{\cP_Y} \cP_{F_2}$ as might at first be expected,
since $D(F_1,F_2) \cap G_1\times G_2$ may be empty or may have multiple
components, while $\sigma_{G_1} \times_{\sigma_G} \sigma_{G_2}$ is nontrivial
and appears exactly once in $\cP_{F_1} \times_{\cP_Y} \cP_{F_2}$.  However,
this is true in some cases, as in example \ref{X:fib_prod_two_blowdowns} below.

\begin{thm}[Existence of smooth fiber products]
If $f_i : X_i \to Y$ are b-transverse maps and if
in addition $\sigma_{F_1}\times_{\sigma_G} \sigma_{F_2}$ is a smooth monoid
whenever $G = (f_1)_\#(F_1) \cap (f_2)_\#(F_2) \in \cM(Y)$ then
\[
	\wt {X_1\times_Y X_2}= \bigsqcup_{F_1,F_2} [D(F_1,F_2); \cP_{D(F_1,F_2)}]
\]
is a union of smooth manifold with corners and is the universal fiber
product.
\label{T:smooth_fiber_products}
\end{thm}
\noindent
Observe that $\wt {X_1\times_Y X_2}$ is not equal to $X_1\times_Y X_2$.  In
categorical language, the forgetful functor from manifolds with
corners to sets does not preserve fiber products.  However there is a unique
map $\wt{X_1\times_Y X_2} \to X_1\times_Y X_2.$
\begin{proof}
It follows from Theorem~\ref{T:universal_resolution} that $\wt{X_1\times_Y
X_2}$ is a smooth manifold with corners, which has the requisite universal
property by Proposition~\ref{P:fiber_components_mapping}.  Indeed, any b-maps
$g_i : Z \to X_i$ from a manifold $Z$ which commute with the $f_i$ factor
through a unique map $h : Z \to D(F_1,F_2)$ for some $D(F_1,F_2)$, and this has
a unique lift to the universal resolution $[D(F_1,F_2); \cP_{D(F_1,F_2)}]$.
Finally, the uniqueness of $\wt {X_1 \times_Y X_2}$ up to diffeomorphism
follows from the universality.
\end{proof}

In general, even for b-transverse maps, a smooth fiber product does not
exist.  Nevertheless, the space
\[
	\bigsqcup_{F_1,F_2} D(F_1,F_2)
\]
serves as a fiber product in the category of interior binomial
subvarieties with its attendant resolution theory, which we summarize below.

\begin{thm}
Let $f_i : X_i \to Y$, $i = 1,2$ be b-transverse maps and set
$D = \bigsqcup_{F_i \in \cM(X_i), i = 1,2} D(F_1,F_2)$
\begin{enumerate}[{\normalfont (a)}]
\item \label{I:fibprod1}
For every smooth refinement $\cR_D \to \cP_D$ there is a smooth manifold with
corners $[D; \cR_D]$ with maps $h_i : [D; \cR_D] \to X_i$ forming a commutative
square with $f_i : X_i \to Y$:
\[
\begin{tikzpicture}[->,>=to,auto]
\matrix (m) [matrix of math nodes, column sep=1cm, row sep=1cm, text depth=0.25ex]
{{[D;\cR_D]} & X_2 \\  X_1 & Y. \\};
\path (m-1-1) edge [dashed] node {$h_2$} (m-1-2);
\path (m-1-1) edge [dashed] node {$h_1$} (m-2-1);
\path (m-2-1) edge node {$f_1$} (m-2-2);
\path (m-1-2) edge node {$f_2$} (m-2-2);
\end{tikzpicture}
\]
\item \label{I:fibprod2}
Any two such resolutions $[D; \cR_i]$ $i = 1,2$ have a mutual smooth
resolution, which is to say a third manifold $[D; \cR_0]$ with maps $[D; \cR_0]
\to [D; \cR_i]$, $i = 1,2$ forming commutative diagrams with the maps $[D;
\cR_i] \to X_j$.
\[
\begin{tikzpicture}[->,>=to,auto]
\matrix (m) [matrix of math nodes, column sep=1cm, row sep=1cm, text depth=0.25ex]
{ {[D; \cR_0]} & {[D; \cR_2]}  & \\ {[D; \cR_1]} & & X_2 \\ & X_1 & Y \\};
\path (m-1-1) edge [dashed] (m-1-2); 
\path (m-1-1) edge [dashed] (m-2-1); 
\path (m-2-1) edge node [pos=0.3] {$h_{12}$} (m-2-3); 
\path (m-2-1) edge node {$h_{11}$} (m-3-2); 
\path (m-1-2) edge node [pos=0.3] {$h_{21}$} (m-3-2); 
\path (m-1-2) edge node {$h_{22}$} (m-2-3); 
\path (m-2-3) edge node {$f_2$}(m-3-3); 
\path (m-3-2) edge node {$f_1$} (m-3-3); 
\end{tikzpicture}
\]
\item \label{I:fibprod3}
If $Z$ is a manifold with maps $g_i : Z \to X_i$ such that $f_1 \circ g_1 = f_2
\circ g_2$, and if the morphism $\cP_Z \to \cP_D$ factors through $\cR_D \to
\cP_D$, then there is a unique map $g : Z \to [D; \cR_D]$ such that $h_i \circ g = g_i$.
\[
\begin{tikzpicture}[->,>=to,auto]
\matrix (m) [matrix of math nodes, column sep=1cm, row sep=1cm, text depth=0.25ex]
{ Z & & \\  & {[D;\cR_D]} & X_2 \\ & X_1 & Y \\};
\path (m-1-1) edge [dashed] node {$g$} (m-2-2);
\path (m-2-2) edge node {$h_2$} (m-2-3);
\path (m-2-2) edge node {$h_1$} (m-3-2);
\path (m-3-2) edge node {$f_1$} (m-3-3);
\path (m-2-3) edge node {$f_2$} (m-3-3);
\path (m-1-1) edge [bend left,pos=0.6] node  {$g_2$} (m-2-3);
\path (m-1-1) edge [bend right] node  {$g_1$} (m-3-2);
\end{tikzpicture}
\]
\item \label{I:fibprod4}
Given a manifold $Z$ with maps $g_i : Z \to X_i$ such that $f_1 \circ g_1 = f_2 \circ g_2$, 
for any resolution $[D; \cR_D] \to D$, there exists a generalized blow-up $\beta : [Z; \cR] \to Z$ 
and a unique map $g : [Z; \cR] \to [D; \cR_D]$ such that $h_i \circ g = g_i \circ \beta$:
\[
\begin{tikzpicture}[->,>=to,auto]
\matrix (m) [matrix of math nodes, column sep=1cm, row sep=1cm, text depth=0.25ex]
{ [Z; \cR] & & \\ Z & {[D;\cR_D]} & X_2 \\ & X_1 & Y \\};
\path (m-1-1) edge [dashed] node {$g$} (m-2-2);
\path (m-1-1) edge [dashed] node {$\beta$} (m-2-1);
\path (m-2-2) edge node {$h_2$} (m-2-3);
\path (m-2-2) edge node {$h_1$} (m-3-2);
\path (m-3-2) edge node {$f_1$} (m-3-3);
\path (m-2-3) edge node {$f_2$} (m-3-3);
\path (m-2-1) edge [bend left] node  {$g_2$} (m-2-3);
\path (m-2-1) edge [bend right] node  {$g_1$} (m-3-2);
\end{tikzpicture}
\]
\end{enumerate}
\label{T:singular_fiber_product}
\end{thm}
\noindent
For a manifold $Z$ with maps $g_i : Z \to X_i$, it is not generally true that
there exists a resolution $[D; \cR_D] \to D$ through which $Z$ factors, without
blowing up the domain.

\begin{proof}
The smooth manifold $[D; \cR_D]$ is well-defined by the results in Section
\ref{S:binvarres}, giving \eqref{I:fibprod1}.  \eqref{I:fibprod2} follows from
the existence of mutual refinements, letting $\cR_0$ be a smooth refinement of
$\cR_1 \times_{\cP_D} \cR_2.$  \eqref{I:fibprod3} follows from
Theorem \ref{T:lifting_b_maps}, and \eqref{I:fibprod4} follows from Theorem
\ref{T:blowing_up_domain}, letting $\cR$ be a smooth refinement of $\cP_Z
\times_{\cP_D} \cR_D$.
\end{proof}

Some of the situations we have already considered are interesting examples of fiber products.
\begin{ex}[A blow-up and a binomial subvariety]
Let $Y$ be a smooth manifold which is included in $X$ as an interior binomial
variety, and let $X_1 = [X; \cR] \to X$ be a generalized blow-down.  Since $i :
Y \hookrightarrow X$ is injective, the fiber product $X_1 \times_X Y$ can be
identified with $\beta^{-1}(Y) \subset X_1$.  The interior subvariety
$D(Y,X_1)$ is just the lift/proper transform $\beta^\#(Y)$, which we showed to
be an interior binomial subvariety in Section \ref{S:binvarres}.  Observe,
however that $\beta^{-1}(Y)$ generally contains other subvarieties as well,
namely $F \cap \beta^{-1}(G)$, where $F \in \cM(X_1)$ and $G \in \cM(D)$.  See
Figure~\ref{F:fibprod}.
\label{X:fib_prod_blowup_and_binom}
\end{ex}

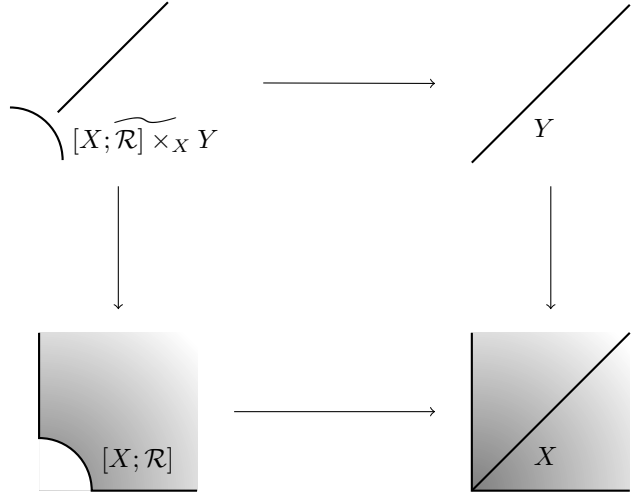
\begin{figure}[tb]
\begin{tikzpicture}
\matrix (m) [matrix of nodes, column sep=3cm, row sep=2cm]
{ 
 \begin{tikzpicture}[scale=0.7]
 \draw[thick] (0,1) to[bend left=45] coordinate[midway] (c) (1,0);
 \draw[thick] (c)+(0.2,0.2) -- (3,3);
 \path (1,1) node[below right] {$\wt{[X;\cR] \times_X Y}$};
 \end{tikzpicture}
&

 \begin{tikzpicture}[scale=0.7]
 \draw[thick] (0,0) -- (3,3);
 \path (1,1) node[below right] {$Y$};
 \end{tikzpicture}
\\
 \begin{tikzpicture}[scale=0.7]
 \begin{scope} \clip (0,0) rectangle (3,3);
   	\shade[shading=radial] (-4,-4) rectangle (4,4);
 \end{scope}
 \fill[color=white] (0,1) to[bend left=45] (1,0) -- (0,0) --cycle;
 \draw[thick] (0,3) -- (0,1) to[bend left=45] (1,0) -- (3,0);
 \path (1,1) node[below right] {$[X; \cR]$};
 \end{tikzpicture}
& 
 \begin{tikzpicture}[scale=0.7]
 \begin{scope} \clip (0,0) rectangle (3,3);
   	\shade[shading=radial] (-4,-4) rectangle (4,4);
 \end{scope}
 \draw[thick] (0,3) -- (0,0) -- (3,0);
 \draw[thick] (0,0) -- (3,3);
 \path (1,1) node[below right] {$X$};
 \end{tikzpicture}
\\};
\path (m-1-1) edge[->,>=to,shorten >=10pt,shorten <=10pt] (m-1-2);
\path (m-1-1) edge[->,>=to,shorten >=5pt,shorten <=5pt] (m-2-1);
\path (m-1-2) edge[->,>=to,shorten >=5pt,shorten <=5pt] (m-2-2);
\path (m-2-1) edge[->,>=to,shorten >=10pt,shorten <=10pt] (m-2-2);
\end{tikzpicture}

\caption{The fiber product of a generalized blow-down and the inclusion of 
a manifold as an interior binomial subvariety.}
\label{F:fibprod}
\end{figure}

\begin{ex}[Two blow-ups]
We leave it as an exercise for the reader to show that the fiber product of
two blow-down maps $\beta_i : [X; \cR_i] \to X$, $i = 1,2$ is a subvariety with
the monoidal complex $\cR_1 \times_{\cP_X} \cR_2$.  By Proposition
\ref{P:fiber_products_complexes} this is a (not necessarily smooth) refinement
of $\cP_X$, which can be identified with the intersection complex $\cR_1 \cap
\cR_2 = \set{\sigma_{F_1} \cap \sigma_{F_2} \subset \sigma_F \;;\; \sigma_{F_i}
\in \cR_i(F)}$.  While not smooth in general, one can construct nontrivial
examples where $\cR_1 \times_{\cP_X} \cR_2$ is smooth, which therefore give
nontrivial instances of Theorem \ref{T:smooth_fiber_products}.
\label{X:fib_prod_two_blowdowns}
\end{ex}

\begin{ex}[Joyce's fiber products]
In \cite{joyce2009manifolds}, Joyce proposes a category of manifolds with
corners in which the morphisms are what might be called `simple b-maps' $f : X
\to Y$ which have the property that whenever
\[
	f^\ast \cI_H = \prod_{G \in \cM_1(X)} \cI_G^{\alpha(H,G)}, \quad H \in \cM_1(Y)
\]
the $\alpha(\cdot,\cdot)$ are either zero or one, and furthermore for every $H$
there is at most one $G$ such that $\alpha(H,G) \neq 0$.  As a consequence
the morphisms
\[
	f_\natural : \sigma_F \to \sigma_{f_\#(F)}
\]
are always injective, with images which have orthogonal generators $\set{v_i}$
all of whose components are zero or one.

Joyce defines a transversality condition which is essentially equivalent to
b-transversality (though he does not use the b-differential, his is an iterated
transversality condition on boundary faces analogous to Proposition
\ref{P:iterated_b_transversality}), and shows that for two transversal maps in
his sense, the set-theoretic fiber product is a manifold with corners.  

This also follows from Theorem \ref{T:smooth_fiber_products}.  Indeed, for
b-transversal simple b-maps, the monoids $\sigma_{F_1} \times_{\sigma_G}
\sigma_{F_2}$ can be identified with the intersections $\sigma_{F_1} \cap
\sigma_{F_2} \subset \sigma_G$ by injectivity, and with the properties above,
these intersections are always smooth.
\label{X:fib_prod_joyce}
\end{ex}

\bibliographystyle{amsalpha}
\bibliography{references}

\end{document}